\documentclass[aos,preprint]{imsart}

\usepackage{amsthm,amsmath,natbib}
\RequirePackage[colorlinks,citecolor=blue,urlcolor=blue,breaklinks=true]{hyperref}
\usepackage{latexsym,amsmath,amssymb,amsfonts,amsthm,bbm,mathrsfs,stmaryrd,color,dsfont,breakcites}
\usepackage{dsfont,url}

\newtheorem{theorem}{Theorem}
\newtheorem{lemma}[theorem]{Lemma}
\newtheorem{example}[]{Example}
\newtheorem{proposition}[theorem]{Proposition}

\newenvironment{remark}[1][Remark]{\begin{trivlist}
\item[\hskip \labelsep {\bfseries #1}]}{\end{trivlist}}

\DeclareMathOperator*{\argmin}{argmin}
\DeclareMathOperator*{\argmax}{argmax}

\usepackage{xcolor}
\usepackage[draft,inline,nomargin,index]{fixme}
\fxsetup{theme=color,mode=multiuser}
\FXRegisterAuthor{tc}{atc}{\color{red} TC}
\FXRegisterAuthor{tb}{atb}{\color{red} TB}

\usepackage{xr}

\arxiv{arXiv:1704.00642}

\startlocaldefs

\endlocaldefs

\begin{document}

\begin{frontmatter}

\title{Local nearest neighbour classification with applications to semi-supervised learning}
\runtitle{Local nearest neighbour classification}


\begin{aug}
\author{\fnms{Timothy I.} \snm{Cannings},\ead[label=e1]{timothy.cannings@ed.ac.uk}\ead[label=u1,url]{http://www.maths.ed.ac.uk/\%7Etcannings/}\thanksref{m1}}

\author{\fnms{Thomas B.} \snm{Berrett}\ead[label=e2]{t.berrett@statslab.cam.ac.uk}\ead[label=u2,url]{www.statslab.cam.ac.uk/\%7Etbb26/}\thanksref{m2,t1}}
\and
\author{\fnms{Richard J.} \snm{Samworth}\corref{}\ead[label=e3]{r.samworth@statslab.cam.ac.uk}\ead[label=u3,url]{www.statslab.cam.ac.uk/\%7Erjs57/}\thanksref{m2,t2}} 

\affiliation{University of Edinburgh\thanksmark{m1} and University of Cambridge\thanksmark{m2}}

\thankstext{t1}{Research supported by an Engineering and Physical Sciences Research Council (EPSRC) programme grant.}

\thankstext{t2}{Research supported by an EPSRC Fellowship and programme grant, as well as a grant from the Leverhulme Trust.}

\address{School of Mathematics \\ James Clerk Maxwell Building\\ Peter Guthrie Tait Road\\ Edinburgh EH9 3FD\\ \printead{e1}\\ \printead{u1}}

\address{Statistical Laboratory \\ Centre for Mathematical Sciences\\ Wilberforce Road\\ Cambridge  CB3 0WB\\ \printead{e2} \\ \printead{u2}\\ \printead{e3}\\\printead{u3}}

\runauthor{T. I. Cannings, T. B. Berrett and R. J. Samworth}
\end{aug}

\begin{abstract}
  We derive a new asymptotic expansion for the global excess risk of a local-$k$-nearest neighbour classifier, where the choice of $k$ may depend upon the test point.  This expansion elucidates conditions under which the dominant contribution to the excess risk comes from the decision boundary of the optimal Bayes classifier, but we also show that if these conditions are not satisfied, then the dominant contribution may arise from the tails of the marginal distribution of the features.  Moreover, we prove that, provided the $d$-dimensional marginal distribution of the features has a finite $\rho$th moment for some $\rho > 4$ (as well as other regularity conditions), a local choice of $k$ can yield a rate of convergence of the excess risk of $O(n^{-4/(d+4)})$, where $n$ is the sample size, whereas for the standard $k$-nearest neighbour classifier, our theory would require $d \geq 5$ and $\rho > 4d/(d-4)$ finite moments to achieve this rate.  These results motivate a new $k$-nearest neighbour classifier for semi-supervised learning problems, where the unlabelled data are used to obtain an estimate of the marginal feature density, and fewer neighbours are used for classification when this density estimate is small.  Our worst-case rates are complemented by a minimax lower bound, which reveals that the local, semi-supervised $k$-nearest neighbour classifier attains the minimax optimal rate over our classes for the excess risk, up to a subpolynomial factor in $n$.  These theoretical improvements over the standard $k$-nearest neighbour classifier are also illustrated through a simulation study.  
\end{abstract}



\end{frontmatter}

\section{Introduction}

Supervised classification problems represent some of the most frequently-occurring statistical challenges in a wide variety of fields, including fraud detection, medical diagnoses and targeted advertising, to name just a few.  The area has received an enormous amount of attention within both the statistics and machine learning communities; for an excellent survey with pointers to much of the relevant literature, see \citet{BBL:05}.

The $k$-nearest neighbour classifier, which assigns the test point according to a majority vote over the classes of its $k$ nearest points in the training set, was introduced in the seminal work of \citet{Fix:51} (later republished as \citet{Fix:89}), and is arguably the simplest and most intuitive nonparametric classifier.  \citet{Cover:67} provided mild conditions under which the asymptotic risk of the $1$-nearest neighbour classifier is bounded above by twice the risk of the optimal Bayes classifier.  \citet{Stone:77} proved that if $k=k_n$ is chosen such that $k \rightarrow \infty$ and $k/n \rightarrow 0$ as $n \rightarrow \infty$, then the $k$-nearest neighbour classifier is universally consistent, in the sense that under any data generating mechanism, its risk converges to the Bayes risk.  Further recent contributions, some of which treat the $k$-nearest neighbour classifier as a special case of a plug-in classifier, include \citet{Kulkarni:95}, \citet{Audibert:07}, \citet{Hall:08}, \citet{Biau:10}, \citet{Samworth:12a}, \citet{Chaudhuri:14} and \citet{Celisse:2018}.  Nearest neighbour methods have also been extensively used in other statistical problems, including density estimation \citep{Loftsgaarden:1965,Mack:1979,Mack:1983}, nonparametric clustering, \citep{Heckel:15}, entropy and other functional estimation \citep{Kozachenko:87,Berrett:16,Berrett:19} and testing problems \citep{Schilling:86,Berrett:19b}; see also the recent book \citet{Biau:15}. 

Despite these aforementioned works, the behaviour of the $k$-nearest neighbour classifier in the tails of a distribution remains poorly understood.  Indeed, writing $(X,Y)$ for a generic data pair, where the $d$-dimensional feature vector $X$ has marginal density $\bar{f}$ and $Y$ denotes a binary class label, most of the results in the papers mentioned in the previous paragraph pertain either to situations where $\bar{f}$ is compactly supported and bounded away from zero on its support, or where the excess risk over that of the Bayes classifier is computed only over a compact subset of $\mathbb{R}^d$.  As such, many questions remain regarding the effect of tail behaviour on the excess risk. 

In this paper, we consider classes of distributions that allow the feature vectors to have unbounded support.  Our first goal is to provide a new asymptotic expansion for the global excess risk of a $k$-nearest neighbour classifier, whose error term can be bounded uniformly over our classes (Theorem~\ref{thm:exriskRdstandard}).  This expansion elucidates conditions under which the dominant contribution to the excess risk comes from the decision boundary of the Bayes classifier, but we also show that if these conditions are not satisfied, then the dominant contribution may arise from the tails of the marginal distribution of the features.  The threshold for these two different regimes is governed by a parameter $\rho$ that controls the number of finite moments of the marginal feature distribution: if $d \geq 5$ and $\rho > 4d/(d-4)$, then we obtain a rate of $O(n^{-4/(d+4)})$ uniformly over our classes, while if $d \leq 4$ or $d \geq 5$ and $\rho \leq 4d/(d-4)$ then our rate is slower, namely $O(n^{-\frac{\rho}{2\rho+d}+\epsilon})$, for every $\epsilon > 0$.

The proof of Theorem~\ref{thm:exriskRdstandard} also reveals a local bias-variance trade-off that motivates a modification of the standard $k$-nearest neighbour classifier in semi-supervised learning settings, where, as well as the labelled training data, we have access to another, independent, sample of unlabelled observations.  Such semi-supervised problems occur in a wide range of applications, especially where it is expensive or time-consuming to obtain the labels associated with observations; in fact, it is often the case that unlabelled observations may vastly outnumber labelled ones.  For an overview of semi-supervised learning applications and techniques, see \citet{Chapelle:06}.  

Our second contribution is to propose to allow the choice of $k$ in $k$-nearest neighbour classification to depend on an estimate of $\bar{f}$ at the test point $x \in \mathbb{R}^d$ in semi-supervised settings.  Such a local choice of $k$ is analagous to the use of local bandwidths in the context of kernel density estimation, as studied by, e.g., \citet{Breiman1977}, \citet{Abramson1982} and \citet{Gine2010}.  However, for density estimation, it is more common to choose a family of bandwidths $\{h(X_i):i=1,\ldots,n\}$ rather than $h = h(x)$, to ensure that the resulting estimate is itself a density.  Moreover, theory there suggests that one should then choose $h(X_i) \propto \bar{f}^{-1/2}(X_i)$ in order to cancel the leading term in the asymptotic bias expansion \citep{Abramson1982}.  By contrast, we find that when choosing $k = k(x)$, by using \emph{fewer} neighbours in low density regions, we are able to achieve a better balance in the local bias-variance trade-off for estimating our main quantity of interest, namely the regression function.  In particular, we initially study an oracle choice of $k = k(x)$ that depends on $\bar{f}(x)$, and show that the excess risk of the resulting classifier, computed over the whole of $\mathbb{R}^d$, is $O(n^{-4/(d+4)})$, again uniformly over our classes, for every $d \in \mathbb{N}$ and provided only that $\rho > 4$.  Moreover, in the more challenging case where $\rho \leq 4$, we obtain a rate of $O(n^{-\frac{\rho}{\rho+d}+\epsilon})$, for every $\epsilon > 0$, which still reflects an improvement through the locally-adaptive choice of $k$.  Assuming further that $\bar{f}$ has H\"older smoothness $\gamma \in (0,2]$, we show that if $m$ additional, unlabelled observations are used to estimate $\bar{f}$ by $\hat{f}_m$, and if $m = m_n$ satisfies $\liminf_{n \rightarrow \infty} m_n/n^{2 + d/\gamma} > 0$, then our semi-supervised $k$-nearest-neighbour classifier mimics the asymptotic performance of the oracle.

Finally, we consider corresponding minimax lower bounds.  We show in particular that the rates of convergence achieved by our semi-supervised, local-$k$-nearest neighbour classifier are optimal up to subpolynomial factors in $n$.  Interestingly, our arguments also reveal that these rates cannot be improved with the additional knowledge of $\bar{f}$.  

As mentioned previously, studies of global excess risk rates of convergence in nonparametric classification for unbounded feature vector distributions are comparatively rare.  \citet{Hall:05a} studied the tail error properties of a classifier based on kernel density estimates of the class conditional densities for univariate data.  As an illustrative example, they showed that if, for large $x$, one class has density $ax^{-\alpha}$, while the other has density $bx^{-\beta}$, for some $a,b >0$ and $1<\alpha<\beta<\alpha+1<\infty$, then the excess risk from the right tail is of larger order than that in the body of the distribution.  

Perhaps most closely related to this work, \citet{Gadat:16} recently obtained upper bounds on the supremum excess risk of the $k$-nearest neighbour classifier, over classes where $\eta$ is Lipschitz, the well-known margin assumption of \citet{Mammen:99} is satisfied with parameter $\alpha > 0$, and assuming the tail condition that $\mathbb{P}\{\bar{f}(X) < \delta\} \leq \psi(\delta)$ is satisfied for some function $\psi$ and sufficiently small $\delta > 0$.  \citet{Gadat:16} obtained a minimax lower bound over these classes, as well as providing an upper bound for the rate of the standard $k$-nearest neighbour classifier.  Since these rates do not match, they further introduced regions of the form $\bigl\{\bar{f}^{-1}\bigl((a_{j+1},a_j]\bigr):j \in \mathbb{N}\bigr\}$ with $a_{j+1} = a_j/2$, and proved that when we choose $k=k(j)$ and specialise to the case where $\psi$ is the identity function, the resulting sliced $k$-nearest neighbour classifier attains the minimax optimal rate of $n^{-(1+\alpha)/(2+\alpha+d)}$ up to a polylogarithmic factor in $n$.  Neither our smoothness and tail assumptions, nor our conclusions are directly comparable with the work of \citet{Gadat:16}.  In particular, we make a stronger smoothness assumption on $\eta$ in a neighbourhood of the Bayes decision boundary, implying that the margin assumption holds with parameter $\alpha = 1$; see Lemma~\ref{Lemma:Margin} in Appendix \ref{Sec:Margin}.  This enables us to show that our semi-supervised classifier attains faster rates than are achievable under just a Lipschitz condition, and that these rates are minimax optimal up to subpolynomial factors in $n$, over all possible values of our tail parameter $\rho$; moreover, we are also able to provide the leading constants in the asymptotic expansion of the excess risk in some cases.

The remainder of this paper is organised as follows.  After introducing our setting in Section~\ref{Sec:Prelim}, we present in Section~\ref{sec:mainstandard} our main results for the standard $k$-nearest neighbour classifier.  This leads on, in Section~\ref{sec:semisup}, to our study of the semi-supervised setting, where we derive asymptotic results of the excess risk of our local-$k$-nearest neighbour classifier.  Our minimax lower bound in presented in Section~\ref{Sec:LowerBounds}.  The main arguments of the proofs of our theoretical results are given in Section~\ref{sec:appendix}, while in the appendices, we prove several claims made in the main text, bound various remainder terms, illustrate the finite-sample benefits of the semi-supervised classifier over the standard $k$-nearest neighbour classifier in a simulation study and provide an introduction to the ideas of differential geometry that underpin much of our analysis. 

Finally we fix here some notation used throughout the paper.  Let $\|\cdot\|$ denote the Euclidean norm and, for $r>0$ and $x\in \mathbb{R}^d$, let $B_{r}(x) := \{z \in \mathbb{R}^d : \|x-z\|< r\}$ and $\bar{B}_r(x) := \{z \in \mathbb{R}^d : \|x-z\| \leq r\}$ denote respectively the open and closed Euclidean balls of radius $r$ centred at $x$. Let $a_d := \frac{2\pi^{d/2}}{d\Gamma(d/2)}$ denote the $d$-dimensional Lebesgue measure of $B_1(0)$.  For a real-valued function $g$ defined on $A \subseteq \mathbb{R}^d$ that is twice differentiable at $x$, write $\dot{g}(x) = (g_1(x),\ldots,g_d(x))^T$ and $\ddot{g}(x) = \bigl(g_{jk}(x)\bigr)$ for its gradient vector and Hessian matrix at $x$, and let $\|g\|_{\infty} = \sup_{x\in A} |g(x)|$.  We write $\|\cdot\|_{\mathrm{op}}$ for the operator norm of a matrix. 

\section{Statistical setting}
\label{Sec:Prelim}
Let  $(X, Y), (X_1,Y_1), \dots, (X_{n+m}, Y_{n+m})$ be independent and identically distributed random pairs taking values in $\mathbb{R}^d \times \{0,1\}$.  Let $\pi_r := \mathbb{P}(Y = r)$, for $r =0,1$, and $X | Y = r \sim P_r$, for $r=0,1$, where $P_r$ is a probability measure on $\mathbb{R}^d$. Let $\eta(x) := \mathbb{P}(Y=1 | X=x)$ denote the regression function and $P_X := \pi_{0} P_0 + \pi_{1} P_1$ denote the marginal distribution of $X$.   We observe \textit{labelled training data}, $\mathcal{T}_n := \{(X_1,Y_1), \dots, (X_n, Y_n)\}$, and \textit{unlabelled training data}, $\mathcal{T}_m' := \{X_{n+1}, \dots, X_{n+m}\}$, and are presented with the task of assigning the \textit{test point} $X$ to either class 0 or 1.

A \textit{classifier} is a Borel measurable function $C : \mathbb{R}^d \to \{0,1\}$, with the interpretation that $C$ assigns $x \in \mathbb{R}^d$ to the class $C(x)$.  Given a Borel measurable set  $\mathcal{R} \subseteq \mathbb{R}^d$, the misclassification rate, or \textit{risk}, over $\mathcal{R}$ is
\[
R_{\mathcal{R}}(C): = \mathbb{P}[\{C(X) \neq Y\}\cap\{X\in\mathcal{R}\}].
\]
When $\mathcal{R} = \mathbb{R}^d$, we drop the subscript for convenience.  The \textit{Bayes classifier}
\[
 C^{\mathrm{Bayes}}(x) := \left\{ \begin{array}{ll}
         1 & \mbox{if $\eta(x) \geq 1/2$};\\
         0 & \mbox{otherwise},\end{array} \right.
\]
minimises the risk over any region $\mathcal{R}$ \citep[p.~20]{PTPR:96}.  The performance of a classifier $C$ is therefore measured via its \textit{excess risk}, $R_{\mathcal{R}}(C)-R_{\mathcal{R}}(C^{\mathrm{Bayes}})$.  

We can now formally define the local-$k$-nearest neighbour classifier, which allows the number of neighbours considered to vary depending on the location of the test point.  Suppose $k_{\mathrm{L}}: \mathbb{R}^d \to \{1, \dots,n\}$ is measurable. Given the test point $x \in \mathbb{R}^d$, let $(X_{(1)}, Y_{(1)}), \ldots, (X_{(n)}, Y_{(n)})$ be a reordering of the training data such that $\|X_{(1)} - x\| \leq \dots \leq \|X_{(n)}-x\|$.  We will later assume that $P_X$ is absolutely continuous with respect to $d$-dimensional Lebesgue measure, which ensures that ties occur with probability zero; where helpful for clarity, we also write $X_{(i)}(x)$ for the $i$th nearest neighbour of $x$.  Let $\hat{S}_{n}(x) := k_{\mathrm{L}}(x)^{-1} \sum_{i=1}^{k_{\mathrm{L}}(x)} \mathbbm{1}_{\{Y_{(i)}=1\}}$.  Then the \textit{local-$k$-nearest neighbour ($k_{\mathrm{L}}$nn) classifier} is defined to be 
\[
\hat{C}_n^{k_{\mathrm{L}}\mathrm{nn}}(x)  := \left\{ \begin{array}{ll}
         1 & \mbox{if $\hat{S}_{n}(x) \geq 1/2$};\\
         0 & \mbox{otherwise}.\end{array} \right.
\]

Given $k \in \{1, \ldots, n\}$, let $k_0$ denote the constant function $k_{0}(x) := k$ for all $x \in \mathbb{R}^{d}$.  Using $k_{\mathrm{L}} = k_0$ the definition above reduces to the standard \textit{$k$-nearest neighbour classifier} ($k$nn), and we will write $\hat{C}_n^{k\mathrm{nn}}$ in place of $\hat{C}_n^{k_0\mathrm{nn}}$.  For $\beta \in (0,1/2)$, let 
\[
K_{\beta} \equiv K_{\beta,n} : = \bigl\{\lceil (n-1)^\beta \rceil, \lceil (n-1)^\beta\rceil+ 1, \ldots, \lfloor (n-1)^{1-\beta} \rfloor \bigr\}
\]
denote a range of values of $k$ that will be of interest to us. Note that  $K_{\beta_{1}} \supseteq K_{\beta_{2}}$, for $\beta_{1} < \beta_{2}$.   Moreover, when $\beta$ is small, the restriction that $k \in K_\beta$ is only a slightly stronger requirement than the consistency conditions of \citet{Stone:77}, namely that $k = k_{n} \to \infty$, $k_{n}/n \to 0$ as $n \to \infty$. 

\section{Global risk of the $k$-nearest neighbour classifier}
\label{sec:mainstandard}
In this section we provide an asymptotic expansion for the global risk of the standard (non-local) $k$-nearest neighbour classifier.  We first define the classes of data generating mechanisms over which our results will hold.  Let $\mathcal{L}$ denote the class of decreasing functions $\ell:(0,\infty) \rightarrow [1,\infty)$ such that $\ell(\delta) = o(\delta^{-\tau})$ as $\delta \searrow 0$, for every $\tau > 0$.  Let $\mathcal{G}$ denote the class of strictly increasing functions $g:(0,1) \rightarrow (0,1)$ with $g(\epsilon) = o(\epsilon^M)$ as $\epsilon \searrow 0$, for every $M > 0$.  Recall from Section~\ref{Sec:Prelim} that, to any distribution $P$ on $\mathbb{R}^d \times \{0,1\}$, we associate conditional distributions $P_{0}, P_{1}$, a regression function $\eta$, marginal probabilities $\pi_{0}, \pi_{1}$ and a marginal distribution $P_X$.  Now, for $\Theta := (0,\infty) \times [1,\infty) \times (0,\infty) \times \mathcal{L} \times \mathcal{G}$, and $\theta = (\epsilon_0,M_0,\rho,\ell,g) \in \Theta$, let $\mathcal{P}_{d,\theta}$ denote the class of distributions $P$ on $\mathbb{R}^d \times \{0,1\}$ such that the probability measures $P_0$ and $P_1$ are absolutely continuous with respect to Lebesgue measure, with Radon--Nikodym derivatives $f_0$ and $f_1$, respectively.  Moreover, we assume that there exist versions of $f_0$ and $f_1$ for which the following conditions hold:
\begin{description}
\item[\textbf{(A.1)}]
  \ \ The marginal density of $X$, namely $\bar{f} := \pi_{0} f_{0} + \pi_{1} f_{1}$, is continuous $P_X$-almost everywhere and the set $\mathcal{X}_{\bar{f}}$ of continuity points of $\bar{f}$ is open. 
\end{description}
Thus $\eta(x) := \pi_1f_1(x)/\{\pi_0f_0(x) + \pi_1f_1(x)\}$, where we define $0/0 := 0$.  Let $\mathcal{S} := \{x \in \mathbb{R}^d : \eta(x)= 1/2\}$ and, for $\epsilon > 0$, let $\mathcal{S}^\epsilon := \mathcal{S} + B_\epsilon(0)$.  In our assumptions below, we will place further assumptions on $\mathcal{S}$, which ensure not only that this set is non-empty, but in fact that it is a $(d-1)$-dimensional, orientable manifold.
\begin{description}
\item[\textbf{(A.2)}]
\ \ The set $\mathcal{S} \cap \{x \in \mathbb{R}^d : \bar{f}(x) > 0\}$ is non-empty and $\sup_{x_0 \in \mathcal{S}}\bar{f}(x_0) \leq M_0$.  The function $\bar{f}$ is twice continuously differentiable on $\mathcal{S}^{\epsilon_0}$, and
\begin{equation}
\label{Eq:A2eq}
\max\biggl\{ \|\dot{\bar{f}}(x_0)\|, \sup_{u \in B_{\epsilon_0}(0)}\|\ddot{\bar{f}}(x_0+u)\|_{\mathrm{op}}\biggr\} \leq \bar{f}(x_0) \ell\bigl(\bar{f}(x_0)\bigr),
\end{equation}
for all $x_0 \in \mathcal{S}$.  Furthermore, writing $p_r(x) := P_X\bigl(B_r(x)\bigr)$, we have for all $x \in \mathbb{R}^{d} \setminus \mathcal{S}^{\epsilon_0}$ and $r \in (0,\epsilon_{0}]$ that
\[
p_r(x) \geq \epsilon_0 a_d r^{d} \bar{f}(x).
\]
\item[\textbf{(A.3)}] 
  \ \ We have that $\eta$ is twice differentiable on $\mathcal{S}^{2\epsilon_0}$ with $\inf_{x_0\in \mathcal{S}}\|\dot{\eta}(x_0)\| \geq \epsilon_0M_0$.  Moreover, $\sup_{x \in \mathcal{S}^{2\epsilon_0}} \|\dot{\eta}(x)\| \leq M_0$, $\sup_{x\in \mathcal{S}^{2\epsilon_0}} \|\ddot{\eta} (x)\|_{\mathrm{op}} \leq M_0$ and given $\epsilon > 0$,
  \[
    \sup_{x,z \in \mathcal{S}^{2\epsilon_0}:\|z-x\| \leq g(\epsilon)}  \|\ddot{\eta}(z) - \ddot{\eta}(x)\|_{\mathrm{op}} \leq \epsilon.
  \]
Finally, the function $\eta$ is continuous on $\{x:\bar{f}(x) > 0\}$, and
\[
  |\eta(x) - 1/2| \geq \frac{1}{\ell\bigl(\bar{f}(x)\bigr)}
\]
for all $x \in \mathbb{R}^d \setminus \mathcal{S}^{\epsilon_0}$.

\item[\textbf{(A.4)}] 
\  We have $\int_{\mathbb{R}^d}  \|x\|^{\rho} \, dP_{X}(x) \leq M_0$. 
\end{description}

\begin{example}
\label{eg:classexample}
Consider the distribution $P$ on $\mathbb{R}^d \times \{0,1\}$ for which $\bar{f}(x)=\frac{\Gamma(3+d/2)}{2 \pi^{d/2}}(1-\|x\|^2)^2 \mathbbm{1}_{\{x \in B_1(0)\}}$ and $\eta(x) = \min(\|x\|^2,1)$.  In Appendix~\ref{Sec:Example}, we show that $P \in \mathcal{P}_{d,\theta}$ with $\theta = (\epsilon_0,M_0,\rho,\ell,g) \in \Theta$ for any $\rho>0$, $g \in \mathcal{G}$, and provided that $M_0 \geq \max\bigl\{2  , \frac{\Gamma(3+d/2)}{8 \pi^{d/2}} \bigr\}$, $\epsilon_0 \leq \min\bigl(\frac{1}{10},2^{-d},\frac{2^{1/2}}{M_0}\bigr)$ and $\ell \in \mathcal{L}$ satisfies $\ell(\delta) \geq \max(48,\epsilon_0^{-1})$ for all $\delta > 0$.
\end{example}

Asking for $P_X$ to have a Lebesgue density allows us to define the tail of the distribution as the region where $\bar{f}$ is smaller than some threshold.  Condition~\textbf{(A.1)} ensures that for all $\delta > 0$ sufficiently small, the set $\mathcal{R} := \{x:\bar{f}(x) > \delta\} \cap \mathcal{X}_{\bar{f}}$ is a $d$-dimensional manifold, and $P_X(\mathcal{R}^c) \leq \mathbb{P}\bigl\{\bar{f}(X) \leq \delta\bigr\}$, where the latter quantity can be bounded using~\textbf{(A.4)}.  The first part of~\textbf{(A.2)} asks for a certain level of smoothness for $\bar{f}$ in a neighbourhood of $\mathcal{S}$, and controls the behaviour of its first and second derivatives there relative to the original density.  In particular, the greater degree of regularity asked of these derivatives in the tails of the marginal density in~\eqref{Eq:A2eq} allows us still to control the error of a Taylor approximation even in this region.  The condition~\eqref{Eq:A2eq} is satisfied by all Gaussian and multivariate-$t$ densities, for example, for appropriate choices of $\epsilon_0$ and~$\ell$.  The last part of~\textbf{(A.2)} concerns the behaviour of the marginal feature distribution away from $\mathcal{S}^{\epsilon_0}$ and is often referred to as the strong minimal mass assumption \citep[e.g.][]{Gadat:16}.  It requires that the mass of the marginal feature distribution is not concentrated in the neighbourhood of a point and is a rather weaker condition than we ask for on $\mathcal{S}^{\epsilon_0}$; in particular, we do not insist that derivatives of $\bar{f}$ exist in this region.  

The condition $\inf_{x_0\in \mathcal{S}}\|\dot{\eta}(x_0)\| \geq \epsilon_0M_0$ in~\textbf{(A.3)} asks for the class conditional densities, when weighted by their respective prior probabilities, to cross at an angle; in particular, this ensures that $\mathcal{S}$ is a $(d-1)$-dimensional, orientable manifold (cf. Section~\ref{Sec:Forms}).  Moreover, the bounds on the first and second derivatives of $\eta$ in a neighbourhood of $\mathcal{S}$ ensure that we can estimate $\eta$ sufficiently well.  The last part of~\textbf{(A.3)} asks that~$\eta$ does not approach the critical value of $1/2$ too fast on the complement of $\mathcal{S}^{\epsilon_0}$.  Assumption~\textbf{(A.4)} is a simple moment condition that, together with~\textbf{(A.2)}, ensures that the constants $B_1$ and $B_2$ in~\eqref{eq:B1B2defRd} below are finite where needed.

Let $d\mathrm{Vol}^{d-1}$ denote the $(d-1)$-dimensional volume form on $\mathcal{S}$ (cf.~Section~\ref{Sec:Forms}).  Now let
\begin{equation}
\label{eq:B1B2defRd}
B_{1} := \int_{\mathcal{S}} \frac{\bar{f}(x_0)}{4\|\dot\eta(x_0)\|} \, d\mathrm{Vol}^{d-1}(x_0) \ \mathrm{and} \ B_{2} := \int_{\mathcal{S}} \frac{\bar{f}(x_0)^{1-4/d}}{\|\dot\eta(x_0)\|}a(x_0)^2 \, d\mathrm{Vol}^{d-1}(x_0),
\end{equation}
where
\begin{equation}
\label{Eq:a}
a(x) := \frac{\sum_{j=1}^d \bigl\{\eta_j(x)\bar{f}_j(x) + \frac{1}{2}\eta_{jj}(x)\bar{f}(x)\bigr\}}{(d+2)a_d^{2/d}\bar{f}(x)}.
\end{equation}
We are now in a position to present our asymptotic expansion for the global excess risk of the standard $k$-nearest neighbour classifier.
\begin{theorem}
\label{thm:exriskRdstandard}
Fix $d \in \mathbb{N}$ and $\theta = (\epsilon_0,M_0,\rho,\ell,g) \in \Theta$ such that $\mathcal{P}_{d,\theta} \neq \emptyset$.  

(i) Suppose that $d \geq 5$ and $\rho > \frac{4d}{d-4}$.  Then for each $\beta \in (0,1/2)$,
\[
  \sup_{P \in \mathcal{P}_{d,\theta}} \biggl|R(\hat{C}_n^{k\mathrm{nn}}) - R(C^{\mathrm{Bayes}}) - \frac{B_1}{k} - B_{2} \Bigl(\frac{k}{n}\Bigr)^{4/d}\biggr| = o\biggl(\frac{1}{k} + \Bigl(\frac{k}{n}\Bigr)^{4/d}\biggr)
  \]
as $n \to \infty$, uniformly for $k \in K_{\beta}$.

 (ii) Suppose that either $d \leq 4$, or, $d \geq 5$ and $\rho \leq \frac{4d}{d-4}$.  Then for each $\beta \in (0,1/2)$ and each $\epsilon> 0$ we have
\[
\sup_{P \in \mathcal{P}_{d,\theta}} \biggl|R(\hat{C}_n^{k\mathrm{nn}}) - R(C^{\mathrm{Bayes}}) - \frac{B_{1}}{k}\biggr| = o\Bigl(\frac{1}{k} +\Bigl(\frac{k}{n}\Bigr)^{\frac{\rho}{\rho+d} - \epsilon}\Bigr)
\]
as $n \to \infty$, uniformly for $k \in K_{\beta}$.
\end{theorem}

Theorem~\ref{thm:exriskRdstandard} reveals an interesting dichotomy: when $d \geq 5$ and $\rho > 4d/(d-4)$, the dominant contribution to the excess risk arises from the difficulty of classifying points close to the Bayes decision boundary $\mathcal{S}$.  In such settings, the excess risk of the standard $k$-nearest neighbour classifier converges to zero at rate $O(n^{-4/(d+4)})$ when $k$ is chosen proportional to $n^{4/(d+4)}$.  On the other hand, part~(ii) shows that when either $d \leq 4$ or $d \geq 5$ and $\rho \leq 4d/(d-4)$, the dominant contribution to the excess risk when $k$ is large may come from the challenge of classifying points in the tails of the distribution.  Indeed, Example~\ref{Eg:Counter} below provides one simple setting where this dominant contribution does come from the tails of the distribution.
\begin{example}
\label{Eg:Counter}
Suppose that the joint density of $X$ at $x = (x_1,x_2) \in (0,1) \times \mathbb{R}$ is given by $\bar{f}(x) =  2 x_1 f_2(x_2)$, where $f_2$ is a positive, twice continuously differentiable density with $f_2(x_2) = e^{-|x_2|}/2$ for $|x_2| > 1$.  Suppose also that $\eta(x) = x_1$.  Then the corresponding joint distribution $P$ belongs to $\mathcal{P}_{2,\theta}$ provided $\theta = (\epsilon_0,M_0,\rho,\ell,g)$ is such that $M_0$ is sufficiently large, $\epsilon_0 \leq \min(1/8,1/M_0)$ and $\ell$ is a sufficiently large constant ($\rho > 0$ and $g \in \mathcal{G}$ can be chosen arbitrarily).  We prove in Appendix~\ref{Sec:Lower} that for every $\beta \in (0,1/2)$ and $\epsilon > 0$,
\begin{equation}
\label{Eq:Claim}
\liminf_{n \to \infty} \inf_{k \in K_{\beta}} \biggl\{k + \Bigl(\frac{n}{k}\Bigr)^{1+\epsilon}\biggr\} \bigl\{R(\hat{C}_n^{k\mathrm{nn}}) - R(C^{\mathrm{Bayes}})\bigr\} > 0
\end{equation}
as $n \to \infty$.  Thus the rate of convergence in this example is at best $n^{-1/2}$, up to subpolynomial factors, whereas a rate of $n^{-2/3}$ is achievable over any compact set.
\end{example}

The proof of~Theorem~\ref{thm:exriskRdstandard}, and indeed the proofs of Theorems~\ref{thm:exriskRdOadapt} and~\ref{thm:exriskRdSSadapt} that follow in Section~\ref{sec:semisup} below, depend crucially on Theorem~\ref{thm:exriskRtn} in Section~\ref{sec:appendix}.  This result provides an asymptotic expansion for the excess risk of a general (local or global) $k$-nearest neighbour classifier over a region $\mathcal{R}_n \subseteq \{x \in \mathbb{R}^d:\bar{f}(x) \geq \delta_n(x)\}$, where $\delta_n(x)$, defined in~\eqref{eq:deltan} below, shrinks to zero at a rate slow enough to ensure that $X_{(k)}(x)$ concentrates around~$x$ uniformly over $\mathcal{R}_n$.  The intuition regarding the behaviour of the excess risk, then, is that when $x \in \mathcal{R}_n$ and $x$ is not close to $\mathcal{S}$, with high probability the $k$ nearest neighbours of $x$ are on the same side of $\mathcal{S}$ as $x$; i.e.~$\mathrm{sgn}\bigl(\eta(X_{(i)}) - 1/2\bigr) = \mathrm{sgn}\bigl(\eta(x) - 1/2\bigr)$ for $i = 1,\ldots,k$.  The probability of classifying $x$ differently from the Bayes classifier can therefore be shown to be $O(n^{-M})$ for every $M > 0$, using Hoeffding's inequality.  Thus, the challenging regions for classification consist of neighbourhoods of $\mathcal{S}$, where $\eta$ is close to $1/2$, together with $\mathcal{R}_n^c$, where we no longer enjoy the same nearest neighbour concentration properties.  For the first of these regions, we exploit our smoothness assumptions to derive asymptotic expansions for the bias and variance of $\hat{S}_n(x)$, uniformly over appropriate neighbourhoods of $\mathcal{S}$, and using a normal approximation, we can deduce an asymptotic expansion for the excess risk, uniformly over our classes of distributions and an appropriate set of nearest neighbour classifiers.  For $\mathcal{R}_n^c$ we are unable to bound the probability of classifying differently from the Bayes classifier with anything other than a trivial bound, but we can control $P_X(\mathcal{R}_n^c)$ using~\textbf{(A.4)}.

Finally in this section, we mention that \citet{Samworth:12a} obtained a similar expansion to that in Theorem~\ref{thm:exriskRdstandard}(i) for a fixed distribution $P$ satisfying certain smoothness conditions.  However, there the risk was computed only over a compact set, so the analysis failed to elucidate the important effects of tail behaviour on the excess risk.  Another key difference is that here we define classes $\mathcal{P}_{d,\theta}$, and show that the remainder terms in our asymptotic expansion hold \emph{uniformly} over these classes; the introduction of these classes further facilitates the study of corresponding minimax lower bounds in Section~\ref{Sec:LowerBounds} below.

\section{Local-$k$-nearest neighbour classifiers}
\label{sec:semisup}

In this section we explore the consequences of a local choice of $k$, compared with the global choice in Theorem~\ref{thm:exriskRdstandard}.  Initially, we consider an oracle choice, where $k$ is allowed to depend on the marginal feature density $\bar{f}$ (Section~\ref{Sec:Oracle}), but we then relax this to semi-supervised settings, where $\bar{f}$ can be estimated from unlabelled training data (Section~\ref{Sec:SS}).

\subsection{Oracle classifier}
\label{Sec:Oracle}
Suppose for now that the marginal density $\bar{f}$ is known.  For $\beta \in (0,1/2)$ and $B > 0$, let 
\begin{equation}
\label{eq:koracle}
k_{\mathrm{O}}(x) : = \max\Bigl[ \lceil (n-1)^\beta \rceil \, , \, \min\bigl\{\bigl\lfloor B\bigl\{\bar{f}(x)(n-1)\bigr\}^{4/(d+4)} \bigr\rfloor \, , \, \lfloor (n-1)^{1-\beta} \rfloor \bigr\}\Bigr],
\end{equation}
where the subscript O refers to the fact that this is an oracle choice of the function $k_{\mathrm{L}}$, since it depends on $\bar{f}$.  This choice aims to balance the local bias and variance of $\hat{S}_{n}(x)$.  
\begin{theorem}
\label{thm:exriskRdOadapt}
Fix $d \in \mathbb{N}$ and $\theta = (\epsilon_0,M_0,\rho,\ell,g) \in \Theta$ such that $\mathcal{P}_{d,\theta} \neq \emptyset$.  For each $0 < B_* \leq B^* < \infty$,

(i) if $\rho > 4$ then for $\beta<4d(\rho-4)/\{\rho(d+4)^2\}$,
\[
\sup_{P \in \mathcal{P}_{d,\theta}}\Bigl|R(\hat{C}_n^{k_{\mathrm{O}}\mathrm{nn}}) - R(C^{\mathrm{Bayes}}) - B_3 n^{-4/(d+4)}\Bigr| = o(n^{-4/(d+4)}),
\]
uniformly for $B \in [B_*,B^*]$ as $n \to \infty$, where 
\[
B_3 := \int_{\mathcal{S}} \frac{\bar{f}(x_0)^{d/(d+4)}}{\|\dot{\eta}(x_0)\|} \Bigl\{ \frac{1}{4B} +B^{4/d} a(x_0)^2 \Bigr\}\, d\mathrm{Vol}^{d-1}(x_0) < \infty.
\]

(ii) if $\rho \leq 4$ and $\beta < \min\{1/2,4/(d+4)\}$, then for every $\epsilon > 0$
\[
\sup_{P \in \mathcal{P}_{d,\theta}}\bigl\{R(\hat{C}_n^{k_{\mathrm{O}}\mathrm{nn}}) - R(C^{\mathrm{Bayes}})\bigr\} = o(n^{-\rho/(\rho+d) + \beta +\epsilon}),
\]
uniformly for $B \in [B_*,B^*]$, as $n \to \infty$.
\end{theorem}
Comparing Theorem~\ref{thm:exriskRdOadapt}(i) and Theorem~\ref{thm:exriskRdstandard}(i), we see that, unlike for the global $k$-nearest neighbour classifier, we can guarantee a $O(n^{-4/(d+4)})$ rate of convergence for the excess risk of the oracle classifier, both in low dimensions ($d \leq 4$), and under a weaker condition on $\rho$ when $d \geq 5$.  In particular, the condition on $\rho$ no longer depends on the dimension of the covariates.  The guarantees in Theorem~\ref{thm:exriskRdOadapt}(ii) are also stronger than those provided by Theorem~\ref{thm:exriskRdstandard}(ii) for any global choice of $k$. Examining the proof of Theorem~\ref{thm:exriskRdOadapt}, we find that the key difference with the proof of Theorem~\ref{thm:exriskRdstandard} is that we can now choose the region $\mathcal{R}_n$ (cf.~the discussion of the proof of Theorem~\ref{thm:exriskRdstandard} in Section~\ref{sec:mainstandard}) to be larger.

\subsection{The semi-supervised nearest neighbour classifier}
\label{Sec:SS}
Now consider the more realistic setting where the marginal density $\bar{f}$ of $X$ is unknown, but where we have access to an estimate $\hat{f}_m$ based on the unlabelled training set $\mathcal{T}'_{m}$.  Of course, many different techniques are available, but for simplicity, we focus here on a kernel method.  Let $K$ be a bounded kernel with $\int_{\mathbb{R}^d} K(x) \, dx = 1$, $\int_{\mathbb{R}^d} xK(x) \, dx = 0$, $\int_{\mathbb{R}^d} \|x\|^2 |K(x)| \, dx < \infty$, and let $R(K) := \int_{\mathbb{R}^d} K(x)^2 \, dx$.  We further assume that $K(x) = Q(p(x))$, where $p$ is a polynomial and $Q$ is a function of bounded variation.  Now define a kernel density estimator of $\bar{f}$, given by 
\[ 
\hat{f}_{m}(x) = \hat{f}_{m,h}(x) := \frac{1}{mh^{d}}  \sum_{j=1}^{m} K\Bigl(\frac{x - X_{n+j}}{h}\Bigr).
\]
Motivated by the oracle local choice of $k$ in~\eqref{eq:koracle}, for $\beta \in (0,1/2)$ and $B > 0$, let 
\[
k_{\mathrm{SS}}(x) : = \max\Bigl[\lceil (n-1)^\beta\rceil \, , \, \min\bigl\{\lfloor B\{\hat{f}_m(x)(n-1)\}^{4/(d+4)} \rfloor \, , \, \lfloor (n-1)^{1-\beta}\rfloor\bigr\}\Bigr].
\]
Our main result in this setting will require an additional smoothness condition on the marginal feature density $\bar{f}$ in order to ensure that $\hat{f}_m$ estimates it well.  For $d \in \mathbb{N}$, $\gamma \in (0,1]$ and $\lambda > 0$, let $\mathcal{Q}_{d,\gamma,\lambda}$ denote the class of distributions $P$ on $\mathbb{R}^d \times \{0,1\}$ whose marginal distribution $P_X$ is absolutely continuous with respect to Lebesgue measure with Radon--Nikodym derivative $\bar{f}$ satisfying $\|\bar{f}\|_\infty \leq \lambda$ and
\[
  \| \bar{f}(y)-\bar{f}(x)\| \leq \lambda \|y-x\|^{\gamma}  \quad \text{for all $x,y \in \mathbb{R}^d$.}
\]
If $\gamma \in (1,2]$, then we define $\mathcal{Q}_{d,\gamma,\lambda}$ to consist of distributions $P$ on $\mathbb{R}^d \times \{0,1\}$ whose marginal distribution $P_X$ is again absolutely continuous with Radon--Nikodym derivative $\bar{f}$ satisfying $\|\bar{f}\|_\infty \leq \lambda$, but we now ask that $\bar{f}$ be differentiable, and that 
\[
\| \dot{\bar{f}}(y)-\dot{\bar{f}}(x)\| \leq \lambda \|y-x\|^{\gamma-1}  \quad \text{for all $x,y \in \mathbb{R}^d$}.
\]
In Appendix~\ref{Sec:Example}, we show that the distribution considered in Example~\ref{eg:classexample} belongs to $\mathcal{Q}_{d,\gamma,\lambda}$ with $\gamma=2$ provided that $\lambda \geq 6 \pi^{-d/2} \Gamma(3+d/2)$.
\begin{theorem}
\label{thm:exriskRdSSadapt}
Fix $d \in \mathbb{N}$, $\theta = (\epsilon_0,M_0,\rho,\ell,g) \in \Theta$, $\gamma \in (0,2]$ and $\lambda > 0$ such that $\mathcal{P}_{d,\theta} \cap \mathcal{Q}_{d,\gamma,\lambda} \neq \emptyset$.  Let $m_0 > 0$, let $0 < A_* \leq A^* < \infty$ and $0 < B_* \leq B^* < \infty$, and let $h = h_{m} := A m^{-1/(d+2 \gamma)}$ for some $A > 0$.

(i) If $\rho > 4$ and $\beta<4d(\rho-4)/\{\rho(d+4)^2\}$, 
\[
\sup_{P \in \mathcal{P}_{d,\theta} \cap \mathcal{Q}_{d,\gamma,\lambda}}\Bigl|R(\hat{C}_n^{k_{\mathrm{SS}}\mathrm{nn}}) - R(C^{\mathrm{Bayes}}) - B_3 n^{-4/(d+4)}\Bigr| = o(n^{-4/(d+4)})
\]
uniformly for $A \in [A_*,A^*]$, $B \in [B_*,B^*]$ and $m = m_n \geq m_0 (n-1)^{2 + d/\gamma}$, where $B_3$ was defined in Theorem~\ref{thm:exriskRdOadapt}(i).

(ii) if $\rho \leq 4$ and $\beta < \min\{1/2,4/(d+4)\}$, then for every $\epsilon > 0$,
\[
\sup_{P \in \mathcal{P}_{d,\theta} \cap \mathcal{Q}_{d,\gamma,\lambda}}\bigl\{R(\hat{C}_n^{k_{\mathrm{SS}}\mathrm{nn}}) - R(C^{\mathrm{Bayes}})\bigr\} = o(n^{-\rho/(\rho+d) + \beta +\epsilon}), 
\]
uniformly for $A \in [A_*,A^*]$, $B \in [B_*,B^*]$ and $m = m_n \geq m_0 (n-1)^{2+d/\gamma}$.
\end{theorem}
Examination of the proof of Theorem~\ref{thm:exriskRdSSadapt} reveals that the key property of our kernel estimator $\hat{f}_m$ of $\bar{f}$ is that there exists $\alpha > (1+d/4)\beta$ such that 
\begin{equation}
\label{Eq:tildefm}
\sup_{P \in \mathcal{P}_{d,\theta} \cap \mathcal{Q}_{d,\gamma,\lambda}}\mathbb{P}\biggl( \|\hat{f}_m - \bar{f} \|_{\infty}  \geq \frac{1}{(n-1)^{1-\alpha/2}} \biggr) = o(n^{-4/(d+4)}).
\end{equation}
This observation would allow similar results to Theorem~\ref{thm:exriskRdSSadapt} to be proved for other versions of the semi-supervised nearest neighbour classifier, with alternative estimators of $\bar{f}$ in the definition of $\hat{k}_{\mathrm{SS}}(\cdot)$, subject potentially to suitable modifications of the class $\mathcal{Q}_{d,\gamma,\lambda}$.  It is therefore not our intention to argue that the kernel density approach is superior to other methods of estimating the marginal density $\bar{f}$.

\section{Minimax lower bounds}
\label{Sec:LowerBounds}

Our main minimax lower bound is the following:
\begin{theorem}
  \label{Thm:LowerBound}
  Fix $d \in \mathbb{N}$, $\rho > 0$, $g \in \mathcal{G}$ with $r \mapsto r/g^{-1}(r)$ increasing for sufficiently small $r > 0$, and $\gamma \in (0,2]$.  There exist $\lambda_* > 0$, $\epsilon_* > 0$ and $M_* > 0$, depending only on $d$, such that for $\lambda \geq \lambda_*$, $M_0 \geq M_*$, $\epsilon_0 \in (0,\min(\epsilon_*,1/(4M_0))]$ and $\ell \in \mathcal{L}$ with $\ell(\delta) \geq 2/\epsilon_0$ for all $\delta \in (0,\infty)$, writing $\theta = (\epsilon_{0}, M_0, \rho, \ell, g) \in \Theta$, we can find $c = c(d, \theta,\gamma,\lambda) >0$ such that for all $n \in \mathbb{N}$ and all $\nu \geq 0$, we have
\[
\inf_{C_n} \sup_{P \in \mathcal{P}_{d,\theta} \cap \mathcal{Q}_{d,\gamma,\lambda}} \{R(C_n) - R(C^{\mathrm{Bayes}})\} \geq c \, g^{-1}(1/q)^{\frac{2d(1+\nu)}{4+d+\nu(\rho + d)}}n^{-\frac{4 + \nu \rho}{4 + d + \nu(\rho +d)}},
\]
where $q = q_n \in (1/\|g\|_\infty,\infty)$ is the unique solution to $\frac{q^{4+d+\nu(\rho + d)}}{g^{-1}(1/q)^{2}} = n$ and the infimum is taken over all measurable functions $C_n: (\mathbb{R}^d \times \{0,1\})^{\times n} \times \mathbb{R}^{d} \rightarrow \{0,1\}$.  In particular, for every $\epsilon > 0$, there exists $c = c(d, \theta,\gamma,\lambda,\epsilon) >0$ such that
\[
  \inf_{C_n} \sup_{P \in \mathcal{P}_{d,\theta} \cap \mathcal{Q}_{d,\gamma,\lambda}} \{R(C_n) - R(C^{\mathrm{Bayes}})\} \geq c \, n^{-(\min\{\frac{4}{4 + d},\frac{\rho}{\rho + d}\}+\epsilon)}.
  \]

\end{theorem}
\begin{remark}
The proof of this result also reveals that the lower bound holds if the classifier is allowed to depend on some unlabelled data or even the true marginal $X$ density $\bar{f}$. 
\end{remark}
\begin{example} Consider the case where $g(\epsilon) = \exp(-1/\epsilon)$, so $g \in \mathcal{G}$.  Then for $q \in (1,\infty)$, we have $g^{-1}(1/q) = 1/\log q$, so for $n \in \mathbb{N}$,
\begin{align*}
  g^{-1}(1/q_n)^{\frac{2d(1+\nu)}{4+d+\nu(\rho + d)}}n^{-\frac{4 + \nu \rho}{4 + d + \nu(\rho +d)}} &\geq \frac{1}{\bigl\{1 + \frac{\log n}{4+d+\nu(\rho+d)}\bigr\}^{\frac{2d(1+\nu)}{4+d+\nu(\rho + d)}}}n^{-\frac{4 + \nu \rho}{4 + d + \nu(\rho +d)}}.
\end{align*}
Thus, if $\rho > 4$, then we can take $\nu = 0$ in Theorem~\ref{Thm:LowerBound} to obtain a minimax lower bound of order $n^{-4/(4+d)}/\log^2 n$; on the other hand, if $\rho \leq 4$, then we can take $\nu=\log^{1/2}n$ to obtain a minimax lower bound of order $n^{-(\frac{\rho}{\rho+d}+\epsilon)}$, for every $\epsilon > 0$.  Combining this result with Theorem~\ref{thm:exriskRdSSadapt}, we see that for every $\rho \in (0,\infty)$, our semi-supervised local-$k$-nearest neighbour classifier attains the minimax optimal rate over the class $\mathcal{P}_{d,\theta} \cap \mathcal{Q}_{d,\gamma,\lambda}$ up to polylogarithmic factors when $\rho > 4$ and up to subpolynomial factors when $\rho \leq 4$.
\end{example}

\section{Proofs}
\label{sec:appendix}
The proofs of Theorems~\ref{thm:exriskRdstandard},~\ref{thm:exriskRdOadapt} and~\ref{thm:exriskRdSSadapt} rely on the general asymptotic expansion presented in Theorem~\ref{thm:exriskRtn} below.  We begin with some further notation.  Define the $d \times n$ matrices $X^n := (X_1 \dots X_n)$ and $x^n := (x_1 \dots x_n)$.  Write
\[
 \hat{\mu}_n(x) =  \hat{\mu}_n(x,x^n) := \mathbb{E}\{\hat{S}_n(x) | X^n = x^n\} = \frac 1 {k_{\mathrm{L}}(x)} \sum_{i=1}^{k_{\mathrm{L}}(x)} \eta(x_{(i)}),
\]
and
\[
 \hat{\sigma}_n^2(x)  \!= \! \hat{\sigma}_n^2(x, x^n) \! := \!\mathrm{Var}\{\hat{S}_n(x) | X^n = x^n\}  =  \frac 1 {k_{\mathrm{L}}(x)^2}\sum_{i=1}^{k_{\mathrm{L}}(x)} \eta(x_{(i)})\{1-\eta(x_{(i)})\}.
\]
Here we have used the fact that the ordered labels $Y_{(1)}, \ldots, Y_{(n)}$ are independent given $X^n$, satisfying $\mathbb{P}(Y_{(i)} = 1 | X^{n}) = \eta(X_{(i)})$.  Since $\eta$ takes values in $[0,1]$ it is clear that $0 \leq \hat{\sigma}_n^2(x) \leq \frac 1 {4k_{\mathrm{L}}(x)}$ for all $x \in \mathbb{R}^d$.  Further, write $\mu_n(x) := \mathbb{E}\{\hat{S}_n(x)\} = \frac 1 {k_{\mathrm{L}}(x)} \sum_{i=1}^{k_{\mathrm{L}}(x)} \mathbb{E}\eta(X_{(i)})$ for the unconditional expectation of $\hat{S}_n(x)$.  Recall also that $p_{r}(x) = P_{X}\bigl(B_r(x)\bigr)$.

\subsection{A general asymptotic expansion}
Let 
\[
c_n := \sup_{x_0 \in \mathcal{S}} \ell\biggl(\frac{k_{\mathrm{L}}(x_0)}{n-1}\biggr).
\]
Further, for $x \in \mathbb{R}^d$, let
\begin{equation}
\label{eq:deltan}
\delta_n(x) = \delta_{n,\mathrm{L}}(x):= \frac{k_{\mathrm{L}}(x)}{n-1} c_n^d\log^d\Bigl(\frac{n-1}{k_\mathrm{L}(x)}\Bigr).  
\end{equation}
Recall that $\mathcal{S} = \{x \in \mathbb{R}^d : \eta(x) = 1/2\}$, and note that by Proposition~\ref{Prop:DG1} in Appendix~\ref{Sec:ManifoldIntro}, for $\epsilon >0$, we can write 
\[
\mathcal{S}^{\epsilon} = \biggl\{x_0 + t \frac{\dot{\eta}(x_0)}{\|\dot{\eta}(x_0)\|} : x_0 \in \mathcal{S}, |t| < \epsilon \biggr\}.
\]
Let
\begin{equation}
\label{Eq:epsn}
\epsilon_{n} : = \frac{1}{c_n\beta^{1/2} \log^{1/2} (n-1)},
\end{equation}
and recall the definition of the function $a(\cdot)$ in~\eqref{Eq:a}.
\begin{theorem}
\label{thm:exriskRtn}
Fix $d \in \mathbb{N}$ and $\theta = (\epsilon_0,M_0,\rho,\ell,g) \in \Theta$ such that $\mathcal{P}_{d,\theta} \neq \emptyset$.  For $n$ sufficiently large, let $\mathcal{R}_n \subseteq \bigl\{x \in \mathbb{R}^d: \bar{f}(x) \geq \delta_n(x) \bigr\}$ be a $d$-dimensional manifold.  Write $\partial \mathcal{R}_n$ for the topological boundary of $\mathcal{R}_n$, let $(\partial \mathcal{R}_n)^\epsilon := \partial \mathcal{R}_n + \epsilon\bar{B}_1(0)$, and let $\mathcal{S}_n := \mathcal{S} \cap \mathcal{R}_n$.  For $\beta \in (0,1/2)$ and $\tau > 0$ define the class of functions 
\[
K_{\beta,\tau} \equiv K_{\beta,\tau,n} := \biggl\{k_{\mathrm{L}} : \mathbb{R}^d \to K_\beta : \ \sup_{x_0 \in \mathcal{S}_n} \sup_{|t| < \epsilon_n} \biggl|\frac{k_\mathrm{L}\bigl(x_0 + t \frac{\dot{\eta}(x_0)}{\|\dot{\eta}(x_0)\|}\bigr)}{k_\mathrm{L}(x_0)} - 1\biggr| \leq \tau\biggr\}.
\]
Then for each $\beta \in (0,1/2)$ and each $\tau = \tau_n$ with $\tau_n \searrow 0$, we have 
\begin{align*}
R_{\mathcal{R}_n}(\hat{C}_n^{k_{\mathrm{L}}\mathrm{nn}}) &-  R_{\mathcal{R}_n}(C^{\mathrm{Bayes}}) 
\\& = \int_{\mathcal{S}_n} \frac{\bar{f}(x_0)}{\|\dot\eta(x_0)\|}\biggl\{\frac{1}{4k_{\mathrm{L}}(x_0)} + \Bigl(\frac{k_{\mathrm{L}}(x_0)}{n\bar{f}(x_0)}\Bigr)^{4/d}a(x_0)^2\biggr\} \, d\mathrm{Vol}^{d-1}(x_0) \\
&\hspace{100pt}   + W_{n,1} + W_{n,2}
\end{align*}
as $n \to \infty$, where $\sup_{P \in \mathcal{P}_{d,\theta}} \sup_{k_{\mathrm{L}} \in K_{\beta,\tau}} |W_{n,1}|/\gamma_n(k_{\mathrm{L}}) \rightarrow 0$ with
\begin{align*}
\gamma_n(k_{\mathrm{L}}) &:= \int_{\mathcal{S}_n} \frac{\bar{f}(x_0)}{\|\dot\eta(x_0)\|}\biggl\{\frac{1}{4k_{\mathrm{L}}(x_0)} + \Bigl(\frac{k_{\mathrm{L}}(x_0)}{n\bar{f}(x_0)}\Bigr)^{4/d}\ell\bigl(\bar{f}(x_0)\bigr)^2\biggr\} \, d\mathrm{Vol}^{d-1}(x_0),
\end{align*}
and where $\limsup_{n \rightarrow \infty} \sup_{P \in \mathcal{P}_{d,\theta}} \sup_{k_{\mathrm{L}} \in K_{\beta,\tau}} |W_{n,2}|/P_X\bigl((\partial\mathcal{R}_{n})^{\epsilon_n} \cap \mathcal{S}^{\epsilon_n} \bigr) \leq 1$.
\end{theorem}

\begin{proof}[Proof of Theorem~\ref{thm:exriskRtn}]
First observe that
\begin{align}
\label{eq:exriskint}
R_{\mathcal{R}_{n}}&(\hat{C}_n^{k_{\mathrm{L}}\mathrm{nn}}) - R_{\mathcal{R}_{n}}(C^{\mathrm{Bayes}}) \nonumber \\
& = \int_{\mathcal{R}_{n}} \bigl[ \mathbb{P}\{\hat{S}_n(x) < 1/2\} - \mathbbm{1}_{\{\eta(x) < 1/2\}}\bigr] \{2\eta(x) - 1\} \bar{f}(x) \, dx.
\end{align} 
The proof is presented in seven steps.   We will see that the dominant contribution to the integral in~\eqref{eq:exriskint} arises from a small neighbourhood about the Bayes decision boundary, i.e.\ the region $\mathcal{S}^{\epsilon_n} \cap \mathcal{R}_n$.  On $\mathcal{R}_{n}\setminus\mathcal{S}^{\epsilon_n}$, the $k_{\mathrm{L}}$nn classifier agrees with the Bayes classifier with high probability (asymptotically).  More precisely, we show in Step~4 that 
\[
\sup_{P \in \mathcal{P}_{d,\theta}} \sup_{k_{\mathrm{L}} \in K_{\beta,\tau}} \sup_{x \in \mathcal{R}_{n}\setminus\mathcal{S}^{\epsilon_n}}| \mathbb{P}\{\hat{S}_n(x) < 1/2\} - \mathbbm{1}_{\{\eta(x) < 1/2\}}| = O(n^{-M}),
\]
for each $M>0$, as $n \to \infty$.  In Steps~1,~2  and~3, we derive the key asymptotic properties of the bias, conditional (on $X^n$) bias and variance of $\hat{S}_n(x)$ respectively.  In Step 5 we show that the integral over $\mathcal{S}^{\epsilon_n} \cap \mathcal{R}_n$ can be decomposed into an integral over $\mathcal{S}_n$ and one perpendicular to $\mathcal{S}$.  Step~6 is dedicated to combining the results of Steps 1 - 5; we derive the leading order terms in the asymptotic expansion of the integral in~\eqref{eq:exriskint}.  Finally, we bound the remaining error terms to conclude the proof in Step 7, which is presented in Appendix~\ref{sec:step7}. To ease notation, where it is clear from the context, we write $k_{\mathrm{L}}$ in place of $k_{\mathrm{L}}(x)$. 

\bigskip

\textbf{Step 1}: Let $\mu_n(x) := \mathbb{E}\{\hat{S}_n(x)\}$, and for $x_0 \in \mathcal{S}$ and $t \in \mathbb{R}$, write $x = x(x_0,t) := x_0 + t\frac{\dot{\eta}(x_0)}{\|\dot{\eta}(x_0)\|}$.  We show that
\[
\mu_n(x) - \eta(x) - \Big(\frac{k_{\mathrm{L}}(x)}{n\bar{f}(x)}\Big)^{2/d} a(x) = o\biggl(\Big(\frac{k_{\mathrm{L}}(x_0)}{n \bar{f}(x_0)}\Big)^{2/d}\ell\bigl(\bar{f}(x_0)\bigr)\biggr),
\]
uniformly for $P \in \mathcal{P}_{d,\theta}$, $k_{\mathrm{L}} \in K_{\beta,\tau}$, $x_0 \in  \mathcal{S}_n$ and $|t| < \epsilon_n$.  Write
\begin{align*}
&\mu_n(x) - \eta(x) = \frac{1}{k_{\mathrm{L}}(x)}\sum_{i=1}^{k_{\mathrm{L}}(x)}\mathbb{E}\{\eta(X_{(i)}) - \eta(x)\}
\\ & =  \frac{1}{k_{\mathrm{L}}(x)}\sum_{i=1}^{k_{\mathrm{L}}(x)} \mathbb{E}\{(X_{(i)} - x)^T\dot{\eta}(x)\}  + \frac{1}{2} \mathbb{E}\{(X_{(i)} - x)^T\ddot{\eta}(x)(X_{(i)} - x)\} + R_{1}, 
\end{align*}
where we show in Step 7 that 
\begin{equation}
\label{Eq:R1}
|R_{1}| = o\biggl\{\biggl(\frac{k_{\mathrm{L}}(x_0)}{n\bar{f}(x_0)}\biggr)^{2/d}\biggr\}
\end{equation}
uniformly for $P \in \mathcal{P}_{d,\theta}$, $k_{\mathrm{L}} \in K_{\beta,\tau}$, $x_0 \in \mathcal{S}_n$ and $|t| < \epsilon_n$. 

The density of $X_{(i)} - x$ at $u\in \mathbb{R}^d$ is given by
\begin{equation}
\label{Eq:fi}
f_{(i)}(u) := n\bar{f}(x+u) \binom{n-1}{i-1} p_{\|u\|}^{i-1}(1-p_{\|u\|})^{n-i} = n\bar{f}(x+u)p_{\|u\|}^{n-1}(i-1),
\end{equation}
where $p_{\|u\|}=p_{\|u\|}(x)$ and $p_{\|u\|}^{n-1}(i-1)$ denotes the probability that a $\mathrm{Bin}(n-1, p_{\|u\|})$ random variable equals $i-1$.  Now let 
\begin{equation}
\label{eq:rn}
r_{n} = r_n(x) := \biggl\{\frac{2k_{\mathrm{L}}(x)}{(n-1)\bar{f}(x)a_d}\biggr\}^{1/d}.
\end{equation}
We show in Step 7 that
\begin{equation}
\label{eq:R2}
R_2 := \sup_{P \in \mathcal{P}_{d,\theta}} \sup_{k_{\mathrm{L}} \in K_{\beta,\tau}}\sup_{x_0 \in  \mathcal{S}_n} \sup_{|t| < \epsilon_n} \mathbb{E}\{\|X_{(k_{\mathrm{L}})} - x\|^2 \mathbbm{1}_{\{\|X_{(k_{\mathrm{L}})} - x\| \geq r_{n}\}}\} = O(n^{-M}),
\end{equation}
for each $M>0$, as $n \to \infty$.  It follows from~\eqref{Eq:fi} and~\eqref{eq:R2}, together with the upper bound on $\sup_{x \in \mathcal{S}^{2\epsilon_0}} \|\dot{\eta}(x)\|$ in \textbf{(A.3)} that
\[
\mathbb{E}\{(X_{(i)} - x)^T\dot{\eta}(x)\} = \! \int_{B_{r_n}(0)}\!\!\!\! \!\!\!\!\! \dot{\eta}(x)^T u n \{\bar{f}(x+u) - \bar{f}(x)\} p_{\|u\|}^{n-1}(i-1) \, du  +  O(n^{-M}),
\]
uniformly for $P \in \mathcal{P}_{d,\theta}$, $k_{\mathrm{L}} \in K_{\beta,\tau}$, $i \in \{1,\ldots,k_{\mathrm{L}}\}$, $x_0 \in  \mathcal{S}_n$ and $|t| < \epsilon_n$.  Similarly, using the upper bound on $\sup_{x \in \mathcal{S}^{2\epsilon_0}}\|\ddot{\eta}(x)\|_{\mathrm{op}}$ in \textbf{(A.3)},
\[
\mathbb{E}\{(X_{(i)} - x)^T\ddot{\eta}(x)(X_{(i)} - x)\} =\! \int_{B_{r_n}(0)} \!\!\! \! \! \! \! \! u^T\ddot{\eta}(x) u n \bar{f}(x+u)p_{\|u\|}^{n-1}(i-1) \, du + O(n^{-M}),
\]
uniformly for $P \in \mathcal{P}_{d,\theta}$, $k_{\mathrm{L}} \in K_{\beta,\tau}$, $i \in \{1,\ldots,k_{\mathrm{L}}\}$, $x_0 \in  \mathcal{S}_n$ and $|t| < \epsilon_n$.  Hence, summing over $i$, we see that
\begin{align*}
& \frac{1}{k_{\mathrm{L}}}\sum_{i=1}^{k_{\mathrm{L}}} \mathbb{E}\{(X_{(i)} - x)^T\dot{\eta}(x)\}  + \frac{1}{2k_{\mathrm{L}}}\sum_{i=1}^{k_{\mathrm{L}}} \mathbb{E}\{(X_{(i)} - x)^T\ddot{\eta}(x)(X_{(i)} - x)\}
\\ &  = \int_{B_{r_n}(0)} \Bigl[ \dot{\eta}(x)^T u n \{\bar{f}(x+u) - \bar{f}(x)\}  +  \frac{1}{2} u^T\ddot{\eta}(x) u n \bar{f}(x+u)\Bigr] q_{\|u\|}^{n-1}(k_{\mathrm{L}}) \, du 
\\ & \hspace{300pt} + O(n^{-M}),
\end{align*}
uniformly for $P \in \mathcal{P}_{d,\theta}$, $k_{\mathrm{L}} \in K_{\beta,\tau}$, $i \in \{1,\ldots,k_{\mathrm{L}}\}$, $x_0 \in  \mathcal{S}_n$ and $|t| < \epsilon_n$, where  $q_{\|u\|}^{n-1}(k_{\mathrm{L}})$ denotes the probability that a $\mathrm{Bin}(n-1, p_{\|u\|})$ random variable is less than~$k_{\mathrm{L}}$.  Let $n_0 \in \mathbb{N}$ be large enough that 
\[
\epsilon_n + \sup_{x_0 \in \mathcal{S}_n} \sup_{|t| < \epsilon_n} r_n(x) < \epsilon_0
\] 
for $n \geq n_0$.  That this is possible follows from the fact that, for $\epsilon_n < \epsilon_0$,
\begin{align}
\label{eq:kfbar}
  \sup_{P \in \mathcal{P}_{d,\theta}}&\sup_{k_{\mathrm{L}} \in K_{\beta,\tau}} \sup_{x_0 \in \mathcal{S}_n} \sup_{|t| < \epsilon_n} \max \Bigl\{ \Bigl| \frac{k_\mathrm{L}(x)}{k_\mathrm{L}(x_0)} -1 \Bigr| , \Bigl| \frac{\bar{f}(x)}{\bar{f}(x_0)} -1 \Bigr| \Bigr\} \nonumber \\
                                     & \leq \sup_{P \in \mathcal{P}_{d,\theta}} \max\Bigl\{ \tau , c_n \epsilon_n + \frac{c_n \epsilon_n^2}{2}\Bigr\} \nonumber \\
  &\leq  \max\biggl\{ \tau , \frac{1}{\beta^{1/2}\log^{1/2}(n-1)} + \frac{1}{2\beta \log(n-1)}\biggr\} \rightarrow 0.
\end{align}
By a Taylor expansion of $\bar{f}$ and assumption \textbf{(A.2)}, for all $x_0 \in \mathcal{S}_n$, $|t| < \epsilon_n$, $\|u\| < r_{n}$ and $n \geq n_0$, 
\[
\Bigl| \bar{f}(x+u) - \bar{f}(x) - u^T\dot{\bar{f}}(x) \Bigr| \leq \frac{\|u\|^2}{2} \!\!\sup_{s \in B_{\|u\|}(0)} \| \ddot{\bar{f}}(x+s)\|_{\mathrm{op}} \leq \frac{\|u\|^{2}}{2} \bar{f}(x_0) \ell\bigl(\bar{f}(x_0)\bigr).
\]
Hence, for $x_0 \in \mathcal{S}_n$, $|t| < \epsilon_n$, $r < r_n$ and $n \geq n_0$,
\begin{align}
\label{eq:smallball1}
| p_{r}(x) &- \bar{f}(x) a_d r^{d} |  \leq \int_{B_{r}(0)} |\bar{f}(x+u) - \bar{f}(x) - u^{T}\dot{\bar{f}}(x) | \, du \nonumber
\\ & \leq \frac{1}{2} \bar{f}(x_0)  \ell\bigl(\bar{f}(x_0)\bigr)\int_{B_{r}(0) } \|u\|^{2} \, du =  \frac{da_d}{2(d+2)} \bar{f}(x_0) \ell\bigl(\bar{f}(x_0)\bigr)r^{d+2}. 
\end{align}
Now, for $v \in B_1(0)$, $x_0 \in \mathcal{S}_n$, $|t| < \epsilon_n$ and $n \geq n_0$,
\begin{align*}
k_{\mathrm{L}}(x) - (n-1)p_{\|v\|r_n} & = k_{\mathrm{L}}(x) - (n-1)\bar{f}(x)a_d \|v\|^dr_n^d + R_3
\\&  = k_{\mathrm{L}}(x) (1-2\|v\|^d) + R_3,
\end{align*} 
where
\begin{align*}
|R_3| &\leq  \frac{d a_{d}(n-1) \bar{f}(x_0) \ell\bigl(\bar{f}(x_0)\bigr) \|v\|^{d+2} r_n^{d+2}}{2(d+2)}
\\ & \leq \frac{ 2^{2/d} d k_{\mathrm{L}}(x)}{a_d^{2/d} (d+2) \log^{2}\bigl(\frac{n-1}{k_{\mathrm{L}}(x_0)}\bigr)} \Bigl( \frac{\bar{f}(x_0)}{\bar{f}(x)} \Bigr)^{1+2/d} \Bigl( \frac{k_\mathrm{L}(x)}{k_\mathrm{L}(x_0)} \Bigr)^{2/d}. 
\end{align*} 
It follows from~\eqref{eq:kfbar} that there exists $n_1 \in \mathbb{N}$ such that, for all $x_0 \in \mathcal{S}_n$, $|t| < \epsilon_n$, $\|v\|^d \in (0,1/2 -1/\log{((n-1)/k_{\mathrm{L}}(x_0))}]$ and $n \geq n_1$,
\[
k_{\mathrm{L}}(x) - (n-1)p_{\|v\|r_n}  \geq \frac{k_{\mathrm{L}}(x)}{\log((n-1)/k_{\mathrm{L}}(x_0))},
\]
Similarly, for all $\|v\|^d \in [1/2 + 1/\log((n-1)/k_{\mathrm{L}}(x_0)),1)$ and $n \geq n_1$,
\[
(n-1)p_{\|v\|r_n} - k_{\mathrm{L}}(x) \geq \frac{k_{\mathrm{L}}(x)}{\log((n-1)/k_{\mathrm{L}}(x_0))}.
\]
Hence, by Bernstein's inequality, we have that for each $M > 0$,
\[
\sup_{P \in \mathcal{P}_{d,\theta}} \sup_{k_{\mathrm{L}} \in K_{\beta,\tau}} \sup_{x_0 \in \mathcal{S}_n} \sup_{|t| < \epsilon_n} \sup_{\|v\|^d \in \bigl(0, \frac{1}{2}-\frac{1}{\log((n-1)/k_{\mathrm{L}}(x_0))}\bigr]} \! \! \! \! \! \! 1 - q_{\|v\|r_n}^{n-1} (k_{\mathrm{L}}(x)) = O(n^{-M}),
\]
and
\begin{equation}
\label{eq:binbound2}
\sup_{P \in \mathcal{P}_{d,\theta}} \sup_{k_{\mathrm{L}} \in K_{\beta,\tau}}\sup_{x_0 \in \mathcal{S}_n} \sup_{|t| < \epsilon_n}\sup_{\|v\|^d \in \bigl[\frac{1}{2}-\frac{1}{\log((n-1)/k_{\mathrm{L}}(x_0))}, 1\bigr)} q_{\|v\|r_n}^{n-1} (k_{\mathrm{L}}(x)) = O(n^{-M}).
\end{equation}
We conclude that
\begin{align}
\label{eq:biasexpd}
\frac{1}{k_{\mathrm{L}}(x)} & \int_{B_{r_n}(0)} \Bigl[\dot{\eta}(x)^T u n \{\bar{f}(x+u) - \bar{f}(x)\}  \nonumber
 \\ & \hspace{120 pt} +  \frac{1}{2}u^T\ddot{\eta}(x) u n \bar{f}(x+u)\Bigr] q_{\|u\|}^{n-1}(k_{\mathrm{L}}(x)) \, du \nonumber
\\ & = \frac{1}{k_{\mathrm{L}}(x)} \int_{B_{2^{-1/d}r_n}(0)} \Bigl[\dot{\eta}(x)^T u n \{\bar{f}(x+u) - \bar{f}(x)\}  \nonumber
\\ & \hspace{120 pt} + \frac{1}{2} u^T\ddot{\eta}(x) u n \bar{f}(x+u)\Bigr] \, du + R_{41} \nonumber
\\ & = \Big(\frac{k_{\mathrm{L}}(x)}{n}\Big)^{2/d} \frac{\sum_{j=1}^{d} \{\eta_j(x)\bar{f}_j(x)  +  \frac{1}{2} \eta_{jj}(x) \bar{f}(x) \}}{(d+2) a_d^{2/d} \bar{f}(x)^{1+2/d}} + R_{41} + R_{42} \nonumber
\\ &  = \Bigl(\frac{k_{\mathrm{L}}(x)}{n\bar{f}(x)}\Bigr)^{2/d} a(x) + R_{41} + R_{42},
\end{align}
where 
\[
|R_{41}| + |R_{42}| = o\biggl(\Bigl(\frac{k_{\mathrm{L}}(x_0)}{n\bar{f}(x_0)}\Bigr)^{2/d}\ell\bigl(\bar{f}(x_0)\bigr)\biggr),
\]
uniformly for $P \in \mathcal{P}_{d,\theta}$, $k_{\mathrm{L}} \in K_{\beta,\tau}$, $x_0 \in  \mathcal{S}_n$ and $|t| < \epsilon_n$.

\bigskip

\textbf{Step 2}: Recall that $\hat{\sigma}_n^2(x,x^n) = \mathrm{Var}\{\hat{S}_n(x)| X^n = x^n\}$.  We show that 
\begin{equation}
\label{eq:varexp}
 \Big| \hat\sigma^2_n(x, X^n) - \frac{1}{4k_{\mathrm{L}}} \Big| = o_p(1/k_{\mathrm{L}}),
\end{equation}
uniformly for $P \in \mathcal{P}_{d,\theta}$, $k_{\mathrm{L}} \in K_{\beta,\tau}$, $x_0 \in  \mathcal{S}_n$ and $|t| < \epsilon_n$.  Recall that 
\[
\hat{\sigma}_n^2(x,X^n) = \frac 1 {k_{\mathrm{L}}^2} \sum_{i=1}^{k_{\mathrm{L}}} \eta(X_{(i)})\{1-\eta(X_{(i)})\}.
\]
Let $n_2 \in \mathbb{N}$ be large enough that $1- c_n \epsilon_n - \frac{d+1}{d+2}c_n\epsilon_n^2 \geq \epsilon_0$ for $n \geq n_2$.  Then for $n \geq \max\{n_0, n_2\}$, $P \in \mathcal{P}_{d,\theta}$, $r < \epsilon_n$, $x_0 \in  \mathcal{S}_n$ and $|t| < \epsilon_n$, we have by \textbf{(A.2)} and a very similar argument to that in~\eqref{eq:smallball1} that 
\begin{equation}
\label{eq:plowerbound}
p_{r}(x) \geq \epsilon_{0} a_d r^{d} \bar{f}(x_0) \geq \epsilon_{0} a_d r^{d} \delta_n(x_0).
\end{equation}
Now suppose that $z_{1}, \dots, z_{N} \in \mathcal{R}_{n} \cup \mathcal{S}_n^{\epsilon_n}$ are such that $\|z_{j} - z_{\ell}\| \geq \epsilon_{n}/6$ for all $j \neq \ell$,  but $\sup_{x \in \mathcal{R}_{n} \cup \mathcal{S}_n^{\epsilon_n}} \min_{j=1, \dots, N} \|x - z_{j}\| < \epsilon_{n}/6$.  We have by \textbf{(A.2)} that
\[
1 = P_{X}(\mathbb{R}^{d}) \geq  \sum_{j=1}^{N} p_{\epsilon_{n}/12}(z_{j}) \geq \frac{N \epsilon_{0} a_d \beta^{d/2}\log^{d/2}(n-1)}{12^d(n-1)^{1-\beta}}.
\]
For each $j=1,\ldots,N$, choose 
\[
z_j' \in \argmax_{z \in B_{z_j}(\epsilon_n/6) \cap (\mathcal{R}_n \cup \mathcal{S}_n^{\epsilon_n})} k_{\mathrm{L}}(z).
\]
Now, given $x \in \mathcal{R}_{n} \cup \mathcal{S}_n^{\epsilon_n}$, let $j_{0} := \argmin_j \| x - z_{j}\|$, so that $B_{\epsilon_{n}/6}(z_{j_0}') \subseteq B_{\epsilon_{n}/2}(x)$. Thus, if there are at least $k_\mathrm{L}(z_j')$ points among $\{x_1,\ldots,x_n\}$ inside each of the balls $B_{\epsilon_{n}/6}(z_{j}')$, then for every $x \in \mathcal{R}_{n} \cup \mathcal{S}_n^{\epsilon_n}$ there are at least $k_\mathrm{L}(x)$ of them in $B_{\epsilon_n/2}(x)$.  Moreover by~\eqref{eq:kfbar}, \eqref{eq:plowerbound} and \textbf{(A.2)},
\[
\min_{j=1, \ldots, N} \Bigl\{ np_{\epsilon_n /6}(z_j') - 2k_{\mathrm{L}}(z_j') \Bigr\} \geq  (n-1)^\beta 
\]
for all $P \in \mathcal{P}_{d,\theta}$, $k_{\mathrm{L}} \in K_{\beta,\tau}$ and $n \geq n_3$, say.  Define $A_{k_{\mathrm{L}}} := \bigl\{\|X_{(k_{\mathrm{L}})}(x) - x\|  < \epsilon_{n}/2 \ \mbox{for all} \ x \in \mathcal{R}_{n} \cup \mathcal{S}_n^{\epsilon_n}\bigr\}$.  Then by a standard binomial tail bound \citep[][Equation~(6), p.~440]{Shorack:86}, for $n \geq n_3$ and any $M > 0$, 
\begin{align}
\label{Eq:CompProb}
\mathbb{P}(A_{k_{\mathrm{L}}}^c) & = \mathbb{P}\Bigl\{ \sup_{x \in \mathcal{R}_{n} \cup \mathcal{S}_n^{\epsilon_n}} \|X_{(k_{\mathrm{L}}(x))}(x) - x\| \geq \epsilon_{n}/2\Bigr\} \nonumber \\ 
& \leq \mathbb{P}\Bigl\{\max_{j=1, \ldots,N}  \|X_{(k_{\mathrm{L}}(z_j))}(z_{j}') - z_{j}'\| \geq \epsilon_{n}/6 \Bigr\} \nonumber \\
& \leq \sum_{j=1}^{N} \mathbb{P}\bigl\{\|X_{(k_{\mathrm{L}}(z_j))}(z_{j}') - z_{j}'\| \geq \epsilon_{n}/6 \bigr\} \nonumber \\
& \leq N \max_{j=1,\ldots,N} \exp\Bigl(-\frac{1}{2}np_{\epsilon_n/6}(z_j') +k_{\mathrm{L}}(z_j')\Bigr)  = O(n^{-M}),
\end{align}
uniformly for $P \in \mathcal{P}_{d,\theta}$ and $k_{\mathrm{L}} \in K_{\beta,\tau}$.  Now, for $3\epsilon_n/2 < 2\epsilon_0$, 
\begin{align*}
  \sup_{P \in \mathcal{P}_{d,\theta}} \sup_{k_{\mathrm{L}} \in K_{\beta,\tau}} &\sup_{x_0 \in \mathcal{S}_n} \sup_{|t| < \epsilon_n}  \sup_{x^n \in A_{k_{\mathrm{L}}} }  \max_{1\leq i \leq k_{\mathrm{L}}(x)} |\eta(x_{(i)}(x)) - 1/2| \\
  &\leq 3M_0\sup_{P \in \mathcal{P}_{d,\theta}} \sup_{k_{\mathrm{L}} \in K_{\beta,\tau}} \frac{\epsilon_n}{2} \leq \frac{3M_0}{2\beta^{1/2}\log^{1/2}(n-1)} \to 0.
\end{align*}
It follows that
\begin{equation}
\label{Eq:etavar}
\sup_{x^n \in A_{k_{\mathrm{L}}}}\biggl|\frac{1}{k_{\mathrm{L}}(x)^2}\sum_{i=1}^{k_{\mathrm{L}}(x)}  \eta(x_{(i)}(x) )\{1-\eta(x_{(i)}(x))\}  - \frac{1}{4k_{\mathrm{L}}(x)}\biggr| = o\Bigl(\frac{1}{k_{\mathrm{L}}(x)}\Bigr)
\end{equation}
as $n \to \infty$, uniformly for $P \in \mathcal{P}_{d,\theta}$, $k_{\mathrm{L}} \in K_{\beta,\tau}$, $x_0 \in \mathcal{S}_n$ and $|t|< \epsilon_n$.  The claim~\eqref{eq:varexp} follows from~\eqref{Eq:CompProb} and~\eqref{Eq:etavar}.

\bigskip

\textbf{Step 3}: \ \ In this step, we emphasise the dependence of $\hat{\mu}_n(x,x^n) = \mathbb{E}\{\hat{S}_n(x)| X^n = x^n\}$ on $k_{\mathrm{L}}$ by writing it as $\hat{\mu}_n^{(k_{\mathrm{L}})}(x,x^n)$.  We show that
\begin{equation}
\label{eq:biasvar}
\mathrm{Var}\{\hat{\mu}_n^{(k_{{\mathrm{L}}})}(x,X^n) \} =O \biggl\{ \frac{1}{k_{\mathrm{L}}(x_{0})} \biggl( \frac{k_{\mathrm{L}}(x_0)}{n \bar{f}(x_0)} \biggr)^{2/d} \biggr\}
\end{equation}
uniformly for $P \in \mathcal{P}_{d,\theta}$, $k_{\mathrm{L}} \in K_{\beta,\tau}$, $x_0 \in \mathcal{S}_n$ and $|t|< \epsilon_n$. We will write $X^{n,j}\  := \ (X_1 \ldots X_{j-1} \ X_{j+1} \ldots X_n)$, considered as a random $d \times (n-1)$ matrix, so that
\[
	\hat{\mu}_n^{(k_{\mathrm{L}})}(x,X^n) - \hat{\mu}_{n-1}^{(k_{\mathrm{L}})}(x,X^{n,(i)}) = \frac{1}{k_{\mathrm{L}}} \{ \eta(X_{(i)}) - \eta(X_{(k_{\mathrm{L}}+1)}) \} \mathbbm{1}_{\{ i \leq k_{\mathrm{L}}\}}.
\]
It follows from the Efron--Stein inequality \citep[e.g.][Theorem~3.1]{Boucheron:13} that
\begin{align}
\label{eq:efronstein}
	\mathrm{Var}\{\hat{\mu}_n^{(k_{\mathrm{L}})}(x,X^n) \} &\leq \sum_{i=1}^n \mathbb{E} \bigl[ \{\hat{\mu}_n^{(k_{\mathrm{L}})}(x,X^n) - \hat{\mu}_{n-1}^{(k_{\mathrm{L}})}(x,X^{n,(i)}) \}^2 \bigr] \nonumber \\
&= \frac{1}{k_{\mathrm{L}}^2} \sum_{i=1}^{k_{\mathrm{L}}} \mathbb{E}\bigl[ \{ \eta(X_{(i)}) - \eta(X_{(k_{\mathrm{L}}+1)}) \}^2 \bigr] \nonumber \\
	& \leq \frac{2}{k_{\mathrm{L}}^2} \sum_{i=1}^{k_{\mathrm{L}}} \mathbb{E} \bigl[ \{ \eta(X_{(i)}) - \eta(x) \}^2 + \{ \eta(X_{(k_{\mathrm{L}}+1)}) - \eta(x) \}^2 \bigr].
\end{align}
Recall the definition of $r_n$ given in \eqref{eq:rn}. Now observe that, for $\max (\epsilon_n , r_n) \leq \epsilon_0$ and all $M>0$ we have that
\begin{align}
\label{eq:rn2}
	&\max_{i \in \{1, \ldots, k_{\mathrm{L}}+1\}} \mathbb{E} \bigl[ \{ \eta(X_{(i)}) - \eta(x) \}^2 \bigr] \nonumber \\
	& \hspace{50pt} \leq \max_{i \in \{1, \ldots, k_{\mathrm{L}}+1\}} \mathbb{E} \bigl[ \{ \eta(X_{(i)}) - \eta(x) \}^2 \mathbbm{1}_{\{\|X_{(i)}-x\| \leq r_n\}} \bigr] \nonumber
\\ & \hspace{200 pt} + \mathbb{P}(\|X_{(k_{\mathrm{L}}+1)}-x\| > r_n) \nonumber \\
	& \hspace{50pt} \leq r_n^2M_0 + O(n^{-M}),
\end{align}
uniformly for $P \in \mathcal{P}_{d,\theta}$, $k_{\mathrm{L}} \in K_{\beta,\tau}$, $x_0 \in \mathcal{S}_n$ and $|t|< \epsilon_n$. The final inequality here follows from similar arguments to those used to bound $R_1$. Now~\eqref{eq:biasvar} follows from~\eqref{eq:efronstein} and~\eqref{eq:rn2}.

\bigskip

\textbf{Step 4}: We show that 
\[
\sup_{P \in \mathcal{P}_{d,\theta}} \sup_{k_{\mathrm{L}}  \in K_{\beta,\tau}}\sup_{x \in \mathcal{R}_{n}\setminus\mathcal{S}^{\epsilon_n}}| \mathbb{P}\{\hat{S}_n(x) < 1/2\} - \mathbbm{1}_{\{\eta(x) < 1/2\}}| = O(n^{-M}),
\]
for each $M>0$, as $n \to \infty$.  First, by \textbf{(A.3)} and Proposition~\ref{Prop:DG1} in Section~\ref{Sec:Tubular}, there exists $c_0 > 0$ such that for every $r \in (0,\epsilon_0]$, $P \in \mathcal{P}_{d,\theta}$ and $k_{\mathrm{L}}  \in K_{\beta,\tau}$,
\[
\inf_{x \in \mathcal{R}_{n} \setminus \mathcal{S}^{r}} |\eta(x) - 1/2| \geq c_0\min\biggl\{r \, , \, \inf_{x\in \mathcal{R}_{n} \setminus \mathcal{S}^{\epsilon_0}} \delta_n(x)^{\beta/2}\biggr\}.
\]
Hence, on the event $A_{k_{\mathrm{L}}}$, for $\epsilon_n < \epsilon_0$ and $x \in \mathcal{R}_{n} \setminus \mathcal{S}^{\epsilon_{n}}$, all of the $k_{\mathrm{L}}$ nearest neighbours of $x$ are on the same side of $\mathcal{S}$, so
\begin{align*}
|\hat{\mu}_n(x,X^n) - 1/2| & = \biggl|\frac{1}{k_{\mathrm{L}}} \sum_{i=1}^{k_{\mathrm{L}}} \eta(X_{(i)}) - 1/2\biggr| 
\\ & \geq \! \! \inf_{z \in B_{\epsilon_n/2}(x)}\!\! |\eta(z) - 1/2| \geq \! c_0\min\biggl\{\frac{\epsilon_n}{2}  , \inf_{x\in \mathcal{R}_{n} \setminus \mathcal{S}^{\epsilon_0}} \delta_n(x)^{\beta/2}\biggr\}.
\end{align*}
Now, conditional on $X^n$, $\hat{S}_n(x)$ is the sum of $k_{\mathrm{L}}(x)$ independent terms. Therefore, by Hoeffding's inequality,
\begin{align*}
\sup_{P \in \mathcal{P}_{d,\theta}} &\sup_{k_{\mathrm{L}}  \in K_{\beta,\tau}} \sup_{x\in \mathcal{R}_{n}\setminus \mathcal{S}^{\epsilon_{n}}} \bigl|\mathbb{P}\{\hat{S}_n(x) <1/2\} -  \mathbbm{1}_{\{\eta(x) < 1/2\}}\bigr|   
\\ &= \sup_{P \in \mathcal{P}_{d,\theta}} \sup_{k_{\mathrm{L}}  \in K_{\beta,\tau}} \sup_{x\in \mathcal{R}_{n}\setminus \mathcal{S}^{\epsilon_{n}}} \bigl|\mathbb{E}\{ \mathbb{P}\{\hat{S}_n(x) <1/2 | X^n\}- \mathbbm{1}_{\{\eta(x) < 1/2\}}\bigr|
\\ &\leq \sup_{P \in \mathcal{P}_{d,\theta}} \sup_{k_{\mathrm{L}}  \in K_{\beta,\tau}} \sup_{x\in \mathcal{R}_{n}\setminus \mathcal{S}^{\epsilon_{n}}}\Bigl\{\mathbb{E}\bigl[e^{-2k_{\mathrm{L}} \{ \hat{\mu}_n(x, X^n) -1/2\}^2}\mathbbm{1}_{A_{k_{\mathrm{L}}}}\bigr] + \mathbb{P}(A_{k_{\mathrm{L}}}^c)\Bigr\}
\\&= O(n^{-M})
\end{align*}
for every $M > 0$.  This completes Step~4.

\bigskip

\textbf{Step 5}: It is now convenient to be more explicit in our notation, by writing $x_0^t := x_0 + t \dot{\eta}(x_0)/\|\dot{\eta}(x_0)\|$.  We also let 
\[
\psi(x) :=   \{2\eta(x)-1\} \bar{f}(x)  = \pi_1f_1(x) - \pi_0 f_0(x).
\]  
Recall that $\mathcal{S}_n := \mathcal{S} \cap \mathcal{R}_n$ and let
\[
  W_{n,2} := \biggl(\int_{\mathcal{S}^{\epsilon_{n}} \cap \mathcal{R}_{n}} - \int_{\mathcal{S}_n^{\epsilon_{n}}}\biggr) \psi(x) [\mathbb{P}\{\hat{S}_n(x) < 1/2\} - \mathbbm{1}_{\{\eta(x) <1/2\}} ]  \, dx.
\]
We show that
\begin{align*} 
&\int_{\mathcal{S}^{\epsilon_{n}} \cap \mathcal{R}_{n}}  \psi(x) [\mathbb{P}\{\hat{S}_n(x) < 1/2\} - \mathbbm{1}_{\{\eta(x) <1/2\}} ]  \, dx 
  \\ &= \int_{ \mathcal{S}_n} \int_{-\epsilon_{n}}^{\epsilon_{n}} \psi(x_0^t)[ \mathbb{P}\{\hat{S}_n(x_0^t) < 1/2\} - \mathbbm{1}_{\{t <0\}}]\, dt \, d\mathrm{Vol}^{d-1}(x_0)\{1+o(1)\} \\
  &\hspace{10cm}+ W_{n,2}
\end{align*}
uniformly for $P \in \mathcal{P}_{d,\theta}$ and $k_{\mathrm{L}} \in K_{\beta,\tau}$, and that for all $n \geq 2$,
\begin{equation}
  \label{Eq:Wn2}
  \sup_{P \in \mathcal{P}_{d,\theta}} \sup_{k_{\mathrm{L}}  \in K_{\beta,\tau}} \frac{|W_{n,2}|}{P_X\bigl((\partial\mathcal{R}_n)^{\epsilon_n}\cap \mathcal{S}^{\epsilon_n}\bigr)} \leq 1.
\end{equation}
Now by Proposition~\ref{Prop:DG2} in Section~\ref{Sec:Tubular}, for $\epsilon_n \leq \epsilon_0$, the map $x(x_0, t) = x_0^t$ is a diffeomorphism from $\mathcal{S}_n \times (-\epsilon_n,\epsilon_n)$ to $\mathcal{S}_n^{\epsilon_{n}}$, where
\[
\mathcal{S}_n^\epsilon := \biggl\{x_0 + t \frac{\dot{\eta}(x_0)}{\|\dot{\eta}(x_0)\|} : x_0 \in \mathcal{S}_n, |t| < \epsilon \biggr\}. 
\]
Furthermore, for such $n$, and $|t| < \epsilon_{n}$, $\text{sgn}\{\eta(x_{0}^{t}) -1/2\} = \text{sgn}(t)$.  It follows from this and~\eqref{Eq:Conc} in Section~\ref{Sec:Forms} that 
\begin{align*}
&\int_{\mathcal{S}^{\epsilon_{n}} \cap \mathcal{R}_{n}}  \psi(x) [\mathbb{P}\{\hat{S}_n(x) < 1/2\} - \mathbbm{1}_{\{\eta(x) <1/2\}} ]  \, dx  \\
&= \int_{\mathcal{S}_n^{\epsilon_{n}}} \psi(x)[\mathbb{P}\{\hat{S}_n(x) < 1/2\} - \mathbbm{1}_{\{\eta(x) <1/2\}} ] \, dx + W_{n,2}
\\ &= \int_{ \mathcal{S}_n} \int_{-\epsilon_{n}}^{\epsilon_{n}} \det(I + tB)\psi(x_0^t)[ \mathbb{P}\{\hat{S}_n(x_0^t) < 1/2\} - \mathbbm{1}_{\{t <0\}}]\, dt \, d\mathrm{Vol}^{d-1}(x_0) \\
  &\hspace{10cm}+ W_{n,2},
\end{align*}
where $B$ is defined in~\eqref{Eq:B} in Section~\ref{Sec:Tubular}, and $\det(I+tB) = 1 + o(1)$ as $n \rightarrow \infty$, uniformly for $P \in \mathcal{P}_{d,\theta}$, $x_0 \in \mathcal{S}$ and $t \in (-\epsilon_n,\epsilon_n)$.  Now observe that $(\mathcal{S}^{\epsilon_{n}} \cap \mathcal{R}_{n}) \setminus \mathcal{S}_n^{\epsilon_{n}} \subseteq (\partial \mathcal{R}_n)^{\epsilon_n}\cap \mathcal{S}^{\epsilon_n}$ and $\mathcal{S}_n^{\epsilon_{n}} \setminus (\mathcal{S}^{\epsilon_{n}} \cap \mathcal{R}_{n}) \subseteq (\partial \mathcal{R}_n)^{\epsilon_n}\cap \mathcal{S}^{\epsilon_n}$.  We deduce from this and the definition of $W_{n,2}$ that~\eqref{Eq:Wn2} holds.

\bigskip

\textbf{Step 6}: The last step in the main argument is to show that 
\begin{align*} 
\tilde{W}_{n,1} &:= \int_{\mathcal{S}_n} \int_{-\epsilon_{n}}^{\epsilon_{n}} \psi(x_0^t)[ \mathbb{P}\{\hat{S}_n(x_0^t) < 1/2\} - \mathbbm{1}_{\{t <0\}}] \,dt\, d\mathrm{Vol}^{d-1}(x_0)
  \\ & \hspace{20 pt} - \int_{\mathcal{S}_n} \frac{\bar{f}(x_0)}{\|\dot\eta(x_0)\|}\biggl\{\frac{1}{4k_{\mathrm{L}}(x_0)}+ \Bigl(\frac{k_{\mathrm{L}}(x_0)}{n\bar{f}(x_0)}\Bigr)^{4/d}a(x_0)^2\biggr\} \, d\mathrm{Vol}^{d-1}(x_0) \\
  &= o(\gamma_n(k_\mathrm{L}))
\end{align*}
as $n \to \infty$, uniformly for $P \in \mathcal{P}_{d,\theta}$ and $k_{\mathrm{L}} \in K_{\beta,\tau}$.  First observe that
\begin{align*} 
& \int_{\mathcal{S}_n} \int_{-\epsilon_{n}}^{\epsilon_{n}} \psi(x_0^t)[ \mathbb{P}\{\hat{S}_n(x_0^t) < 1/2\} - \mathbbm{1}_{\{t <0\}}] \,dt\, d\mathrm{Vol}^{d-1}(x_0)
\\ & \hspace{5 pt} =  \int_{\mathcal{S}_n} \int_{-\epsilon_{n}}^{\epsilon_{n}} t\|\dot\psi(x_0)\| [ \mathbb{P}\{\hat{S}_n(x_0^t) < 1/2\} - \mathbbm{1}_{\{t <0\}}] \,dt\, d\mathrm{Vol}^{d-1}(x_0) \{1+o(1)\},
\end{align*}
uniformly for $P \in \mathcal{P}_{d,\theta}$ and $k_{\mathrm{L}} \in K_{\beta,\tau}$.  Now, write $\mathbb{P}\{\hat{S}_n(x_0^t) < 1/2\} - \mathbbm{1}_{\{t <0\}} = \mathbb{E}[\mathbb{P}\{\hat{S}_n(x_0^t) < 1/2 | X^n \} - \mathbbm{1}_{\{t <0\}}].$ Note that, given $X^n$, $\hat{S}_n(x) = \frac{1}{k_{\mathrm{L}}(x)} \sum_{i=1}^{k_{\mathrm{L}}(x)} \mathbbm{1}_{\{Y_{(i)}=1\}}$ is the sum of $k_{\mathrm{L}}(x)$ independent Bernoulli variables, satisfying $\mathbb{P}(Y_{(i)} = 1| X^n) = \eta(X_{(i)})$.  Let $\Phi$ be the standard normal distribution function, and let
\begin{align*}
\hat{\theta}(x) \equiv \hat{\theta}_n(x) &:= -\{ \hat{\mu}_n(x, X^n) -1/2\}/\hat{\sigma}_n(x,X^n) \\
\bar{\theta}(x_0,t) \equiv \bar{\theta}_n(x_0,t) &:= -2k_{\mathrm{L}}(x_0)^{1/2}\biggl\{t\|\dot{\eta}(x_0)\| + \biggl(\frac{k_{\mathrm{L}}(x_0)}{n\bar{f}(x_0)}\biggr)^{2/d}a(x_0)\biggr\}.
\end{align*}
We can write 
\begin{align*}
&\int_{-\epsilon_{n}}^{\epsilon_{n}}  t\|\dot\psi(x_0)\|[ \mathbb{P}\{\hat{S}_n(x_0^t) < 1/2\} - \mathbbm{1}_{\{t <0\}}] \,dt
\\ & \hspace{50 pt}  = \int_{-\epsilon_{n}}^{\epsilon_{n}} t \|\dot{\psi}(x_0)\|\mathbb{E}\bigl\{\Phi\bigl(\hat{\theta}(x_0^t)\bigr) - \mathbbm{1}_{\{t <0\}}\bigr\} \,dt + R_5(x_0)
\\ & \hspace{50 pt} = \int_{-\epsilon_{n}}^{\epsilon_{n}} t \|\dot{\psi}(x_0)\|\bigl\{\Phi\bigl(\bar{\theta}(x_0,t)\bigr) - \mathbbm{1}_{\{t <0\}}\bigr\} \,dt + R_5(x_0) + R_6(x_0),
\end{align*}
where we show in Step 7 that 
\begin{equation}
\label{Eq:R6R7}
\biggl|\int_{\mathcal{S}_n} \bigl\{R_5(x_0) + R_6(x_0)\bigr\} \, d\mathrm{Vol}^{d-1}(x_0)\biggr| = o(\gamma_n(k_\mathrm{L}))
\end{equation}
uniformly for $P \in \mathcal{P}_{d,\theta}$ and $k_{\mathrm{L}} \in K_{\beta,\tau}$.  Then, substituting $u = 2k_{\mathrm{L}}(x_0)^{1/2}t$, we see that
\begin{align*} 
&\int_{-\epsilon_{n}}^{\epsilon_{n}} t \|\dot{\psi}(x_0)\| \bigl[\Phi\bigl(\bar{\theta}(x_0,t)\bigr) - \mathbbm{1}_{\{t <0\}}\bigr] \,dt
\\ &= \frac{1}{4k_{\mathrm{L}}(x_0)} \int_{-2k_{\mathrm{L}}(x_0)^{1/2}\epsilon_{n}}^{2k_{\mathrm{L}}(x_0)^{1/2}\epsilon_{n}} u \|\dot{\psi}(x_0)\|\biggl\{\Phi\biggl(\bar{\theta}\Bigl(x_0,\frac{u}{2k_{\mathrm{L}}(x_0)^{1/2}}\Bigr)\biggr) - \mathbbm{1}_{\{u <0\}}\biggr\} \,du
\\ &= \biggl\{\frac{\bar{f}(x_0)}{4k_{\mathrm{L}}(x_0)\|\dot\eta(x_0)\|} + \Bigl(\frac{k_{\mathrm{L}}(x_0)}{n\bar{f}(x_0)}\Bigr)^{4/d}\frac{\bar{f}(x_0)a(x_0)^2}{\|\dot\eta(x_0)\|}\biggr\} \{1+o(1)\},
\end{align*}
uniformly for $P \in \mathcal{P}_{d,\theta}$, $k_{\mathrm{L}} \in K_{\beta,\tau}$ and $x_0 \in \mathcal{S}_n$.  The conclusion follows by integrating with respect to $d\mathrm{Vol}^{d-1}$ over $\mathcal{S}_n$.

\bigskip
\textbf{Step 7}: It remains to bound the error terms $R_1, R_2, R_5$ and $R_6$ -- these bounds are presented in Appendix~\ref{sec:step7}.
\end{proof}

\subsection{Proof of Theorem~\ref{thm:exriskRdstandard}}

\begin{proof}[Proof of Theorem~\ref{thm:exriskRdstandard}]
Let $k \in K_{\beta}$, and note that since $k_{\mathrm{L}}(x) = k$ is constant, we have that $c_n  =  \ell\bigl(k/(n-1)\bigr)$,  and  $\delta_n\, =\, \frac{k}{n-1} c_n^d \log^d ( \frac{n-1}{k}).$
Now let 
\[
	\mathcal{R}_n = \{x \in \mathbb{R}^d : \bar{f}(x) > \delta_n \} \cap \mathcal{X}_{\bar{f}},
      \]
      and observe that by \citet[][Lemma~10(i)]{Berrett:16}, for $P \in \mathcal{P}_{d,\theta}$,
      \begin{equation}
        \label{Eq:fbarinfty}
        \|\bar{f}\|_\infty^\rho \geq \frac{\rho^\rho d^d}{a_d^\rho M_0^d(\rho+d)^{\rho+d}}.
        \end{equation}
        It follows that we can find $n_0 \in \mathbb{N}$ be large enough that $\mathcal{R}_n$ is non-empty for all $P \in \mathcal{P}_{d,\theta}$, $k \in K_\beta$ and $n \geq n_0$, so that, by Assumption~\textbf{(A.1)}, for $n \geq n_0$ it is an open subset of $\mathbb{R}^d$, and therefore a $d$-dimensional manifold.  Let $\mathcal{S}_n := \mathcal{S} \cap \mathcal{R}_n$,
        \[
	B_{1,n} := \int_{\mathcal{S}_n} \frac{\bar{f}(x_0)}{4\|\dot\eta(x_0)\|} \, d\mathrm{Vol}^{d-1}(x_0) 
\]
and
\[
B_{2,n} := \int_{\mathcal{S}_n} \frac{\bar{f}(x_0)^{1-4/d}}{\|\dot\eta(x_0)\|}a(x_0)^2 \, d\mathrm{Vol}^{d-1}(x_0).
\]
Recalling the definition of $\epsilon_n$ in~\eqref{Eq:epsn}, for $n \geq n_0$, we may apply Theorem~\ref{thm:exriskRtn} with $k_{\mathrm{L}}(x) = k$ for all $x \in \mathbb{R}^d$ to deduce that
\begin{align*}
R_{\mathcal{R}_{n}}(\hat{C}_n^{k\mathrm{nn}}) - R_{\mathcal{R}_{n}}(C^{\mathrm{Bayes}}) = B_{1,n} \frac{1}{k} + B_{2,n} \Bigl(\frac{k}{n}\Bigr)^{4/d} + W_{n,1} + W_{n,2},
\end{align*}
where $\sup_{P \in \mathcal{P}_{d,\theta}} \sup_{k \in K_\beta} |W_{n,1}|/\gamma_n(k) \rightarrow 0$ and where
\[
  \limsup_{n \rightarrow \infty} \sup_{P \in \mathcal{P}_{d,\theta}} \sup_{k \in K_\beta} \frac{|W_{n,2}|}{P_X\bigl((\partial\mathcal{R}_{n})^{\epsilon_n} \cap \mathcal{S}^{\epsilon_n} \bigr)} \leq 1.
\]
We now show that, under the conditions of part (i), $B_{1,n}$ and $B_{2,n}$ are well approximated by integrals over the whole of the manifold $\mathcal{S}$, and that these integrals are uniformly bounded.  Given $x_0 \in \mathcal{S} \cap \{x \in \mathbb{R}^d:\bar{f}(x) > 0\}$, define $\epsilon_0(x_0):=\min \bigl\{ 1, \frac{\epsilon_0\log 2}{2d}, \frac{1}{4 \ell(\bar{f}(x_0))}\bigr\}$.  Then for any $t \in [-\epsilon_0(x_0),\epsilon_0(x_0)]$ we have by~\textbf{(A.2)} and Cauchy--Schwarz that
\begin{align*}
	\biggl| \frac{\bar{f}(x_0^t)}{\bar{f}(x_0)} -1 \biggr| &= \biggl| \frac{ \bar{f}(x_0^t) - \bar{f}(x_0) - (x_0^t - x_0)^T \nabla \bar{f}(x_0)}{\bar{f}(x_0)} + \frac{(x_0^t-x_0)^T \nabla \bar{f}(x_0)}{\bar{f}(x_0)} \biggr| \\
	& \leq \frac{t^2}{2} \ell \bigl(\bar{f}(x_0)\bigr) + |t| \ell\bigl( \bar{f}(x_0)\bigr)  \leq \frac{1}{2}.
\end{align*}
Moreover, writing $\lambda_1,\ldots,\lambda_d$ for the eigenvalues of the matrix $B$ defined in~\eqref{Eq:B}, for $t \in [-\epsilon_0(x_0),\epsilon_0(x_0)]$, we have
\[
  |\log \det(I+tB)| = \biggl|\sum_{j=1}^d \log(1+t\lambda_j)\biggr| \leq 2|t|\sum_{j=1}^d |\lambda_j| \leq 2|t|d\|B\|_{\mathrm{op}} \leq \frac{2|t|d}{\epsilon_0},
\]
so $\det(I+tB) \geq 1/2$.  Hence, for any $\tau \in (d/(\rho+d),1]$ there exists $A_\tau=A_\tau(d,\theta) > 0$ such that, writing $\bar{\tau}:=\frac{1}{2}( \tau + \frac{d}{\rho+d})$, by~\eqref{Eq:Conc}, H\"older's inequality and {\bf (A.4)}, we have
\begin{align}
\label{Eq:A42}
	\int_\mathcal{S}& \bar{f}(x_0)^\tau \, d \mathrm{Vol}^{d-1}(x_0) = \int_\mathcal{S} \frac{1}{2 \epsilon_0(x_0)} \int_{-\epsilon_0(x_0)}^{\epsilon_0(x_0)} \bar{f}(x_0)^\tau \,dt \,d \mathrm{Vol}^{d-1}(x_0) \nonumber \\
	& \leq 2^{\tau-1} \int_\mathcal{S} \int_{-\epsilon_0(x_0)}^{\epsilon_0(x_0)} \max \biggl\{1,\frac{2d}{\epsilon_0\log 2}, 4 \ell(2\bar{f}(x_0^t)/3) \biggr\} \bar{f}(x_0^t)^\tau \,dt \,d \mathrm{Vol}^{d-1}(x_0)\nonumber \\
	& \leq 2^{\tau} \int_{\mathcal{S}^{\epsilon_0}} \max \biggl\{1,\frac{2d}{\epsilon_0\log 2}, 4 \ell(2\bar{f}(x)/3) \biggr\}\bar{f}(x)^\tau \,dx \leq A_\tau \int_{\mathbb{R}^d} \bar{f}(x)^{\bar{\tau}} \,dx\nonumber\\
	& \leq A_\tau (1+M_0)^{\bar{\tau}} \biggl\{ \int_{\mathbb{R}^d} (1+\|x\|^\rho)^{-\frac{\bar{\tau}}{1-\bar{\tau}}} \,dx \biggr\}^{1-\bar{\tau}} =:A_\tau' < \infty.
\end{align}
Now, by Assumption~\textbf{(A.3)}, for any $P \in \mathcal{P}_{d,\theta}$, 
\begin{align*}
  B_{1} = \int_{\mathcal{S}} \frac{\bar{f}(x_0)}{4\|\dot\eta(x_0)\|} \, d\mathrm{Vol}^{d-1}(x_0) &\leq \frac{1}{4\epsilon_0 M_0} \int_{\mathcal{S}}\bar{f}(x_0) \, d\mathrm{Vol}^{d-1}(x_0) \leq \frac{A_1'}{4 \epsilon_0 M_0}.
\end{align*}
Moreover, writing $\bar{\tau}:=\frac{1}{2}(1+ \frac{d}{\rho+d})$,
\begin{align*}
\sup_{P \in \mathcal{P}_{d,\theta}} \sup_{k \in K_\beta} &(B_1 - B_{1,n}) = \sup_{P \in \mathcal{P}_{d,\theta}} \sup_{k \in K_\beta}\int_{\mathcal{S} \setminus \mathcal{R}_{n}} \frac{\bar{f}(x_0)}{4\|\dot\eta(x_0)\|} \, d\mathrm{Vol}^{d-1}(x_0) 
  \\ & \leq \sup_{P \in \mathcal{P}_{d,\theta}} \sup_{k \in K_\beta} \frac{1}{4\epsilon_0M_0} \int_{\mathcal{S} \setminus \mathcal{R}_{n}} \bar{f}(x_0) \, d\mathrm{Vol}^{d-1}(x_0) \\ 
	& \leq \sup_{P \in \mathcal{P}_{d,\theta}} \sup_{k \in K_\beta} \frac{\delta_n^{1-\bar{\tau}}}{4 \epsilon_0 M_0} \int_{\mathcal{S} \setminus \mathcal{R}_{n}} \bar{f}(x_0)^{\bar{\tau}} \,d \mathrm{Vol}^{d-1}(x_0) \\
	&\leq \frac{\ell^{d(1-\bar{\tau})}\bigl(1/(n-1)\bigr)\log^{d(1-\bar{\tau})}(n-1)}{4\epsilon_0M_0(n-1)^{\beta(1-\bar{\tau})}} A_{\bar{\tau}}'  \to 0.
\end{align*}
By Assumptions~\textbf{(A.2)},~\textbf{(A.3)},~\eqref{Eq:A42} and the fact that $\rho/(\rho+d) > 4/d$, we have, writing $\bar{\tau}:=\frac{1}{2}( 1- 4/d + \frac{d}{\rho+d})$, that 
\begin{align*}
\sup_{P \in \mathcal{P}_{d,\theta}} B_2 &= \sup_{P \in \mathcal{P}_{d,\theta}} \int_{\mathcal{S}} \frac{\bar{f}(x_0)^{1-4/d}}{\|\dot\eta(x_0)\|}a(x_0)^2 \, d\mathrm{Vol}^{d-1}(x_0) \\
	& \leq \sup_{P \in \mathcal{P}_{d,\theta}} \sup_{x_0 \in \mathcal{S}} \biggl\{\frac{a(x_0)^2 \bar{f}(x_0)^{\frac{\rho/(\rho+d)-4/d}{2}}}{\|\dot{\eta}(x_0)\| }\biggr\} \int_{\mathcal{S}} \bar{f}(x_0)^{\bar{\tau}} \, d\mathrm{Vol}^{d-1}(x_0) \\
	& \leq \sup_{\delta \in (0,M_0]} \frac{M_0\delta^{\frac{\rho/(\rho+d)-4/d}{2}}\bigl\{\ell(\delta) + 1/2\bigr\}^2}{(d+2)^2a_d^{4/d}\epsilon_0} A_{\bar{\tau}}'  < \infty.
\end{align*} 
Similarly,
\begin{align*}
  \sup_{P \in \mathcal{P}_{d,\theta}} &\sup_{k \in K_\beta} (B_2-B_{2,n}) = \sup_{P \in \mathcal{P}_{d,\theta}} \sup_{k \in K_\beta} \int_{\mathcal{S} \setminus \mathcal{R}_{n}} \frac{\bar{f}(x_0)^{1-4/d}}{\|\dot\eta(x_0)\|}a(x_0)^2 \, d\mathrm{Vol}^{d-1}(x_0) \\
	& \leq \sup_{k \in K_\beta} \sup_{\delta \in (0,\delta_n]}  \frac{M_0\delta^{\frac{\rho/(\rho+d)-4/d}{2}}\bigl\{\ell(\delta) + 1/2\bigr\}^2}{(d+2)^2a_d^{4/d}\epsilon_0} A_{\bar{\tau}}'   \to 0.
\end{align*}
A similar argument shows that $\gamma_n(k)=O\bigl(1/k + (k/n)^{4/d}\bigr)$, uniformly for $P \in \mathcal{P}_{d,\theta}$ and $k \in K_{\beta}$.

Finally, we bound $P_X\bigl((\partial\mathcal{R}_{n})^{\epsilon_n} \cap \mathcal{S}^{\epsilon_n} \bigr)$ and $R_{\mathcal{R}_{n}^c}(\hat{C}_n^{k\mathrm{nn}}) - R_{\mathcal{R}_{n}^c}(C^{\mathrm{Bayes}})$.  Suppose that $x \in(\partial\mathcal{R}_{n})^{\epsilon_n} \cap \mathcal{S}^{\epsilon_n}$.  Then there exists $z \in \partial \mathcal{R}_n \cap B_{\epsilon_n}(x) \cap \mathcal{S}^{2 \epsilon_n} $ with $\bar{f}(z) =\delta_n$. By Assumption~\textbf{(A.2)} we have that
\begin{equation}
\label{eq:nearboundary}
	\Bigl| \frac{\bar{f}(x)}{\bar{f}(z)} -1 \Bigr| \leq \ell\bigl(\bar{f}(z)\bigr) \|x-z\| + \frac{1}{2} \ell\bigl(\bar{f}(z)\bigr) \|x-z\|^2 \leq \frac{1+\epsilon_n/2}{\beta^{1/2} \log^{1/2}(n-1)}. 
\end{equation}
Thus there exists $n_1 \in \mathbb{N}$ such that $(\partial\mathcal{R}_{n})^{\epsilon_n} \cap \mathcal{S}^{\epsilon_n} \subseteq \{x \in \mathbb{R}^d : \bar{f}(x) \leq 2\delta_n \}$ for $n \geq n_1$.   By the moment assumption in \textbf{(A.4)} and H\"older's inequality, observe that for any $\alpha \in (0,1)$, $P \in \mathcal{P}_{d,\theta}$, $n \geq n_1$ and $\epsilon > 0$,
\begin{align}
\label{eq:tailbound1}
P_X\bigl(&(\partial\mathcal{R}_{n})^{\epsilon_n} \cap \mathcal{S}^{\epsilon_n} \bigr) \leq \mathbb{P}\{ \bar{f}(X) \leq 2 \delta_n \}  \nonumber
\\ & \hspace{30 pt} \leq  (2\delta_n)^\frac{\rho(1-\alpha)}{\rho+d} \int_{x: \bar{f}(x) \leq 2 \delta_n}  \bar{f}(x)^{1-\frac{\rho(1-\alpha)}{\rho+d}} \, dx \nonumber
\\ & \hspace{30pt} \leq  (2\delta_n)^\frac{\rho(1-\alpha)}{\rho+d} \Bigl\{\int_{\mathbb{R}^d} (1+\|x\|^\rho)\bar{f}(x) \, dx\Bigr\}^{1- \frac{\rho(1-\alpha)}{\rho+d}}  \nonumber
  \\ & \hspace{150 pt} \biggl\{\int_{\mathbb{R}^d} \frac{1}{(1+\|x\|^\rho)^\frac{d+\rho \alpha}{\rho(1-\alpha)} } \, dx \biggr\}^\frac{\rho(1-\alpha)}{\rho+d} \nonumber \\
  & \hspace{30pt} \leq  (2\delta_n)^\frac{\rho(1-\alpha)}{\rho+d}(1+M_0)^{1- \frac{\rho(1-\alpha)}{\rho+d}}\biggl\{\int_{\mathbb{R}^d} \frac{1}{(1+\|x\|^\rho)^\frac{d+\rho \alpha}{\rho(1-\alpha)} } \, dx \biggr\}^\frac{\rho(1-\alpha)}{\rho+d} \nonumber \\
&\hspace{30pt}= o\biggl(\Bigl(\frac{k}{n}\Bigr)^{\frac{\rho(1-\alpha)}{\rho+d} - \epsilon}\biggr)
\end{align}
uniformly for $k \in K_{\beta}$.  Moreover,
\[
R_{\mathcal{R}_{n}^c}(\hat{C}_n^{k\mathrm{nn}}) - R_{\mathcal{R}_{n}^c}(C^{\mathrm{Bayes}}) \leq P_X(\mathcal{R}_{n}^c) \leq \mathbb{P}\{ \bar{f}(X) \leq 2 \delta_n \},
\]
so the same bound~\eqref{eq:tailbound1} applies.  Since $\rho/(\rho+d) > 4/d$ and $\alpha \in (0,1)$ was arbitrary, this completes the proof of part~(i).

\medskip

For part (ii), in contrast to part (i), the dominant contribution to the excess risk could now arise from the tail of the distribution. First, as in part~(i), we have $B_{1,n} \to B_1 \leq  A_1'/(4 \epsilon_0 M_0)$, uniformly for $P \in \mathcal{P}_{d,\theta}$ and $k \in K_{\beta}$.  Furthermore, using Assumption~\textbf{(A.3)},~\eqref{Eq:A42} and the fact that $4/d \geq \rho/(\rho+d)$, we see that, for any $\epsilon' \in (0,\rho/(\rho+d)]$,
\begin{align*} 
\label{eq:b2n}
B_{2,n} \Bigl(\frac{k}{n}\Bigr)^{4/d}& \leq \delta_n^{\rho/(\rho+d)-\epsilon'} \int_{\mathcal{S}_n} \frac{\delta_n^{4/d - \rho/(\rho+d)} \bar{f}(x_0)^{1-4/d+\epsilon'}}{c_n^{4}\log^4((n-1)/k) \|\dot\eta(x_0)\|}a(x_0)^2 \, d\mathrm{Vol}^{d-1}(x_0)  \nonumber \\
	& \leq \sup_{x_0 \in \mathcal{S}_n} a(x_0)^2\frac{\delta_n^{\rho/(\rho+d)-\epsilon'}A_{d/(\rho+d)+\epsilon'}'}{\epsilon_0M_0 c_n^{4}\log^4((n-1)/k)} = o\bigl((k/n)^{\rho/(\rho+d) - \epsilon}\bigr),
\end{align*}
for every $\epsilon \in \bigl(\epsilon',\rho/(\rho+d)\bigr]$, uniformly for $P \in \mathcal{P}_{d,\theta}$ and $k \in K_{\beta}$, where the final conclusion follows from the fact that $\sup_{P \in \mathcal{P}_{d,\theta}} \sup_{x_0 \in \mathcal{S}_n} a^2(x_0) / c_n^2$ is bounded.  We can also bound $\gamma_n(k)$ by the same argument, so the result follows in the same way as in part~(i).
\end{proof}

\subsection{Proofs of results from Section~\ref{sec:semisup}}
\begin{proof}[Proof of Theorem~\ref{thm:exriskRdOadapt}]
Recall that
\[
k_{\mathrm{O}}(x) = \max\bigl[ \lceil (n-1)^\beta \rceil, \min\bigl\{\bigl\lfloor B\bigl\{\bar{f}(x)(n-1)\bigr\}^{4/(d+4)} \bigr\rfloor, \lfloor (n-1)^{1-\beta}\rfloor \bigr\}\bigr],
\]
and define
\[
	\delta_{n,\mathrm{O}}(x) := \frac{k_\mathrm{O}(x)}{n-1} c_n^d \log^d \Bigl( \frac{n-1}{k_\mathrm{O}(x)} \Bigr),
\]
where $c_n:=\sup_{x_0 \in \mathcal{S}:\bar{f}(x_0) \geq k_\mathrm{O}(x_0)/(n-1)} \ell\bigl(\bar{f}(x_0)\bigr)$. For $\alpha \in ((1+d/4)\beta,1)$ let
\[
	\mathcal{R}_n = \{x \in \mathbb{R}^d : \bar{f}(x) > (n-1)^{-(1-\alpha)} \} \cap \mathcal{X}_{\bar{f}}.
\]
Then there exists $n_0 \in \mathbb{N}$ such that for $n \geq n_0$ we have $\mathcal{R}_n \subseteq \bigl\{x \in \mathbb{R}^{d} :\bar{f}(x) \geq \delta_{n, \mathrm{O}}(x)\bigr\}$ for all $P \in \mathcal{P}_{d,\theta}$ and $B \in [B_*,B^*]$, and by Assumption \textbf{(A.1)} and~\eqref{Eq:fbarinfty}, we then have that $\mathcal{R}_n$ is a $d$-dimensional manifold.  There exists $n_1 \in \mathbb{N}$ such that for all $n \geq n_1$, $P \in \mathcal{P}_{d,\theta}$, $B \in [B_*,B^*]$ and $x \in \mathcal{R}_n \cap \mathcal{S}^{\epsilon_0}$ we have that $k_\mathrm{O}(x) =\bigl\lfloor B\bigl\{\bar{f}(x)(n-1)\bigr\}^{4/(d+4)} \bigr\rfloor$. By \textbf{(A.2)}, we therefore have that $k_\mathrm{O} \in K_{\beta,\tau}$ for some $\tau = \tau_n$ (which does not depend on $P \in \mathcal{P}_{d,\theta}$ or $B \in [B_*,B^*]$) with $\tau_n \searrow 0$.  

By a similar argument to that in~\eqref{eq:nearboundary}, there exists $n_2 \in \mathbb{N}$ such that for $n \geq n_2$, $P \in \mathcal{P}_{d,\theta}$, $B \in [B_*,B^*]$ and $x \in (\partial\mathcal{R}_{n})^{\epsilon_n} \cap \mathcal{S}^{\epsilon_n}$, we have $\bar{f}(x) \leq 2(n-1)^{-(1-\alpha)}$.  But, by Markov's inequality and H\"{o}lder's inequality, for $\tilde{\alpha}\in(0,1)$ and any $P \in \mathcal{P}_{d,\theta}$,
\begin{align}
\label{eq:tailbound}
\mathbb{P}&\{\bar{f}(X) \leq 2(n-1)^{-(1-\alpha)}\}  \nonumber
\\ &\leq \{2(n-1)^{-(1-\alpha)}\}^{\frac{\rho(1-\tilde{\alpha})}{\rho+d}} \int_{\mathbb{R}^d} \bar{f}(x)^{1-\frac{\rho(1-\tilde{\alpha})}{\rho+d}} \, dx \nonumber
\\ & \leq \{2(n-1)^{-(1-\alpha)}\}^{\frac{\rho(1-\tilde{\alpha})}{\rho+d}}(1+M_0)^{1-\frac{\rho(1-\tilde{\alpha})}{\rho+d}} \nonumber
\\ & \hspace{120 pt} \Bigl\{\int_{\mathbb{R}^d} \frac{1}{(1+\|x\|^\rho)^{(\rho+d)/\{\rho(1-\tilde{\alpha})\}-1}} \, dx\Bigr\}^{\frac{\rho(1-\tilde{\alpha})}{\rho+d}}.
\end{align}
Thus, if $\rho>4$, then we can choose $\alpha \in ((1+d/4)\beta, d(\rho-4)/\{\rho(d+4)\})$ and $\tilde{\alpha}<1-4(\rho+d)/\{\rho(1-\alpha)(d+4)\}$ in~\eqref{eq:tailbound} to conclude that 
\[
\sup_{P \in \mathcal{P}_{d,\theta}} P_X(\mathcal{R}_n^c) \leq \sup_{P \in \mathcal{P}_{d,\theta}} \mathbb{P}\{\bar{f}(X) \leq 2(n-1)^{-(1-\alpha)}\} = o(n^{-4/(d+4)}).
\]
Moreover, writing
\[
	B_{3,n} := \int_{\mathcal{S}_n} \frac{\bar{f}(x_0)^{d/(d+4)}}{\|\dot{\eta}(x_0)\|} \Bigl\{ \frac{1}{4B} +B^{4/d} a(x_0)^2 \Bigr\}\, d\mathrm{Vol}^{d-1}(x_0),
      \]
by very similar arguments to those given in the proof of Theorem~\ref{thm:exriskRdstandard}, $B_{3,n} \rightarrow B_3$ and $\gamma_n(k_\mathrm{O})=O(n^{-4/(d+4)})$ as $n \rightarrow \infty$, both uniformly for $P \in \mathcal{P}_{d,\theta}$ and $B \in [B_*,B^*]$. The proof of part~(i) therefore follows from Theorem~\ref{thm:exriskRtn}.

On the other hand, if $\rho \leq 4$, then choosing both $\tilde{\alpha} > 0$ and $\alpha> (1+d/4)\beta$ to be sufficiently small, we find from~\eqref{eq:tailbound} that 
\[
B_{3,n} n^{-4/(d+4)} + \gamma_n(k_\mathrm{O}) + P_X\bigl((\partial\mathcal{R}_{n})^{\epsilon_n} \cap \mathcal{S}^{\epsilon_n} \bigr) + P_X(\mathcal{R}_n^c)\! =\! o\Bigl(n^{-\frac{\rho}{\rho+d} + \beta + \epsilon}\Bigr),
\]
for every $\epsilon>0$, uniformly for $P \in \mathcal{P}_{d,\theta}$ and $B \in [B_*,B^*]$. After another application of Theorem~\ref{thm:exriskRtn}, this proves part~(ii).
\end{proof}

\begin{proof}[Proof of Theorem~\ref{thm:exriskRdSSadapt}]
We prove parts (i) and (ii) of the theorem simultaneously, by appealing to the corresponding arguments in the proof of Theorem~\ref{thm:exriskRdOadapt}.  First, as in the proof of Theorem~\ref{thm:exriskRdOadapt}, for $\alpha \in \bigl((1+d/4)\beta,1\bigr)$, we define $\mathcal{R}_n = \{x \in \mathbb{R}^d : \bar{f}(x) > (n-1)^{-(1-\alpha)} \} \cap \mathcal{X}_{\bar{f}}$ and introduce the following class of functions: for $\tau > 0$, let 
\[
\mathcal{F}_{n,\tau} := \biggl\{ \tilde{f}: \mathbb{R}^{d} \to \mathbb{R}: \tilde{f} \  \mathrm{continuous}, \sup_{x \in \mathcal{R}_{n}}\biggl|\frac{\bar{f}(x)}{\tilde{f}(x)} - 1\biggr| \leq \tau \biggr\}.
\]
Let $\tau = \tau_n:= 2(n-1)^{-\alpha/2}$.  We first show that $\hat{f}_{m} \in \mathcal{F}_{n,\tau}$ with high probability.  For $x \in \mathcal{R}_n$,
\[
\Bigl| \frac{\hat{f}_{m}(x)}{\bar{f}(x)}  - 1 \Bigr|  \leq (n-1)^{1-\alpha} |\hat{f}_{m}(x) - \bar{f}(x) |   \leq  (n-1)^{1-\alpha}  \|\hat{f}_{m} - \bar{f} \|_{\infty}.
\]
Now
\begin{equation}
\label{eq:inftybiasvar}
\|\hat{f}_{m} - \bar{f}\|_{\infty} \leq  \|\hat{f}_{m} - \mathbb{E} \hat{f}_{m} \|_{\infty} +  \|\mathbb{E} \hat{f}_{m} - \bar{f} \|_{\infty}.
\end{equation} 
To bound the first term in~\eqref{eq:inftybiasvar}, by \citet[Corollary~2.2]{Gine:2002}, there exist $C,L > 0$, such that
\begin{equation}
\label{eq:tailboundm}
\sup_{P \in \mathcal{P}_{d,\theta} \cap \mathcal{Q}_{d,\gamma,\lambda}} \mathbb{P}\biggl(\| \hat{f}_{m} - \mathbb{E}  \hat{f}_{m}\|_{\infty} \geq \frac{s}{m^{\gamma/(d+2 \gamma)}}\biggr)   \leq   L\biggl(\frac{4L}{4L+C}\biggr)^{\frac{A^{d}s^{2}}{LC \lambda R(K)}},
\end{equation}
for all  $s \in \Bigl[\frac{C\|\bar{f}\|_{\infty}^{1/2} R(K)^{1/2}}{A^{d/2}} \log^{1/2 }\Bigl(\frac{\|K\|_{\infty}m^{d/(2(d+2 \gamma))}}{\|\bar{f}\|_{\infty}^{1/2} A^{d/2} R(K)^{1/2}}\Bigr), \frac{ C \|\bar{f}\|_{\infty} R(K) m^{\gamma/(d+2 \gamma)} }{\|K\|_{\infty}}\Bigr]$ and $A \in [A_*,A^*]$.

Recall that for $P \in \mathcal{P}_{d,\theta}$, we have $\|\bar{f}\|_\infty \leq \lambda$ and $\|\bar{f}\|_\infty$ also satisfies the lower bound in~\eqref{Eq:fbarinfty}.  Hence, by applying the bound in~\eqref{eq:tailboundm} with $s = s_0 := m^{\gamma/(d+2\gamma)}/(n-1)^{1-\alpha/2}$, since $m \geq m_0 (n-1)^{d/\gamma+2}$, we have that there exists $n_* \in \mathbb{N}$, not depending on $P \in \mathcal{P}_{d,\theta}$ or $A \in [A_*,A^*]$ such that for $n \geq n_*$,
\begin{align*}
\sup_{P \in \mathcal{P}_{d,\theta} \cap \mathcal{Q}_{d,\gamma,\lambda}} \mathbb{P}  \biggl\{\| \hat{f}_{m} - &\mathbb{E}  \hat{f}_{m}\|_{\infty}  \geq  \frac{1} {(n-1)^{1-\alpha/2}} \biggr\} 
\\ &= \sup_{P \in \mathcal{P}_{d,\theta} \cap \mathcal{Q}_{d,\gamma,\lambda}} \mathbb{P}  \Bigl\{\| \hat{f}_{m} - \mathbb{E}  \hat{f}_{m}\|_{\infty} \geq  s_0 m^{-\gamma/(d+2 \gamma)} \Bigr \}  
  \\ &  \leq  L\biggl(\frac{4L}{4L+C}\biggr)^{\frac{A^{d}(n-1)^\alpha m_0^{2\gamma/(d+2 \gamma)}}{LC \lambda R(K)}}    = O(n^{-M}),
\end{align*}
for all $M>0$, uniformly for $A \in [A_*,A^*]$. For the second term in~\eqref{eq:inftybiasvar}, by a Taylor expansion, we have that for all $P \in \mathcal{P}_{d,\theta} \cap \mathcal{Q}_{d,\gamma,\lambda}$ and $A \in [A_*,A^*]$,
\begin{align*}
\|\mathbb{E} \hat{f}_{m} - \bar{f} \|_{\infty} &\leq \lambda A^\gamma m^{-\gamma/(d+2 \gamma)} \int_{\mathbb{R}^d} \|z\|^\gamma |K(z)| \,dz 
\\& \leq \frac{\lambda A^\gamma m_0^{-\gamma/(d+2 \gamma)} }{n-1}\int_{\mathbb{R}^d} \|z\|^\gamma |K(z)| \,dz.
\end{align*} 
It follows that, writing $\tau_0 := 2(n-1)^{-\alpha/2}$, we have
\[
  \sup_{P \in \mathcal{P}_{d,\theta} \cap \mathcal{Q}_{d,\gamma,\lambda}} \sup_{A \in [A_*,A^*]} \mathbb{P} (\hat{f}_{m} \notin \mathcal{F}_{n,\tau_0}) = O(n^{-M})
\]
for all $M>0$.

Now, for $\tilde{f} \in \mathcal{F}_{n,\tau_0}$, let 
\[
k_{\tilde{f}}(x) : = \max\Bigl[\lceil (n-1)^\beta \rceil, \min\bigl\{\lfloor B\{\tilde{f}(x)(n-1)\}^{4/(d+4)} \rfloor, \lfloor (n-1)^{1-\beta}\rfloor\bigr\}\Bigr].
\]
Let $c_n := \sup_{x_0 \in \mathcal{S}:\bar{f}(x_0) \geq k_{\tilde{f}}(x_0)/(n-1)} \ell\bigl(\bar{f}(x_0)\bigr)$, and let
\[
	\delta_{n, \tilde{f}}(x):= \frac{k_{\tilde{f}}(x)}{n-1} c_n^d \log^d \Bigl( \frac{n-1}{k_{\tilde{f}}(x)} \Bigr).
\]
Then there exists $n_0 \in \mathbb{N}$ such that for $n \geq n_0$ and $\tilde{f} \in \mathcal{F}_{n,\tau_0}$, we have $\mathcal{R}_n \subseteq \bigl\{x \in \mathbb{R}^{d} :\bar{f}(x) \geq \delta_{n, \tilde{f}}(x)\bigr\}$ and $k_{\tilde{f}} \in K_{\beta,\tau_0}$.  We can therefore apply Theorem~\ref{thm:exriskRtn} (similarly to the application in the proof of Theorem~\ref{thm:exriskRdOadapt}) to conclude that for every $\epsilon > 0$,
\begin{align*}
\label{eq:exriskG} 
R(\hat{C}_n^{k_{\tilde{f}}\mathrm{nn}}) - R(C^{\mathrm{Bayes}}) = &B_{3,n}n^{-4/(d+4)} + o\Bigl(n^{-4/(d+4)} + n^{-\frac{\rho}{\rho+d}+\beta+\epsilon}\Bigr)\end{align*}
uniformly for $P \in \mathcal{P}_{d,\theta} \cap \mathcal{Q}_{d,\gamma,\lambda}$ and $\tilde{f} \in \mathcal{F}_{n,\tau_0}$, where $B_{3,n}$ was defined in the proof of Theorem~\ref{thm:exriskRdOadapt}.  The proof of both parts~(i) and~(ii) is now completed by following the relevant steps in the proof of Theorem~\ref{thm:exriskRdOadapt}.
\end{proof}

\section*{Acknowledgements}

The authors are grateful to the anonymous reviewers, whose constructive comments helped to improve the paper.  We would also like to thank the Isaac Newton Institute for Mathematical Sciences for support and hospitality during the programme `Statistical Scalability' when work on this paper was undertaken. This work was supported by EPSRC grant number EP/R014604/1.

\begin{appendix}

\section{The relationship between our classes and the margin assumption}
\label{Sec:Margin}
Recall from \citet{Mammen:99} that a distribution $P$ on $\mathbb{R}^d \times \{0,1\}$ with marginal $P_X$ on $\mathbb{R}^d$ and regression function $\eta$ satisfies a \emph{margin assumption} with parameter $\alpha > 0$ if there exists $C > 0$ such that
\[
  P_X\bigl(\{x:|\eta(x) - 1/2| \leq s\}\bigr) \leq Cs^\alpha
\]
for all sufficiently small $s > 0$.  The following lemma clarifies the relationship between our classes and the margin assumption.
\begin{lemma}
\label{Lemma:Margin}
Let $P \in \mathcal{P}_{d,\theta}$ for some $\theta = (\epsilon_0,M_0,\rho,\ell,g) \in \Theta$.  Then $P$ satisfies a margin assumption with parameter $\alpha = 1$.
\end{lemma}
\begin{proof}
By the final part of \textbf{(A.3)}, we have
\begin{align}
  \label{Eq:Decomposition}
    P_X\bigl(\{x:|\eta(x) - 1/2| \leq s\}\bigr) \leq P_X\bigl(\{x:|\eta(x) &- 1/2| \leq s\} \, \cap \, \mathcal{S}^{\epsilon_0}\bigr) \nonumber \\
    &+ P_X\bigl(\{x:\ell(\bar{f}(x)) \geq 1/s\}\bigr).
  \end{align}
  Now, by Proposition~\ref{Prop:DG1} in Section~\ref{Sec:Tubular}, for $x \in \mathcal{S}^{\epsilon_0}$, there exists $x_0 \in \mathcal{S}$ and $t \in (-\epsilon_0,\epsilon_0)$ such that $x = x_0 + t\dot{\eta}(x_0)/\|\dot{\eta}(x_0)\|$.  Thus, by a Taylor expansion,
  \[
    |\eta(x) - 1/2| \geq |t|\epsilon_0M_0 - \frac{1}{2}M_0t^2 \geq \frac{1}{2}|t|\epsilon_0M_0.
  \]
  We deduce as in Step 5 of the proof of Theorem~\ref{thm:exriskRtn} that there exists $s_0 = s_0(d,\theta) > 0$ such that for all $s \in (0,s_0]$,
  \begin{align}
    \label{Eq:Firstetaterm}
    P_X\bigl(\{x:|\eta(x) - 1/2| \leq s\} \, \cap \, &\mathcal{S}^{\epsilon_0}\bigr) \leq P_X\bigl(\mathcal{S}^{\frac{2s}{\epsilon_0M_0}}\bigr) \nonumber \\
    &\leq \frac{8s}{\epsilon_0M_0} \int_{\mathcal{S}} \bar{f}(x_0) \, d\mathrm{Vol}^{d-1}(x_0) \leq  \frac{8sA_1'}{\epsilon_0M_0},
  \end{align}
where the final bound follows from~\eqref{Eq:A42} in the main text.  For the second term in~\eqref{Eq:Decomposition}, we exploit the fact that since $\ell \in \mathcal{L}$, there exists $A = A(d,\theta) > 0$ such that $\ell(\delta) \leq A\delta^{-\frac{\rho}{2(\rho+d)}}$ for all $\delta > 0$.  Hence, arguing as in~\eqref{eq:tailbound1} in the main text, we find that
  \begin{align}
    \label{Eq:Secondetaterm}
    P_X\bigl(\{x:\ell(\bar{f}(x)) \geq 1/s\}\bigr) &\leq P_X\Bigl(\Bigl\{x:\bar{f}(x) \leq (As)^{\frac{2(\rho+d)}{\rho}}\Bigr\}\Bigr) \nonumber \\
    & \hspace{-3pt} \leq As(1+M_0)^{\frac{\rho+2d}{2(\rho+d)}}\biggl\{\int_{\mathbb{R}^d} \frac{1}{(1+\|x\|^\rho)^{\frac{\rho+2d}{\rho}}} \, dx\biggr\}^{\frac{\rho}{2(\rho+d)}}.
  \end{align}
 The result follows from~\eqref{Eq:Decomposition},~\eqref{Eq:Firstetaterm} and~\eqref{Eq:Secondetaterm}.
\end{proof}

\section{Example~1 from the main text}
\label{Sec:Example}
Recall that we consider the distribution $P$ on $\mathbb{R}^d \times \{0,1\}$ for which $\bar{f}(x)=\frac{\Gamma(3+d/2)}{2 \pi^{d/2}}(1-\|x\|^2)^2 \mathbbm{1}_{\{x \in B_1(0)\}}$ and $\eta(x) = \min(\|x\|^2,1)$. Since $\bar{f}$ is continuous on all of $\mathbb{R}^d$, it is clear that \textbf{(A.1)} is satisfied.

Now, $\mathcal{S}=\{x \in \mathbb{R}^d : \|x\|=2^{-1/2}\}$ and clearly $\mathcal{S} \cap \{x \in \mathbb{R}^d : \bar{f}(x) > 0\}$ is non-empty. For all $x_0 \in \mathcal{S}$ we have that $\bar{f}(x_0) = \frac{\Gamma(3+d/2)}{8 \pi^{d/2}} \leq M_0$. Since $\epsilon_0 \leq 1/10$ we have that $\mathcal{S}^{\epsilon_0} \subseteq B_{9/10}(0) \setminus B_{3/5}(0)$ and thus $\bar{f}$ is twice continuously differentiable on $\mathcal{S}^{\epsilon_0}$. Differentiating $\bar{f}$ twice on $B_1(0)$, we have that $\dot{\bar{f}}(x)=-2 \pi^{-d/2} \Gamma(3+d/2)(1-\|x\|^2)x$ and
\[
	\ddot{\bar{f}}(x)= 2 \pi^{-d/2} \Gamma(3+d/2) \{ 2xx^T -(1-\|x\|^2)I\}.
\]
Thus, for $x_0 \in \mathcal{S}$, we have $\|\dot{\bar{f}}(x_0)\|/\bar{f}(x_0) = 2^{5/2}\leq \ell(\bar{f}(x_0))$. We also have that, for any $x \in B_1(0)$,
\[
	\|\ddot{\bar{f}}(x)\|_\mathrm{op} = 2 \pi^{-d/2} \Gamma(3+d/2) \|2xx^T - (1-\|x\|^2)I\|_\mathrm{op} \leq \frac{6\Gamma(3+d/2)}{\pi^{d/2}},
\]
so that $\sup_{u \in B_{\epsilon_0}(0)} \|\ddot{\bar{f}}(x_0+u)\|_\mathrm{op} / \bar{f}(x_0) < 48 \leq \ell(\bar{f}(x_0))$ for any $x_0 \in \mathcal{S}$. Finally for \textbf{(A.2)} we consider the cases $x \in B_1(0) \setminus B_{\epsilon_0}(0)$ and $x \in B_{\epsilon_0}(0)$ separately. If $x \in B_1(0) \setminus B_{\epsilon_0}(0)$ then, for $r \in (0,\epsilon_0]$, at least a proportion $2^{-d}$ of the ball $B_{r}(x)$ is closer to the origin than $x$, and thus has larger density. This gives us that, for such $x$ and $r$, $p_r(x) \geq 2^{-d}a_d r^d \bar{f}(x) \geq \epsilon_0 a_d r^d \bar{f}(x)$. When $x \in B_{\epsilon_0}(0)$ and $r \in (0,\epsilon_0]$ we instead have that
\[
	p_r(x) \geq a_d r^d \frac{\Gamma(3+d/2)}{2 \pi^{d/2}} (1-4 \epsilon_0^2)^2 \geq a_d (1-4 \epsilon_0^2)^2 r^d \bar{f}(x) \geq \epsilon_0 a_d r^d \bar{f}(x).
\]

We now turn to condition \textbf{(A.3)}. First, for any $x_0 \in \mathcal{S}$ we have that $\|\dot{\eta}(x_0)\|=\|2 x_0\| = 2^{1/2} \geq \epsilon_0M_0$. For $x \in \mathcal{S}^{2 \epsilon_0}$ we have that $\| \dot{\eta}(x)\| \leq 2(2^{-1/2}+2 \epsilon_0) \leq M_0$ and $\| \ddot{\eta}(x)\|_\mathrm{op} = \|2 I\|_\mathrm{op}=2\leq M_0$. Since $\ddot{\eta}$ is constant on $\mathcal{S}^{2 \epsilon_0}$ it is trivially true that
\[
	\sup_{x,z \in \mathcal{S}^{2\epsilon_0}:\|z-x\| \leq g(\epsilon)}  \|\ddot{\eta}(z) - \ddot{\eta}(x)\|_{\mathrm{op}} \leq \epsilon
\]
for any $g \in \mathcal{G}$. Now for $x \in \mathbb{R}^d \setminus \mathcal{S}^{\epsilon_0}$ we have that
\[
	|\eta(x)-1/2| \geq 2^{1/2} \epsilon_0 - \epsilon_0^2 \geq \epsilon_0 \geq 1/\ell(\bar{f}(x)).
\]

Since the support of $\bar{f}$ is equal to $B_1(0)$, we have that $\int_{\mathbb{R}^d} \|x\|^\rho dP_X(x) \leq 1 \leq M_0$, so \textbf{(A.4)} is satisfied.

We finally check \textbf{(A.5)} to show that $P \in \mathcal{Q}_{d,2,\lambda}$ for $\lambda \geq 6 \pi^{-d/2} \Gamma(3+d/2)$. First, it is clear that $\|\bar{f}\|_\infty \leq \lambda$. Now, for any $x,y \in \mathbb{R}^d$ we have that
\[
	\|\dot{\bar{f}}(y)-\dot{\bar{f}}(x)\| \leq \|y-x\| \sup_{z \in B_1(0)} \| \ddot{\bar{f}}(z)\|_\mathrm{op} \leq 6 \pi^{-d/2} \Gamma(3+d/2) \|y-x\|.
\]

\section{Example~2 from the main text}
\label{Sec:Lower}

\begin{proof}[Proof of claim in Example~\ref{Eg:Counter}]
Fix $\epsilon > 0$ and $k \in K_\beta$, let
\[
\mathcal{T}_{n} := (0,1/2) \times \bigl((1+\epsilon)\log(n/k), \infty\bigr),
\]
and for $\gamma > 0$, let
\[
B_{k, \gamma} := \bigcap_{x = (x_1,x_2) \in \mathcal{T}_n} \{ \gamma < \|X_{(k+1)}(x) - x\| < x_2-1\}.
\]
Now, for $\epsilon \beta \log{n} > 4$ and $\gamma \in [2,\epsilon \log(n/k)/2)$, 
\[
\mathbb{P}(B_{k, \gamma}^c) \leq \mathbb{P}(T \geq k+1) + \mathbb{P}(T' \leq k),
\]
where $T \sim \mathrm{Bin}(n,p_\gamma^*)$, $T' \sim \mathrm{Bin}(n,p_*)$,
\begin{align*}
p^*_{\gamma} &:= \int_{0}^{1} \int_{(1+\epsilon)\log(n/k)-\gamma}^{\infty} t_1  \exp(-t_2) \,dt_1 dt_2 \leq \frac{1}{2} \Bigl(\frac{k}{n}\Bigr)^{1+\epsilon} e^\gamma \leq \frac{1}{2}\Bigl(\frac{k}{n}\Bigr)^{1+\epsilon/2}, \\
p_* &:= \int_{0}^{1} \int_{3-3^{1/2}}^{3+3^{1/2}} t_1  \exp(-t_2) \,dt_1 dt_2 \geq \frac{1}{8}.
\end{align*}
Therefore, there exists $n_0 \in \mathbb{N}$ such that $np_* - (k+1) \geq k/2$ and $k + 1 - np_\gamma^* \geq k/2$ for all $k \in K_{\beta}$, $\gamma \in [2,\epsilon \log(n/k)/2)$ and $n \geq n_0$.  It follows by Bernstein's inequality that $\sup_{k \in K_\beta} \sup_{\gamma \in [2,\epsilon \log(n/k)/2)} \mathbb{P}(B_{k,\gamma}^c) = O(n^{-M})$ for every $M > 0$.

Now, for $x = (x_1,x_2) \in \mathcal{T}_n$, $\epsilon \beta \log{n} > 4$ and $\gamma \in [2,x_2-1)$, we have that
\begin{align*}
\frac{ \int_{B_\gamma(x)} \eta(t) \bar{f}(t) \,dt}{ \int_{B_\gamma(x)} \bar{f}(t) \,dt} &=  \frac{\int_0^1 \int_{x_2 - \{\gamma^2 - (t_1-x_1)^2\}^{1/2}}^{x_2 + \{\gamma^2 - (t_1-x_1)^2\}^{1/2}} t_1^2e^{-t_2} \, dt_2 \, dt_1}{\int_0^1 \int_{x_2 - \{\gamma^2 - (t_1-x_1)^2\}^{1/2}}^{x_2 + \{\gamma^2 - (t_1-x_1)^2\}^{1/2}} t_1e^{-t_2} \, dt_2 \, dt_1} 
\\ & = \frac{\int_0^1 t_1^2 \sinh( \{ \gamma^2 - (t_1-x_1)^2 \}^{1/2}) \,dt_1}{\int_0^1 t_1 \sinh( \{ \gamma^2 - (t_1-x_1)^2 \}^{1/2})\,dt_1} \\
&\geq \frac{2}{3} \frac{\sinh\bigl( (\gamma^2 -1 )^{1/2}\bigr)}{\sinh(\gamma)} \geq \frac{2}{3} \frac{\sinh(3^{1/2})}{\sinh(2)} > \frac{1}{2}.
\end{align*}
Our next observation is that for $\gamma \in [0,\infty)$ and $x_{(k+1)} \in \mathbb{R}^d$ such that $\|x_{(k+1)}-x\| = \gamma$, we have that $(X_{(1)},Y_{(1)}, \ldots,X_{(k)},Y_{(k)})|(X_{(k+1)} = x_{(k+1)})\! \stackrel{d}{=} (\tilde{X}_{(1)},\tilde{Y}_{(1)},\ldots,\tilde{X}_{(k)},\tilde{Y}_{(k)})$,  where the pairs $(\tilde{X}_1,\tilde{Y}_1),\ldots,(\tilde{X}_k, \tilde{Y}_k)$ are independent and identically distributed, and then $(\tilde{X}_{(1)},\tilde{Y}_{(1)}),\ldots,(\tilde{X}_{(k)},\tilde{Y}_{(k)})$ is a reordering such that $\|\tilde{X}_{(1)} - x\| \leq \ldots \leq \|\tilde{X}_{(k)} - x\|$.  Here $\tilde{X}_1  \stackrel{d}{=} X | (\|X\!-x\| \leq \gamma)$ and $\mathbb{P}(\tilde{Y}_1=1 |\tilde{X}_1 = x) = \eta(x)$. Writing $\tilde{S}_n(x) := \frac{1}{k} \sum_{i=1}^k \mathbbm{1}_{\{\tilde{Y}_i=1\}}$ we therefore have by Hoeffding's inequality that, for $x \in \mathcal{T}_n$, $\epsilon \beta \log{n} > 4$ and $\|x_{(k+1)}-x\| \in [2,x_2-1)$,
\begin{align*}
	\mathbb{P}\{\hat{S}_n(x) < 1/2 \big| X_{(k+1)} &= x_{(k+1)} \} = \!\mathbb{P}\{ \tilde{S}_n(x) < 1/2  \} \\
	& = \mathbb{P}\{ \tilde{S}_n(x) - \mathbb{E} \tilde{S}_n(x) < -( \mathbb{E} \eta(\tilde{X}_1) - 1/2) \} \\
	& \leq \! \exp \biggl( - 2k \Bigl(\frac{2}{3} \frac{\sinh(3^{1/2})}{\sinh(2)}-\frac{1}{2} \Bigr)^2 \biggr) = O(n^{-M})
\end{align*}
for all $M>0$, uniformly for $k \in K_\beta$. Writing $P_{(k+1)}$ for the marginal distribution of $X_{(k+1)}$, we deduce that 
\begin{align*}
\mathbb{P}&\{\hat{S}_n(x) < 1/2 \} 
\\ &\leq \mathbb{P}\{\hat{S}_n(x) < 1/2, \|X_{(k+1)} - x\| \in [2,x_2-1)\} + \mathbb{P}(B_{k,2}^c) 
\\ & = \int_{B_{x_2-1}(x)\setminus B_2(x)} \!\!\!\!\!\!\!\!\!\!\!\! \mathbb{P}\{\hat{S}_n(x) < 1/2 \big| X_{(k+1)}= x_{(k+1)} \} \, dP_{(k+1)}(x_{(k+1)}) + O(n^{-M}) \\
&= O(n^{-M})
\end{align*}
for all $M>0$, uniformly for $k \in K_\beta$.  We conclude that for every $M > 0$,
\begin{align*}
&R_{\mathcal{T}_{n}}(\hat{C}_n^{k\mathrm{nn}}) - R_{\mathcal{T}_{n}}(C^{\mathrm{Bayes}}) \\
&= \int_{\mathcal{T}_{n}} \Bigl[\mathbb{P}\{\hat{S}_n(x) < 1/2\} - \mathbbm{1}_{\{\eta(x) < 1/2\}} \Bigr] \{2\eta(x)-1\}\bar{f}(x) \, dx \\
 &= \int_{(1+\epsilon)\log(n/k)}^{\infty} \int_{0}^{1/2} \mathbb{P}\{\hat{S}_n(x) \geq 1/2\} (1-2x_1) x_1 \exp(-x_2)\, dx_1 \, dx_2 \\
&= \frac{1}{24}\Bigl(\frac{k}{n}\Bigr)^{1+\epsilon} + O(n^{-M}), 
\end{align*}
uniformly for $k \in K_{\beta}$.  The claim~\eqref{Eq:Claim} follows from this together with Theorem~\ref{thm:exriskRdstandard}(ii).
\end{proof}

\section{Proof of Theorem~4}
\begin{proof}[Proof of Theorem~\ref{Thm:LowerBound}]
For an integer $q \geq 3$ and $\nu \geq 0$, define a grid on $\mathbb{R}^{d}$ by
\begin{align*}
  G_{q, \nu} := \Bigl\{\Bigl(&\gamma_{1}, \gamma_{2} + \frac{2\kappa_{2} + 1}{2q}, \gamma_{3} + \frac{2\kappa_{3} + 1}{2q},\ldots, \gamma_d +  \frac{2\kappa_{d} + 1}{2q}\Bigr): \\
  &\gamma_1,\ldots,\gamma_d \in \{1, \dots, \lceil q^{\nu} \rceil \}, \kappa_2,\ldots,\kappa_d \in \{0,1, \ldots, q-1\}\Bigr\}.
\end{align*}
Now, for $x \in \mathbb{R}^{d}$,  let $n_{q}(x)$ be the closest point to $x$ among those in $G_{q, \nu}$ (if there are multiple points, pick the one that is smallest in the lexicographic ordering).  Let $m := \lceil q^\nu \rceil^dq^{d-1}$ and define closed Euclidean balls $\mathcal{X}_{1}, \ldots, \mathcal{X}_m$ in $\mathbb{R}^{d}$ of radius $1/(2q)$, where the $l$th ball is centered at the $l$th grid point in the lexicographic ordering.

Writing $[z]$ for the closest integer to $z$ (where we round half-integers to the nearest even integer), define the `saw-tooth' function $\eta_{0}: \mathbb{R}^{d} \rightarrow [3/8,5/8]$, by $\eta_{0}(x) := 3/8 + \bigl|x_1+1/4-[x_1+1/4]\bigr|/2$, for $x = (x_1,\ldots,x_d)^T$.   Further, for $x \in \mathbb{R}^d$, set $u(x) : = \frac{\alpha_0g^{-1}(1/q)}{q^{2}}\bigl(1/4 - q^2\|x - n_{q}(x)\|^2\bigr)^4$, where $\alpha_0 := 1/27$.

For $\sigma := (\sigma_{1}, \ldots, \sigma_{m})^T \in \{-1,1\}^{m}$, we now define the distribution $P_\sigma$ on $\mathbb{R}^d \times \{0,1\}$ by setting the regression function to be $\eta_{\sigma}(x) := \eta_{0}(x) + \frac{1}{2}\sigma_{l} u(x)$, for $x \in \mathcal{X}_{l}$, $l = 1, \ldots, m$, and setting $\eta_{\sigma}(x) := \eta_{0}(x)$, otherwise.  To define the marginal distribution on $\mathbb{R}^d$ induced by $P_\sigma$, which will be the same for each $\sigma$, we first define the boxes $B_0 := (0, \lceil q^{\nu} \rceil +3/2)^d$ and $B_r := [-r/2+1/4 - a/16,-r/2+1/4+a/16] \times [-a,a]^{d-1}$ for $r =1,\ldots,20$ and some $a > 0$ to be chosen later.  We further define a modified bump function by
\[
  h(x) := \left\{ \begin{array}{ll} 0 & \mbox{if $x \leq 0$} \\
                    \Phi\bigl(\frac{2x-1}{x(1-x)}\bigr) & \mbox{if $x \in (0,1)$} \\
                    1 & \mbox{if $x \geq 1$,} \end{array} \right.
              \]
where $\Phi$ denotes the standard normal distribution function.  For $x \in \mathbb{R}^d$ we then set
              \[
                \bar{f}(x) := w_0 h\bigl(1 - 4\mathrm{dist}(x,B_0)\bigr) + h\Bigl(1 - 16 \min_{r =1,\ldots,20}\mathrm{dist}(x,B_r)\Bigr)
              \]
for some $w_{0} < 1/(\lceil q^{\nu} \rceil+2)^d$ to be specified later.   Here, $a$ in the definition of $B_r$ is chosen such that $\int_{\mathbb{R}^d} \bar{f} = 1$, and we note that
              \[
                1 \geq 20\frac{a}{8}(2a)^{d-1} = \frac{5}{4} (2a)^d,
              \]
so $a \leq (4/5)^{1/d}/2$.              

Let 
\[
\mathcal{P}_{m} := \Bigl \{P_{\sigma}: \sigma := (\sigma_{1}, \ldots, \sigma_{m}) \in \{-1,1\}^{m} \Bigr\}.
\]
We show below that $\mathcal{P}_{m} \subseteq \mathcal{P}_{d,\theta} \cap \mathcal{Q}_{d,2,\lambda}$ for all $\theta \in \Theta$ and $\lambda > 0$ satisfying the conditions of the theorem.

Letting $\mathbb{E}_{\sigma}$ denote expectation with respect to $P_{\sigma}^{\otimes n}$ and writing $[[x_1]] := x_1 - [x_1+1/4]$ for $x_1 \in \mathbb{R}$, we have that, for any classifier $C_n$,
\begin{align*} 
&\sup_{P \in \mathcal{P}_{d, \theta} \cap \mathcal{Q}_{d,2,\lambda}} \{R(C_{n})  - R(C^{\mathrm{Bayes}})\} \geq \max_{P \in \mathcal{P}_m} \{R(C_{n})  - R(C^{\mathrm{Bayes}})\}  
\\ & = \max_{\sigma \in \{-1, 1\}^{m}} \int_{\mathbb{R}^{d}} \mathbb{E}_{\sigma} \{ \mathbbm{1}_{\{C_{n}(x) = 0\}} - \mathbbm{1}_{\{\eta_{\sigma}(x) < 1/2 \}} \} \{2\eta_{\sigma}(x) - 1\} \, dP_{X}(x) 
  \\ & \geq \max_{\sigma \in \{-1, 1\}^{m}} \sum_{l =1}^{m} \int_{\mathcal{X}_{l}} \mathbb{E}_{\sigma}\{ \mathbbm{1}_{\{C_{n}(x) = 0\}} - \mathbbm{1}_{\{\eta_{\sigma}(x) < 1/2 \}} \} \bigl\{[[x_1]] + \sigma_{l} u(x)\bigr\} \, dP_{X}(x) \\
  &\geq \frac{1}{2^m}\! \! \!\sum_{\sigma \in \{-1,1\}^m} \sum_{l =1}^{m} \int_{\mathcal{X}_{l}} \! \! \mathbb{E}_{\sigma}\{ \mathbbm{1}_{\{C_{n}(x) = 0\}} \!-\! \mathbbm{1}_{\{\eta_{\sigma}(x) < 1/2 \}} \} \bigl\{[[x_1]] \!+\! \sigma_{l} u(x)\bigr\} \, dP_{X}(x).
\end{align*} 
Now let $\sigma_{l,r} := (\sigma_{1}, \ldots, \sigma_{l-1}, r, \sigma_{l+1}, \ldots, \sigma_{m})$ for $l = 1, \ldots, m$, and $r \in \{-1,0,1\}$, and define the distribution $P_{l,r}$ on $\mathbb{R}^d \times \{0,1\}$ by $\eta_{l,r}(x) := \eta_{0}(x) + (1/2)ru(x)$, for $x \in \mathcal{X}_{l}$ and $\eta_{l,r}(x) = \eta_{\sigma_{l,r}}(x) := \eta_{\sigma}(x)$ otherwise (the marginal distribution on $\mathbb{R}^d$ is again taken to be $P_{X}$).  We write $\mathbb{E}_{l,r}$ to denote expectation with respect to $P_{l,r}^{\otimes n}$.

For $l = 1, \ldots, m$ and $r \in \{-1,1\}$ define 
\[
L_{l,r} := \frac{\prod_{i = 1}^{n} [Y_{i} \eta_{l,r}(X_{i}) + (1 - Y_{i}) \{1 - \eta_{l,r}(X_{i})\}] }{\prod_{i = 1}^{n} [Y_{i} \eta_{l,0}(X_{i}) + (1 - Y_{i}) \{1 - \eta_{l,0}(X_{i})\}] }.
\]
By the Radon--Nikodym theorem, we have that
\begin{align*}
  & \frac{1}{2^m}\sum_{\sigma \in \{-1,1\}^m} \sum_{l =1}^{m} \int_{\mathcal{X}_{l}} \mathbb{E}_{\sigma}\{ \mathbbm{1}_{\{C_{n}(x) = 0\}} - \mathbbm{1}_{\{\eta_{\sigma}(x) < 1/2 \}} \} \bigl\{[[x_1]] + \sigma_{l} u(x)\bigr\} \, dP_{X}(x)
  \\ &=\frac{1}{2}\sum_{l =1}^{m} \mathbb{E}_{l,0}\biggl(\int_{\mathcal{X}_{l}}\bigl[L_{l,1}\{ \mathbbm{1}_{\{C_{n}(x) = 0\}} - \mathbbm{1}_{\{\eta_{l,1}(x) < 1/2 \}} \}\bigr]\bigl\{[[x_1]] + u(x)\bigr\} \, dP_{X}(x) \\
  &\hspace{40pt} + \int_{\mathcal{X}_{l}}\bigl[L_{l,-1}\{ \mathbbm{1}_{\{C_{n}(x) = 0\}} - \mathbbm{1}_{\{\eta_{l,-1}(x) < 1/2 \}} \}\bigr]\bigl\{[[x_1]] - u(x)\bigr\} \, dP_{X}(x)\biggr) \\
 &\geq \frac{1}{2}\sum_{l =1}^{m} \mathbb{E}_{l,0}\biggl\{\biggl(\int_{\mathcal{X}_{l}}\{ \mathbbm{1}_{\{C_{n}(x) = 0\}} - \mathbbm{1}_{\{\eta_{l,1}(x) < 1/2 \}} \}\bigl\{[[x_1]] + u(x)\bigr\} \, dP_{X}(x) \\
  &\hspace{5pt} + \int_{\mathcal{X}_{l}}\!\!\{ \mathbbm{1}_{\{C_{n}(x) = 0\}} \!-\! \mathbbm{1}_{\{\eta_{l,-1}(x) < 1/2 \}}\}\bigl\{[[x_1]] \!-\! u(x)\bigr\} \, dP_{X}(x)\!\biggr)\!\min(L_{l,1},L_{l,-1})\biggr\}.
\end{align*} 
Now fix $x = (x_1,\ldots,x_d)^T \in \mathcal{X}_l$, and writing $C_n = C_n(x), \eta_{l,1} = \eta_{l,1}(x)$ and $\eta_{l,-1} = \eta_{l,-1}(x)$ as shorthand, observe that 
\begin{align*}
  & \{ \mathbbm{1}_{\{C_{n} = 0\}} - \mathbbm{1}_{\{\eta_{l,1} < 1/2 \}} \}\bigl\{[[x_1]] + u(x)\bigr\} \\
  &\hspace{5cm}+ \{ \mathbbm{1}_{\{C_{n} = 0\}} - \mathbbm{1}_{\{\eta_{l,-1} < 1/2 \}} \}\bigl\{[[x_1]] - u(x)\bigr\}  
\\ &= 2\Bigl\{ \mathbbm{1}_{\{C_{n} = 0,\eta_{l,1} \geq 1/2,\eta_{l,-1} \geq 1/2 \}} 
-  \mathbbm{1}_{\{C_{n} = 1,\eta_{l,1} < 1/2,\eta_{l,-1} < 1/2 \}} \Bigr\} [[x_1]] 
\\ & \hspace{20 pt}+ \Bigl\{\mathbbm{1}_{\{C_{n} = 0,\eta_{l,1} \geq 1/2,\eta_{l,-1} < 1/2 \}} - \mathbbm{1}_{\{C_{n} = 1,\eta_{l,1} < 1/2,\eta_{l,-1} \geq 1/2 \}} \Bigr\}  \{[[x_1]] + u(x)\}  
  \\ & \hspace{20 pt} +  \Bigl\{\mathbbm{1}_{\{C_{n} = 0,\eta_{l,1} < 1/2,\eta_{l,-1} \geq 1/2 \}} -  \mathbbm{1}_{\{C_{n} = 1,\eta_{l,1} \geq 1/2,\eta_{l,-1} < 1/2 \}}\Bigr\} \{[[x_1]] - u(x)\}
\\ &= 2\Bigl\{ \mathbbm{1}_{\{C_{n}=0,\eta_{l,1}\geq 1/2,\eta_{l,-1} \geq 1/2 \}} -  \mathbbm{1}_{\{C_{n} = 1,\eta_{l,1} < 1/2,\eta_{l,-1} < 1/2 \}}\Bigr\}[[x_1]]
\\ & \hspace{20 pt}+ \mathbbm{1}_{\{\eta_{l,1} \geq 1/2,\eta_{l,-1} < 1/2\}}\bigl[\mathbbm{1}_{\{C_n=0\}}\{[[x_1]] + u(x)\} - \mathbbm{1}_{\{C_{n} = 1\}} \{[[x_1]] - u(x)\}\bigr]  \\     
 &\geq \mathbbm{1}_{\{\eta_{l,1} \geq 1/2,\eta_{l,-1} < 1/2\}}\bigl\{u(x) - \bigl|[[x_1]]\bigr|\bigr\}.
\end{align*} 
Here we used the fact that $\eta_{l, 1}(x) \geq \eta_{l, -1}(x)$, so $\mathbbm{1}_{\{\eta_{l,1}(x) < 1/2,\eta_{l,-1}(x) \geq 1/2 \}}  = 0$, and that the minimum is attained by taking $C_{n}(x) = \mathbbm{1}_{\{[[x_1]] \geq 0\}}$ for $x \in \mathcal{X}_l$; it is interesting to note that this remains the optimal classifier even if $\bar{f}$ is known.  Moreover, whenever $[[x_1]] \geq 0$, we have $\eta_{l,1}(x) \geq 1/2$, and when $[[x_1]] < 0$, we have $\eta_{l,-1}(x) < 1/2$.  It follows that
\begin{align}
  \label{Eq:LowerBdDecomp}
  &\sup_{P \in \mathcal{P}_{d, \theta} } \{R(C_{n})  - R(C^{\mathrm{Bayes}})\} \nonumber \\
  &\geq \frac{1}{2}\sum_{l =1}^{m} \mathbb{E}_{l,0}\biggl\{\min(L_{l,1},L_{l,-1})\! \int_{\mathcal{X}_l} \! \! \! \! \mathbbm{1}_{\{\eta_{l,1} \geq 1/2,\eta_{l,-1} < 1/2 \}}\bigl\{u(x) \!-\! \bigl|[[x_1]]\bigr|\bigr\} \, dP_X(x)\biggr\} \nonumber \\
                                        & = \sum_{l =1}^{m} \mathbb{E}_{l,0}\bigl\{\min(L_{l,1},L_{l,-1})\bigr\}\int_{\mathcal{X}_l \cap \{[[x_1]] \geq 0\}} \! \! \mathbbm{1}_{\{\eta_{l,-1} < 1/2 \}}\bigl\{u(x) - [[x_1]]\bigr\} \, dP_X(x) \nonumber \\
  &= m w_{0}\mathbb{E}_{1,0}\bigl\{\min(L_{1,1},L_{1,-1})\bigr\}\int_{B_{1/(2q)}(0) \cap \{x_{1} \geq 0\}}\mathbbm{1}_{\{\tilde{\eta}(x) < 1/2 \}} \{\tilde{u}(x) - x_{1}\}  \, dx,
\end{align}
where $\tilde{u}(x) := \alpha_0g^{-1}(1/q) q^{6} (\frac{1}{4q^{2}} - \|x\|^2)^{4}$ and $\tilde{\eta}(x) := \frac{1}{2}\{1 + x_{1}  - \tilde{u}(x) \}$.

Now, observe that 
\[
\mathbb{E}_{1,0}\bigl\{\min(L_{1,1},L_{1,-1})\bigr\} =  1 - d_{\mathrm{TV}}(P_{1,1}^{\otimes n}, P_{1,-1}^{\otimes n}) 
\]
and 
\[
d_{\mathrm{TV}}^{2}(P_{1,1}^{\otimes n}, P_{1,-1}^{\otimes n}) \leq \frac{1}{2}d_{\mathrm{KL}}^2(P_{1,1}^{\otimes n}, P_{1,-1}^{\otimes n}) = \frac{n}{2}d_{\mathrm{KL}}^2(P_{1,1}, P_{1,-1}).
\] 
Moreover, using the fact that $\log(1+ x) \leq x$ for $x \geq 0$, we have that
\begin{align*}
  d_{\mathrm{KL}}^2&(P_{1,1}, P_{1,-1}) \\
  &=  \int_{\mathbb{R}^{d}} \eta_{1,1}(x) \log \biggl(\frac{\eta_{1,1}(x)}{\eta_{1,-1}(x)}\biggr)  + \{1- \eta_{1,1}(x)\} \log\biggl(\frac{1 - \eta_{1,1}(x)}{1 - \eta_{1,-1}(x)} \biggr)\, dP_X(x) 
\\ & \leq 24 \int_{\mathcal{X}_{1}} u^{2}(x) \, dP_X(x) 
\\ & = 24 \alpha_0^2w_{0} q^{12} g^{-1}(1/q)^{2} \int_{B_{1/(2q)}(0)} \Bigl(\frac{1}{4q^{2}} - \|x\|^2\Bigr)^{8} \, dx
\\ & = \frac{945 \alpha_0^2a_{d} w_{0} g^{-1}(1/q)^{2}\Gamma(1+d/2)}{2^{d+6}\Gamma(9+d/2)} q^{-(4+d)}.
\end{align*} 
We now turn to finding a lower bound for the integral in~\eqref{Eq:LowerBdDecomp}.  First, we observe that $\mathrm{sgn}\bigl(\tilde{u}(x) - x_1\bigr) = \mathrm{sgn}\bigl(1/2 - \tilde{\eta}(x)\bigr)$, and moreover for $d =1$ and $0 \leq x_1 < \frac{\alpha_0g^{-1}(1/q)}{2^{13} q^{2}}$, we have that 
\begin{align*}
  \tilde{u}(x_1) - x_1 = q^{6} \alpha_0g^{-1}(1/q) \biggl(\frac{1}{4q^{2}} - x_1^2\biggr)^{4} - x_1 &> \frac{\alpha_0g^{-1}(1/q)}{2^{12}q^{2}} - x_1 \\
  &> \frac{\alpha_0g^{-1}(1/q)}{2^{13}q^{2}}> 0.
\end{align*}
Thus
\begin{align*} 
  \int_{0}^{1/(2q)} \mathbbm{1}_{\{\tilde{\eta} (x_1) < 1/2 \}}\{\tilde{u}(x_1) - x_1\} \, dx_1 &\geq \frac{\alpha_0^2g^{-1}(1/q)^2}{2^{26}q^{4}}.
\end{align*}
Furthermore, for $d \geq 2$, writing $x_{-1}:= (x_{2}, \ldots, x_{d})^T$, we have that $\tilde{\eta}(x) < 1/2$ if and only if
\[
0 >  x_{1}  -  q^{6}\alpha_0g^{-1}(1/q)\Bigl(\frac{1}{4q^{2}} - \|x\|^2\Bigr)^{4}\]
which is satisfied if 
\[
\|x_{-1}\| < (1-2^{-1/4})^{1/2}\sqrt{\frac{1}{4q^{2}} - \Bigl( \frac{x_{1}}{ q^{6} \alpha_0g^{-1}(1/q) }\Bigr)^{1/4} - x_{1}^{2}} \, \, =: t(x_{1}).
\]
Now $t(x_{1})$ is real if $0 \leq x_{1} \leq \frac{\alpha_0g^{-1}(1/q)}{2^{14} q^{2}}$.  Moreover, $t(x_{1}) > 1/(8q)$ for $x_{1} \in \bigl[0, \frac{\alpha_0 g^{-1}(1/q)}{2^{14} q^2}\bigr]$.
We also require the observation that $\tilde{u}(x) - x_1 \geq \frac{\alpha_0 g^{-1}(1/q)}{2^{14}q^2}$ when $x_{1} \in \bigl[0, \frac{\alpha_0 g^{-1}(1/q)}{2^{14} q^2}\bigr]$ and $\|x_{-1}\| < t(x_1)$.  Hence
\begin{align*} 
  \int_{B_{1/(2q)}(0) \cap \{x_{1} \geq 0\}}&\mathbbm{1}_{\{\tilde{\eta}(x) < 1/2 \}} \{\tilde{u}(x) - x_{1}\}  \, dx  \\
  &\geq \int_{0}^{\frac{\alpha_0 g^{-1}(1/q)}{2^{14} q^{2}}} \int_{\|x_{-1}\| < t(x_{1})} \{\tilde{u}(x) - x_{1}\} \, dx_{-1} \, dx_{1} \\
                                                                                                                     &\geq \frac{\alpha_0g^{-1}(1/q)}{2^{14}q^2}a_{d-1}\int_{0}^{\frac{\alpha_0 g^{-1}(1/q)}{2^{14} q^{2}}} t(x_1)^{d-1} \, dx_1 \\
                                            &\geq \frac{\alpha_0^2g^{-1}(1/q)^2}{2^{28}q^{3+d}}a_{d-1}2^{-3(d-1)}.
\end{align*} 
We have therefore shown that, for $q \geq 3$, 
\begin{align*}
  &\sup_{P \in \mathcal{P}_{d, \theta} } \{R(C_{n})  - R(C^{\mathrm{Bayes}})\}  \\
  &\hspace{1cm}\geq \frac{m w_{0} a_{d-1}\alpha_0^2g^{-1}(1/q)^2}{2^{28 + 3(d-1)} q^{3+d}} \Biggl\{1 - \sqrt{\frac{n 945 a_{d} w_{0} \alpha_0^2g^{-1}(1/q)^{2} \Gamma(1+d/2)}{2^{d+6} \Gamma(9 + d/2) q^{4+d}}}\Biggr\},
\end{align*}
where $a_0 := 1$.  It follows that if we set
\[
w_{0} = \frac{q^{4+d}}{4^{d+1}ng^{-1}(1/q)^{2}},
\]
and choose $q$ to satisfy $\frac{q^{4+d+\nu(\rho + d)}}{g^{-1}(1/q)^{2}} = n$, then 
\begin{align*}
  \sup_{P \in \mathcal{P}_{d, \theta} }  \{R(C_{n})  - R(C^{\mathrm{Bayes}})\} &\geq \frac{ q^{d + \nu d} }{n} \frac{a_{d-1}\alpha_0^2}{2^{28 + 5d} } \\
  &= g^{-1}(1/q)^{\frac{2d(1+\nu)}{4+d+\nu(\rho + d)}}n^{-\frac{4 + \nu\rho}{4+d + \nu(\rho+d)}} \frac{a_{d-1}\alpha_0^2}{2^{28 + 5d} }.
\end{align*}

It remains to show that $P_{\sigma}$ belongs to the desired classes $\mathcal{P}_{d, \theta} \cap \mathcal{Q}_{d,2,\lambda}$ for each $\sigma$.  First note that
\[
  w_{0} = \frac{q^{4+d}}{4^{d+1}ng^{-1}(1/q)^{2}} = \frac{1}{4^{d+1}}q^{- \nu(\rho+d)} < \frac{1}{(\lceil q^\nu \rceil + 2)^d}.
\]
Condition~\textbf{(A.1)} is satisfied by $\bar{f}$ by construction.  To verify the minimal mass assumption, we take $\epsilon_* < 2^{-\max(d,5)}$, and observe that when $\epsilon_0 \in (0,\epsilon_*]$,
\begin{align*}
  &\inf_{r_0 \in (0,\epsilon_0],x \in \mathbb{R}^d:\bar{f}(x) > 0} \frac{1}{a_dr_0^d\bar{f}(x)} \int_{B_{r_0}(x)} \bar{f} \\
  &\hspace{2cm}\geq \inf_{r_0 \in (0,\epsilon_0]} \frac{1}{a_dr_0^d} \int_{B_{r_0}(0)} \Phi\biggl(\frac{1 - 32\|x\|}{16\|x\|(1-16\|x\|)}\biggr) \, dx \wedge 2^{-d} \geq 2^{-d},
\end{align*}
as required.  It follows that \textbf{(A.2)} is satisfied for such $\epsilon_0 \in (0,\epsilon_*]$ and for any $M_0 \geq 1$.

The main condition to check is \textbf{(A.3)}.   
For $x \in B_{1/(2q)}(0)$, consider
\[
  \tilde{\eta}_{\pm}(x) := (1/2)\biggl\{1 + x_{1} \pm q^{6} \alpha_0g^{-1}(1/q) \biggl(\frac{1}{4q^{2}} - \|x\|^2 \biggr)^{4}\biggr\}.
\]
Then
\[
\dot{\tilde{\eta}}_{\pm}(x) =   (1/2, 0, \ldots, 0)^T \mp 8 q^{6} \alpha_0g^{-1}(1/q) \biggl(\frac{1}{4q^{2}}-\|x\|^2 \biggr)^{3}x, 
\]
and
\begin{align*}
  \ddot{\tilde{\eta}}_{\pm}(x) &=  \mp 8 q^{6} \alpha_0g^{-1}(1/q)  \biggl(\frac{1}{4q^{2}} - \|x\|^{2}\biggr)^{3} I_{d \times d} \\
  &\hspace{2cm}\pm 48 q^{6} \alpha_0g^{-1}(1/q) \biggl(\frac{1}{4q^{2}} - \|x\|^2 \biggr)^{2}xx^{T}.
\end{align*}
From these calculations, we see that each $\eta_\sigma$ is twice continuously differentiable on $\mathcal{S}^{2\epsilon_0}$, with $\|\dot{\eta}_\sigma(x)\| \in (1/4,3/4)$ for all $x \in \mathcal{S}^{2\epsilon_0}$ and $\|\ddot{\eta}_\sigma(x)\|_{\mathrm{op}} \leq 1$.  We have that, when $n_q(z) = n_q(x)$, 
\begin{align}
  \label{Eq:SameBalls}
  &\|\ddot{\eta}_\sigma(z) - \ddot{\eta}_\sigma(x)\|_{\mathrm{op}} \nonumber \\
   &\leq 8 q^{6} \alpha_0g^{-1}(1/q)\biggl|\biggl(\frac{1}{4q^{2}} - \|z\|^{2}\biggr)^{3} - \biggl(\frac{1}{4q^{2}} - \|x\|^{2}\biggr)^{3}\biggr| \nonumber \\
  &\hspace{1cm}+ 48 q^{6} \alpha_0g^{-1}(1/q)\biggl\|\biggl(\frac{1}{4q^{2}} - \|z\|^2 \biggr)^{2}zz^T - \biggl(\frac{1}{4q^{2}} - \|x\|^2 \biggr)^{2}xx^T\biggr\| \nonumber \\
  &= 8 q^{6} \alpha_0g^{-1}(1/q)(\|z\|+\|x\|)\bigl|\|z\|-\|x\|\bigr|\biggl\{\Bigl(\frac{1}{4q^{2}} \!-\! \|z\|^{2}\Bigr)^2 \nonumber \\
  &\hspace{3cm}+ \Bigl(\frac{1}{4q^{2}} \!-\! \|z\|^{2}\Bigr)\Bigl(\frac{1}{4q^{2}} \!-\! \|x\|^{2}\Bigr) + \Bigl(\frac{1}{4q^{2}} \!-\! \|x\|^{2}\Bigr)^2\biggr\} \nonumber \\
    &\hspace{1cm}+ 48 q^{6} \alpha_0g^{-1}(1/q)\biggl\|\Bigl(\frac{1}{4q^{2}} \!-\! \|z\|^2 \Bigr)^{2}\Bigl\{(z\!-\!x)(z\!-\!x)^T \nonumber \\
    &\hspace{1cm}+ x(z\!-\!x)^T \!+\! (z\!-\!x)x^T\Bigr\} + (\|x\|^2 \!-\! \|z\|^2)\Bigl\{\frac{1}{2q^2}-\|x\|^2-\|z\|^2\Bigr\}xx^T\biggr\| \nonumber \\
  &\leq \frac{1}{2}qg^{-1}(1/q)\|z-x\|.
\end{align}
Hence, using the fact that $r \mapsto r/g^{-1}(r)$ is increasing for sufficiently small $r > 0$, we have that for sufficiently large $q$, 
\[
\sup_{\|x - z\| \leq g(\epsilon),n_q(x)=n_q(z)} \|\ddot{\eta}_\sigma(x) - \ddot{\eta}_\sigma(z)\|_{\mathrm{op}} \leq \epsilon.
\]
Now consider the case where $z \in \mathcal{X}_l$ and $x \in \mathcal{X}_{l'}$ with $l \neq l'$, so that $n_q(z) \neq n_q(x)$.  Let $z'$ denote the closest point in $\mathcal{X}_l$ to $\mathcal{X}_{l'}$ on the line segment joining $x$ to $z$, and similarly let $x'$ denote the closest point in $\mathcal{X}_{l'}$ to $\mathcal{X}_{l}$ on the same line segment.  Then $\ddot{\eta}_\sigma(x') = \ddot{\eta}_\sigma(z') = 0$, so, by~\eqref{Eq:SameBalls},
\begin{align*}
  &\|\ddot{\eta}_\sigma(z) - \ddot{\eta}_\sigma(x)\|_{\mathrm{op}} \leq \|\ddot{\eta}_\sigma(z) - \ddot{\eta}_\sigma(z')\|_{\mathrm{op}} + \|\ddot{\eta}_\sigma(x') - \ddot{\eta}_\sigma(x)\|_{\mathrm{op}} \\
& \hspace{10pt} \leq \frac{1}{2} qg^{-1}(1/q)(\|z-z'\| + \|x-x'\|) \leq \frac{1}{2}\bigl\{g^{-1}(\|z-z'\|) + g^{-1}(\|x-x'\|)\bigr\}.
\end{align*}
We therefore deduce that
\[
\sup_{\|x - z\| \leq g(\epsilon)} \|\ddot{\eta}_\sigma(x) - \ddot{\eta}_\sigma(z)\|_{\mathrm{op}} \leq \epsilon.
\]
For the final part of \textbf{(A.3)}, we note that
\[
  \inf_{x \in (\epsilon_0 \pm \mathbb{Z}/2) \times \mathbb{R}^{d-1}} \Bigl|\eta_\sigma(x) - \frac{1}{2}\Bigr| \geq \frac{\epsilon_0}{2}.
\]
Finally, we check the moment condition in \textbf{(A.4)}.  First, 
\begin{align*} 
 & \int_{\mathbb{R}^d} \|x\|^\rho \bar{f}(x) \, dx = w_{0} \int_{x:\mathrm{dist}(x,B_0)\leq 1/4} \|x\|^{\rho}h\Bigl(4\bigl(1 - \mathrm{dist}(x,B_0)\bigr)\Bigr) \, dx \\
  &\hspace{18pt}+ \int_{[-10,-1] \times [-a-1/16,a+1/16]^{d-1}} \|x\|^{\rho}h\biggl(16\Bigl(1 - \min_{r=1,\ldots,20} \mathrm{dist}(x,B_r)\Bigr)\biggr) \, dx \\
& \leq w_0d^{\frac{\rho}{2}}(\lceil q^\nu \rceil+2)^{d+\rho} + \max(1,2^{\frac{\rho-2}{2}})\{100^{\frac{\rho}{2}} + (d-1)^{\frac{\rho}{2}}(a+1/16)^{\frac{\rho}{2}}\} \\
&\leq \frac{3^{d+\rho}d^{\frac{\rho}{2}}}{4^{d+1}}+ \max(1,2^{(\rho-2)/2})\{100^{\frac{\rho}{2}} + (d-1)^{\frac{\rho}{2}}(a+1/16)^{\frac{\rho}{2}}\} =: M_{01}(\rho),
\end{align*} 
say.  We conclude that there exists $q_* = q_*(d)$ such that for $q \geq q_*$ and any $\nu \geq 0$, we have $P \in \mathcal{P}_{d,\theta}$ for $\theta = (\epsilon_0,M_0,\rho,\ell,g)$ with any $\rho > 0$, $M_0 \geq \max(M_{01}(\rho),1)$, $\epsilon_0 \in (0,\min(2^{-\max(d,5)},1/(4M_0)))$, any $\ell \in \mathcal{L}$ with $\ell \geq 2/\epsilon_0$ and any $g \in \mathcal{G}$.  

Finally, we note that $\|\bar{f}\|_\infty \leq 1$ and
\begin{align*}
  \|\ddot{\bar{f}}\|_\infty &\leq 2^8 \! \! \sup_{x \in (0,1)} \phi\biggl(\frac{2x-1}{x(1-x)}\biggr)\biggl\{\! \frac{2}{(1-x)^3} \!-\! \frac{2}{x^3} \!-\! \frac{2x-1}{x(1-x)}\Bigl(\frac{1}{(1-x)^2} + \frac{1}{x^2}\Bigr) \! \biggr\} \\
	&\leq 2^{10} \times 5.
\end{align*}
Hence $P \in \mathcal{Q}_{d,2,\lambda}$ for $\lambda \geq 2^{10} \times 5$.
\end{proof}

\section{Proof of Theorem~5 (continued)}
\label{sec:step7}

\begin{proof}[Proof of Theorem~\ref{thm:exriskRtn} -- \textbf{Step 7}]
To complete the proof of Theorem~\ref{thm:exriskRtn}, it remains to bound the error terms $R_1, R_2, R_5$ and $R_6$.

\textit{To bound $R_1$}: We have
\begin{align*}
R_{1} = \frac{1}{k_{\mathrm{L}}}  \sum_{i=1}^{k_{\mathrm{L}}}\biggl(\mathbb{E}\eta(X_{(i)})  - \eta(x) & - \mathbb{E}\{(X_{(i)} - x)^T\dot{\eta}(x)\}  
\\ & - \frac{1}{2}\mathbb{E}\{(X_{(i)} - x)^T\ddot{\eta}(x)(X_{(i)} - x)\}\biggr). 
\end{align*}
By a Taylor expansion and \textbf{(A.3)}, for all $\epsilon \in (0,1)$, $x \in \mathcal{S}^{\epsilon_0}$ and $\|z-x\| < \min\{g(\epsilon),\epsilon_0\} =: r$,
\[
\biggl|\eta(z) - \eta(x) - (z - x)^T\dot{\eta}(x)  -  \frac{1}{2} (z - x)^T\ddot{\eta}(x)(z - x)\biggr| \leq \epsilon \|z-x\|^2.
\]
Hence
\begin{align}
\label{Eq:R1Main} 
|R_{1}| & \leq \epsilon \frac{1}{k_{\mathrm{L}}} \sum_{i=1}^{k_{\mathrm{L}}}\mathbb{E}\{\|X_{(i)} - x\|^2 \mathbbm{1}_{\{\|X_{(k_{\mathrm{L}})} - x\| \leq r\}} \} +2\mathbb{P}{\{\|X_{(k_{\mathrm{L}})} - x\| > r\}} \nonumber 
\\ & \hspace{25 pt} + \sup_{z \in \mathcal{S}^{\epsilon_0}}\|\dot{\eta}(z)\| \mathbb{E}\{\|X_{(k_{\mathrm{L}})} - x\| \mathbbm{1}_{\{\|X_{(k_{\mathrm{L}})} - x\| > r\}} \} \nonumber
\\ & \hspace {50 pt} +  \sup_{z \in \mathcal{S}^{\epsilon_0}}\|\ddot{\eta}(z)\|_{\mathrm{op}} \mathbb{E}\{\|X_{(k_{\mathrm{L}})} - x\|^2 \mathbbm{1}_{\{\|X_{(k_{\mathrm{L}})} - x\| > r\}} \}.
\end{align} 
Now, by similar arguments to those leading to~\eqref{eq:biasexpd}, we have that
\begin{align}
\label{Eq:R1Main2}
\frac{\epsilon}{k_{\mathrm{L}}} &\sum_{i=1}^{k_{\mathrm{L}}}\mathbb{E}(\|X_{(i)} - x\|^2  \mathbbm{1}_{\{\|X_{(k_{\mathrm{L}})} - x\| \leq r\}}) = \epsilon \Bigl(\frac{k_{\mathrm{L}}}{na_{d}\bar{f}(x)}\Bigr)^{2/d} \frac{d}{d+2} \{1+o(1)\},
\end{align} 
uniformly for $P \in \mathcal{P}_{d,\theta}$, $k_{\mathrm{L}} \in K_{\beta,\tau}$, $x_0 \in \mathcal{S}_n$ and $|t|<\epsilon_n$.  Moreover, for every $M > 0$,
\begin{equation}
\label{Eq:R1Prob}
\mathbb{P}\{\|X_{(k_{\mathrm{L}})} - x\| > r\} = q_r^n(k_{\mathrm{L}}) = O(n^{-M}),
\end{equation}
uniformly for $P \in \mathcal{P}_{d,\theta}$, $k_{\mathrm{L}} \in K_{\beta,\tau}$, $x_0 \in \mathcal{S}_n$ and $|t|<\epsilon_n$, by~\eqref{eq:binbound2} in Step~1.  For the remaining terms, note that
\begin{align}
\label{Eq:R12ndMoment}
\mathbb{E}\{\|X_{(k_{\mathrm{L}})} & - x\|^2 \mathbbm{1}_{\{\|X_{(k_{\mathrm{L}})} - x\| > r\}} \}  \nonumber
\\ &= \mathbb{P}\{\|X_{(k_{\mathrm{L}})} - x\| > r\} + \int_{r^2}^{\infty} \mathbb{P}\{\|X_{(k_{\mathrm{L}})} - x\| > \sqrt{t}\} \, dt \nonumber
\\ & = q_{r}^{n}(k_{\mathrm{L}}) +  \int_{r^2}^{\infty} q_{\sqrt{t}}^n(k_{\mathrm{L}}) \, dt.
\end{align}
Let $t_0 = t_0(x) := 5^{2/\rho}(1+2^{\rho-1})^{2/\rho}\bigl(M_0 + \|x\|^\rho\bigr)^{2/\rho}$.  Then, for $t \geq t_0$, we have 
\[
1 - p_{\sqrt{t}} \leq (1+2^{\rho-1})\frac{\mathbb{E}(\|X\|^\rho) + \|x\|^\rho}{t^{\rho/2}} \leq \frac{1}{5}.
\]
It follows by Bennett's inequality that for $\rho\{n-(n-1)^{1-\beta}\} > 4$,
\begin{align*}
& \int_{t_0}^\infty q_{\sqrt{t}}^n(k_{\mathrm{L}}) \, dt 
\\ &\leq e^{k_{\mathrm{L}}}(1+2^{\rho-1})^{(n-k_{\mathrm{L}})/2}\bigl\{M_0 + \|x\|^\rho\bigr\}^{(n-k_{\mathrm{L}})/2} \int_{t_0}^\infty t^{-\rho(n-k_{\mathrm{L}})/4} \, dt \\
&= \frac{4e^{k_{\mathrm{L}}}5^{2/\rho}}{\rho(n-k_{\mathrm{L}}) - 4}(1+2^{\rho-1})^{2/\rho}\bigl\{M_0 + \|x\|^\rho\bigr\}^{2/\rho}5^{-(n-k_{\mathrm{L}})/2}.
\end{align*}
But, when $\beta \log(n-1) \geq (d+2)/d$ and $n \geq \max\{n_0, n_2\}$,
\[
\sup_{x \in \mathcal{R}_n \cup \mathcal{S}_n^{\epsilon_n}}  \|x\| \leq \epsilon_0 + \biggl\{\frac{(n-1)^{1-\beta} c_n^dM_0}{\mu_0 \beta^{d/2} \log^{d/2}(n-1)}\biggr\}^{1/\rho}.
\]
We deduce that for every $M > 0$,
\begin{equation}
\label{Eq:R12ndMoment1}
\sup_{P \in \mathcal{P}_{d,\theta}} \sup_{k \in K_{\beta,\tau}} \sup_{x \in \mathcal{R}_n \cup \mathcal{S}_n^{\epsilon_n}} \int_{t_0}^\infty q_{\sqrt{t}}^n(k_{\mathrm{L}}) \, dt = O(n^{-M}).
\end{equation}
Moreover, by Bernstein's inequality, for every $M > 0$,
\begin{equation}
\label{Eq:R12ndMoment2}
\sup_{P \in \mathcal{P}_{d,\theta}} \sup_{k_{\mathrm{L}}\in K_{\beta,\tau} } \sup_{x\in \mathcal{R}_n \cup \mathcal{S}_n^{\epsilon_n}} \Bigl\{q_{r}^{n}(k_{\mathrm{L}}) +  \int_{r^2}^{t_0} q_{\sqrt{t}}^n(k_{\mathrm{L}}) \, dt \Bigr\} = O(n^{-M}). 
\end{equation}
We conclude from~\eqref{eq:kfbar},~\eqref{Eq:R1Main},~\eqref{Eq:R1Main2},~\eqref{Eq:R1Prob},~\eqref{Eq:R12ndMoment},~\eqref{Eq:R12ndMoment1} and~\eqref{Eq:R12ndMoment2}, together with Jensen's inequality to deal with the third term on the right-hand side of~\eqref{Eq:R1Main}, that~\eqref{Eq:R1} holds.  With only simple modifications, we have also shown~\eqref{eq:R2}, which bounds $R_2$.

\bigskip

\textit{To bound $R_5$}: Write
\begin{align*}
R_5 :=& \int_{\mathcal{S}_n} R_5(x_0) \, d\mathrm{Vol}^{d-1}(x_0) \\
=& \int_{\mathcal{S}_n} \int_{-\epsilon_{n}}^{\epsilon_{n}}  t\|\dot{\psi}(x_0)\| \Bigl[ \mathbb{P}\{\hat{S}_n(x_0^t) < 1/2\} - \mathbb{E}\Phi\bigl(\hat\theta(x_0^t)\bigr) \Bigr] \,dt \, d\mathrm{Vol}^{d-1}(x_0).
\end{align*} 
Now by a non-uniform version of the Berry--Esseen theorem \citep[][Theorem~1]{Paditz1989}, for every $t \in (-\epsilon_n,\epsilon_n)$ and $x_0 \in \mathcal{S}_n$,
\begin{equation}
\label{eq:berryesseen}
\bigl|\mathbb{P}\{\hat{S}_n(x_0^t) < 1/2|X^n\} - \Phi\bigl(\hat\theta(x_0^t)\bigr)\bigr| \leq \frac{32}{k_{\mathrm{L}}(x_0^t) \hat{\sigma}_n(x_0^t,X^n)}\frac{1}{1+|\hat{\theta}(x_0^t)|^3}.
\end{equation}
Let
\[
	t_n=t_n(x_0) := C \max \biggl\{ k_{\mathrm{L}}(x_0)^{-1/2} , \Bigl( \frac{k_{\mathrm{L}}(x_0)}{n \bar{f}(x_0)} \Bigr)^{2/d}\ell\bigl(\bar{f}(x_0)\bigr) \biggr\},
\]
where
\[
	C := \frac{4}{a_d^{2/d}\epsilon_0}. 
\]
In the following we integrate the bound in~\eqref{eq:berryesseen} over the regions $|t|\leq t_n$ and $|t|\in (t_n, \epsilon_n)$ separately.  Define the event 
\[
B_{k_{\mathrm{L}}} := \biggl\{\hat{\sigma}_n(x_0^t,X^n) \geq \frac{1}{3k_{\mathrm{L}}(x_0^t)^{1/2}} \ \text{for all $x_0 \in \mathcal{S}_n$, $t \in (-\epsilon_n,\epsilon_n)$}\biggr\},
\]
so that, by very similar arguments to those used to bound $\mathbb{P}(A_{k_{\mathrm{L}}}^c)$ in Step~2, we have $\mathbb{P}(B_{k_{\mathrm{L}}}^c) = O(n^{-M})$ for every $M > 0$, uniformly for $P \in \mathcal{P}_{d,\theta}$ and $k_{\mathrm{L}} \in K_{\beta,\tau}$.  It follows by~\eqref{eq:berryesseen} and Step~2 that there exists $n_4 \in \mathbb{N}$ such that for all $n \geq n_4$, $k_{\mathrm{L}} \in K_{\beta,\tau}$ and $x_0 \in \mathcal{S}_n$,
\begin{align}
\label{eq:R61}
	\biggl|  \int_{-t_n}^{t_n}  t \Bigl[&  \mathbb{P}\{\hat{S}_n(x_0^t) < 1/2\} - \mathbb{E}\Phi\bigl(\hat\theta(x_0^t)\bigr) \Bigr] \,dt \biggr| \nonumber 
\\ & \leq  \int_{-t_n}^{t_n} \mathbb{E} \biggl( \frac{32 |t|\mathbbm{1}_{B_{k_{\mathrm{L}}}}}{k_{\mathrm{L}}(x_0^t) \hat{\sigma}_n(x_0^t,X^n)} \biggr) \,dt + t_n^2\mathbb{P}(B_{k_{\mathrm{L}}}^c)\leq \frac{128 t_n^2}{k_\mathrm{L}(x_0)^{1/2}}.
\end{align}
By Step~1, there exists $n_5 \in \mathbb{N}$ such that for $n \geq n_5$, $P \in \mathcal{P}_{d,\theta}$, $k_\mathrm{L} \in K_{\beta,\tau}$, $x_0 \in \mathcal{S}_n$ and $|t| \in (t_n,\epsilon_n)$,
\begin{align}
\label{Eq:mun}
	|\mu_n(x_0^t)-1/2| &\geq |\eta(x_0^t)-1/2| - |\mu_n(x_0^t)-\eta(x_0^t)| \nonumber \\
	& \geq \frac{1}{2} \inf_{z \in \mathcal{S}} \|\dot{\eta}(z)\| |t| - \frac{1}{4} C \epsilon_0M_0  \Bigl( \frac{k_\mathrm{L}(x_0)}{n \bar{f}(x_0)} \Bigr)^{2/d}\ell\bigl(\bar{f}(x_0)\bigr) \nonumber
\\ & > \frac{1}{4} \epsilon_0 M_0 |t|.
\end{align}
Thus for $n \geq n_5$, $P \in \mathcal{P}_{d,\theta}$, $k_\mathrm{L} \in K_{\beta,\tau}$, $x_0 \in \mathcal{S}_n$ and $|t| \in (t_n,\epsilon_n)$, we have that
\begin{align}
\label{eq:thetaprob}
	\mathbb{P} &\Bigl\{ |\hat{\theta}(x_0^t)| < \frac{1}{4} \epsilon_0 M_0 k_\mathrm{L}^{1/2}(x_0) |t|\Bigr\} \nonumber \\
	& \leq \mathbb{P} \Bigl\{ |\hat{\mu}_n(x_0^t,X^n)- \mu_n(x_0^t)| > |\mu_n(x_0^t)-1/2| -\frac{1}{8} \epsilon_0 M_0|t|\Bigr\} \nonumber \\
	& \leq \mathbb{P} \Bigl\{ |\hat{\mu}_n(x_0^t,X^n)- \mu_n(x_0^t)| > \frac{1}{8} \epsilon_0M_0|t|\Bigr\} \leq \frac{64 \mathrm{Var}\{ \hat{\mu}_n(x_0^t,X^n) \}}{ \epsilon_0^2 M_0^2 t^2}.
\end{align}
It follows by~\eqref{eq:berryesseen},~\eqref{eq:thetaprob} and Step~3 that, for $n \geq n_5$,
\begin{align}
\label{eq:tnen}
	\biggl|  \int_{|t| \in (t_n,\epsilon_n)} t &\bigl[ \mathbb{P}\{\hat{S}_n(x_0^t) < 1/2\} - \mathbb{E}\Phi\bigl(\hat\theta(x_0^t)\bigr) \bigr] \,dt \biggr| \nonumber \\
	& \hspace{6pt} \leq \int_{|t| \in (t_n,\epsilon_n)}\!\!\!\! |t| \mathbb{E} \biggl( \frac{32\mathbbm{1}_{B_{k_{\mathrm{L}}}}}{k_{\mathrm{L}}(x_0^t) \hat{\sigma}_n(x_0^t,X^n)}\frac{1}{1\!+\!\frac{1}{64}\epsilon_0^3M_0^3 k_\mathrm{L}(x_0)^{3/2}|t|^3} \biggr) \,dt \nonumber \\
	& \hspace{36pt} + \int_{|t| \in (t_n,\epsilon_n)} \frac{64 \mathrm{Var}\{ \hat{\mu}_n(x_0^t,X^n) \}}{ \epsilon_0^2M_0^2 |t|} \,dt + \epsilon_n^2\mathbb{P}(B_{k_{\mathrm{L}}}^c) \nonumber \\
	& \hspace{6pt} \leq \frac{192}{k_{\mathrm{L}}(x_0)^{3/2}}  \int_0^\infty \frac{u}{1+\frac{1}{64}\epsilon_0^3 M_0^3 u^3} \,du \nonumber \\
	& \hspace{36pt} + \frac{128}{ \epsilon_0^2 M_0^2} \sup_{|t| \in (t_n,\epsilon_n)} \mathrm{Var}\{ \hat{\mu}_n(x_0^t,X^n) \} \log \Bigl( \frac{\epsilon_n}{t_n} \Bigr) + \epsilon_n^2\mathbb{P}(B_{k_{\mathrm{L}}}^c) \nonumber
\\ &  \hspace{6pt} = o \Bigl( \frac{1}{k_\mathrm{L}(x_0)} \Bigr)
\end{align}
uniformly for $P \in \mathcal{P}_{d,\theta}$, $k_\mathrm{L} \in K_{\beta,\tau}$ and $x_0 \in \mathcal{S}_n$. We conclude from~\eqref{eq:R61} and~\eqref{eq:tnen} that $|R_5|=o(\gamma_n(k_\mathrm{L}))$, uniformly for $P \in \mathcal{P}_{d,\theta}$ and $k_\mathrm{L} \in K_{\beta,\tau}$.

\medskip

\textit{To bound $R_6$}: Let $\theta(x_0^t) := -2k_\mathrm{L}(x_0^t)^{1/2}\{\mu_n(x_0^t)-1/2\}$. Write
\[
	R_6:= \int_{\mathcal{S}_n} R_6(x_0) \, d\mathrm{Vol}^{d-1}(x_0) = R_{61} + R_{62},
\]
where
\[
	R_{61}:= \int_{\mathcal{S}_n} \int_{-\epsilon_{n}}^{\epsilon_{n}}  t\|\dot{\psi}(x_0)\| \Bigl[ \mathbb{E}\Phi\bigl(\hat\theta(x_0^t)\bigr) - \Phi\bigl(\theta(x_0^t)\bigr) \Bigr] \,dt \, d\mathrm{Vol}^{d-1}(x_0)
\]
and
\[
	R_{62}:= \int_{\mathcal{S}_n} \int_{-\epsilon_{n}}^{\epsilon_{n}}  t\|\dot{\psi}(x_0)\| \Bigl[\Phi\bigl(\theta(x_0^t)\bigr) - \Phi\bigl(\bar{\theta}(x_0,t)\bigr) \Bigr] \,dt \, d\mathrm{Vol}^{d-1}(x_0).
\]

\bigskip

\textit{To bound $R_{61}$}: We again deal with the regions $|t|\leq t_n$ and $|t| \in (t_n, \epsilon_n)$ separately. First let $\tilde{\theta}(x_0^t):= -2k_\mathrm{L}(x_0^t)^{1/2}\{ \hat{\mu}_n(x_0^t,X^n)-1/2\}$.  Writing $\phi$ for the standard normal density, and using the facts that $|\hat{\theta}(x_0^t)| \geq |\tilde{\theta}(x_0^t)|$, that $\hat{\theta}(x_0^t)$ and $\tilde{\theta}(x_0^t)$ have the same sign, and that $|x\phi(x)| \leq 1$, we have
\begin{align*}
	&\biggl| \int_{-t_n}^{t_n} t \bigl[ \mathbb{E}\Phi\bigl(\hat\theta(x_0^t)\bigr) - \Phi\bigl(\theta(x_0^t)\bigr) \bigr] \,dt \biggr| \\
&\leq \int_{-t_n}^{t_n} |t| \mathbb{E} \Bigl\{|\hat{\theta}(x_0^t)-\tilde{\theta}(x_0^t)|\phi \bigl( \tilde{\theta}(x_0^t) \bigr)\mathbbm{1}_{A_{k_{\mathrm{L}}}} + |\tilde{\theta}(x_0^t)-\theta(x_0^t)|\Bigr\} \,dt + t_n^2\mathbb{P}(A_{k_{\mathrm{L}}}^c) \\
	&\leq \int_{-t_n}^{t_n} |t| \biggl[\mathbb{E} \biggl\{\mathbbm{1}_{A_{k_{\mathrm{L}}}}  \Bigl| \frac{1}{2k_\mathrm{L}(x_0^t)^{1/2} \hat{\sigma}_n(x_0^t,X^n)} -1 \Bigr|\biggr\} 
\\ & \hspace{80 pt} + 2 k_\mathrm{L}(x_0^t)^{1/2}  \mathrm{Var}^{1/2} \{ \hat{\mu}_n(x_0^t,X^n) \}\biggr] \,dt + t_n^2\mathbb{P}(A_{k_{\mathrm{L}}}^c) = o(t_n^2)
\end{align*}
uniformly for $P \in \mathcal{P}_{d,\theta}$, $k_\mathrm{L} \in K_{\beta,\tau}$ and $x_0 \in \mathcal{S}_n$.  Note that for $|t| \in (t_n,\epsilon_n)$ and $x_0 \in \mathcal{S}_n$, we have when $\epsilon_n < \epsilon_0$ and $n \geq n_5$ that
\begin{align}
\label{Eq:hattheta}
	\mathbb{E}\bigl\{\mathbbm{1}_{A_{k_{\mathrm{L}}} \cap B_{k_{\mathrm{L}}}}&\bigl| \hat{\theta}(x_0^t)-\theta(x_0^t) \bigr|\bigr\} \nonumber
\\ &\leq \mathbb{E} \biggl\{ \frac{\mathbbm{1}_{A_{k_{\mathrm{L}}} \cap B_{k_{\mathrm{L}}}}}{\hat{\sigma}_n(x_0^t,X^n)}|\hat{\mu}_n(x_0^t,X^n)-\mu_n(x_0^t)| \nonumber
\\ &\hspace{60 pt} + \mathbbm{1}_{A_{k_{\mathrm{L}}} \cap B_{k_{\mathrm{L}}}}|\theta(x_0^t)|\Bigl| \frac{1}{2k_\mathrm{L}(x_0^t)^{1/2} \hat{\sigma}_n(x_0^t,X^n)} -1 \Bigr| \biggr\} \nonumber\\
&\leq 3k_\mathrm{L}(x_0)^{1/2}  \mathrm{Var}^{1/2}\{ \hat{\mu}_n(x_0^t,X^n) \} \nonumber \\
&\hspace{8pt}+ \frac{5}{2}k_\mathrm{L}(x_0)^{1/2}M_0 |t| \mathbb{E} \biggl\{\mathbbm{1}_{A_{k_{\mathrm{L}}} \cap B_{k_{\mathrm{L}}}}\Bigl|\frac{1}{2k_\mathrm{L}(x_0^t)^{1/2} \hat{\sigma}_n(x_0^t,X^n)} -1 \Bigr|\biggr\}.
\end{align}
Thus by~\eqref{Eq:mun},~\eqref{eq:thetaprob},~\eqref{Eq:hattheta} and Step~3, for $\epsilon_n < \epsilon_0$ and $n \geq n_5$, 
\begin{align}
\label{Eq:R711}
	&\int_{|t| \in (t_n, \epsilon_n)} |t| \bigl| \mathbb{E}\Phi\bigl(\hat\theta(x_0^t)\bigr) - \Phi\bigl(\theta(x_0^t)\bigr) \bigr| \,dt \nonumber \\
	&\leq \int_{|t| \in (t_n, \epsilon_n)} |t| \mathbb{E} \Bigl\{\mathbbm{1}_{A_{k_{\mathrm{L}}} \cap B_{k_{\mathrm{L}}}}\bigl| \hat{\theta}(x_0^t)- \theta(x_0^t) \bigr|\Bigr\} \phi \Bigl(\frac{1}{4} \epsilon_0 M_0 k_\mathrm{L}^{1/2}(x_0) |t|\Bigr)  \,dt \nonumber
\\ & \hspace{30 pt} + \mathbb{P}(A_{k_{\mathrm{L}}}^c \cup B_{k_{\mathrm{L}}}^c) + \frac{128}{ \epsilon_0^2M_0^2} \sup_{|t| \in (t_n,\epsilon_n)} \mathrm{Var}\{ \hat{\mu}_n(x_0^t,X^n) \} \log \Bigl( \frac{\epsilon_n}{t_n} \Bigr) \nonumber
\\ &  = o \Bigl( \frac{1}{k_\mathrm{L}(x_0)} \Bigr)
\end{align}
uniformly for $P \in \mathcal{P}_{d,\theta}$, $k_\mathrm{L} \in K_{\beta,\tau}$ and $x_0 \in \mathcal{S}_n$.

\medskip
 
\textit{To bound $R_{62}$}: Let
\[
u(x) \equiv u_n(x) := k_\mathrm{L}(x)^{1/2}\Bigl( \frac{k_\mathrm{L}(x)}{n \bar{f}(x)} \Bigr)^{2/d}.
\]
Given $\epsilon > 0$ small enough that $\epsilon^2 + \frac{\epsilon}{2\epsilon_0} < 1/2$, by Step~1 there exists $n_6 \in \mathbb{N}$ such that for $n \geq n_6$, $P \in \mathcal{P}_{d,\theta}$, $k_{\mathrm{L}} \in K_{\beta,\tau}$, $x_0 \in \mathcal{S}_n$ and $|t| < \epsilon_n$,  
\[
\bigl|\theta(x_0^t) - \bar{\theta}(x_0,t)\bigr| \leq \epsilon^2 \bigl\{|t|k_{\mathrm{L}}(x_0)^{1/2} + u(x_0)\ell\bigl(\bar{f}(x_0)\bigr)\bigr\}.
\]
By decreasing $\epsilon$ and increasing $n_6$ if necessary, it follows that
\begin{align*}
& \bigl|\Phi\bigl(\theta(x_0^t)\bigr) - \Phi\bigl(\bar{\theta}(x_0,t)\bigr)\bigr| \leq \epsilon^2 \bigl\{|t|k_{\mathrm{L}}(x_0)^{1/2} + u(x_0)\ell\bigl(\bar{f}(x_0)\bigr) \bigr\} \phi\Bigl(\frac{1}{2}\bar{\theta}(x_0,t)\Bigr),
\end{align*}
for all $n \geq n_6$, $P \in \mathcal{P}_{d,\theta}$, $k_{\mathrm{L}} \in K_{\beta,\tau}$, $x_0 \in \mathcal{S}_n$ and $t \in (-\epsilon_n,\epsilon_n)$ satisfying $2\epsilon u(x_0)\ell\bigl(\bar{f}(x_0)\bigr)\|\dot{\eta}(x_0)\| \leq |\bar{\theta}(x_0,t)|$.   Substituting $u = \bar{\theta}(x_0,t)/2$, it follows that there exists $C^* > 0$ such that for all $n \geq n_6$, $P \in \mathcal{P}_{d,\theta}$ and $k_{\mathrm{L}} \in K_{\beta,\tau}$,
\begin{align}
\label{Eq:R721} 
&|R_{62}| \nonumber \\
&\leq  \int_{\mathcal{S}_n} \int_{|u| \leq \epsilon u(x_0)\ell(\bar{f}(x_0))\|\dot{\eta}(x_0)\|} \frac{2\bar{f}(x_0)}{\|\dot{\eta}(x_0)\|k_{\mathrm{L}}(x_0)}|u + u(x_0)a(x_0)| \, du \, d\mathrm{Vol}^{d-1}(x_0) \nonumber
\\ &\hspace{60 pt} + \int_{\mathcal{S}_n} \int_{-\infty}^{\infty} \frac{2\bar{f}(x_0)|u + u(x_0)a(x_0)|}{\|\dot{\eta}(x_0)\|^2 k_{\mathrm{L}}(x_0)}\bigl\{\epsilon^2|u + u(x_0)a(x_0)| \nonumber
\\ & \hspace{150 pt} + \epsilon |u| \bigr\} \phi(u) \,du \, d\mathrm{Vol}^{d-1}(x_0)  \leq C^* \epsilon \gamma_n(k_\mathrm{L}).
\end{align}
The combination of~\eqref{Eq:R711} and~\eqref{Eq:R721} yields the desired error bound on $|R_6|$ in~\eqref{Eq:R6R7}, uniformly for $P \in \mathcal{P}_{d,\theta}$, $k_{\mathrm{L}} \in K_{\beta,\tau}$, and therefore completes the proof.
\end{proof}

\section{Empirical analysis}
\label{sec:adaptsims}
In this section, we compare the $k_{\mathrm{O}}$nn and $k_{\mathrm{SS}}$nn classifiers, introduced in Section~\ref{sec:semisup} of the main text, with the standard $k$nn classifier studied in Section~\ref{sec:mainstandard} of the main text.  We investigate three settings that reflect the differences between the main results in these sections.
\begin{itemize}
\item Setting 1: $P_1$ is the distribution of $d$ independent $N(0,1)$ components; whereas $P_0$ is the distribution of $d$ independent $N(1,1/4)$ components.
\item Setting 2: $P_1$ is the distribution of $d$ independent $t_{5}$ components; $P_0$ is the distribution of $d$ independent components, the first $\lfloor d/2 \rfloor$ having a $t_{5}$ distribution and the remainder having a $N(1,1)$ distribution. 
\item Setting 3: $P_1$ is the distribution of $d$ independent standard Cauchy components; $P_0$ is the distribution of $d$ independent components, the first $\lfloor d/2 \rfloor$ being standard Cauchy and the remainder standard normal.
\end{itemize} 

The corresponding marginal distribution $P_X$ in Setting 1 satisfies \textbf{(A.4)} for every $\rho > 0$.  Hence, for the standard $k$-nearest neighbour classifier when $d \geq 5$, we are in the setting of Theorem~\ref{thm:exriskRdstandard}(i), while for $d \leq 4$, we can only appeal to Theorem~\ref{thm:exriskRdstandard}(ii).  On the other hand, for the local-$k$-nearest neighbour classifiers, the results of Theorems~\ref{thm:exriskRdOadapt}(i) and~\ref{thm:exriskRdSSadapt}(i) apply for all dimensions, and we can expect the excess risk to converge to zero at rate $O(n^{-4/(d+4)})$.  In Setting~2, \textbf{(A.4)} holds for $\rho < 5$, but not for $\rho \geq 5$.  Thus, for the standard $k$-nearest neighbour classifier, we are in the setting of Theorem~\ref{thm:exriskRdstandard}(ii) for $d < 20$, whereas Theorems~\ref{thm:exriskRdOadapt}(i) and~\ref{thm:exriskRdSSadapt}(i) again apply for all dimensions for the local classifiers.  Finally, in Setting~3, \textbf{(A.4)} does not hold for any $\rho \geq 1$, and only the conditions of Theorems~\ref{thm:exriskRdstandard}(ii),~\ref{thm:exriskRdOadapt}(ii) and~\ref{thm:exriskRdSSadapt}(ii) apply.

\begin{table}[ht!]
  \centering
  \caption{\label{tab:1}Misclassification rates for Settings 1, 2 and 3.  In the final two columns we present the regret ratios given in~\eqref{eq:regretratio} (with standard errors calculated via the delta method).}
    \begin{tabular}{c c r r r r r r}
\hline
 $d$&Bayes risk& $n$&$\hat{k}$nn risk&$\hat{k}_{\mathrm{O}}$nn risk&$\hat{k}_{\mathrm{SS}}$nn risk&O RR& SS RR\\
\hline
   \hline
\multicolumn{3}{l}{Setting 1}\\
1&22.67&50&${26.85}_{0.13}$&${25.91}_{0.12}$&${25.98}_{0.13}$&${0.78}_{0.022}$&${0.79}_{0.023}$\\
&         &200&${24.07}_{0.06}$&${23.52}_{0.06}$&${23.48}_{0.05}$&${0.61}_{0.030}$&${0.58}_{0.029}$\\
&       &1000&${23.20}_{0.04}$&${22.93}_{0.04}$&${22.94}_{0.04}$&${0.48}_{0.048}$&${0.50}_{0.048}$\\
2&13.30&50&${17.70}_{0.09}$&${16.96}_{0.08}$&${16.95}_{0.08}$&${0.83}_{0.015}$&${0.83}_{0.015}$\\
&        &200&${15.09}_{0.05}$&${14.69}_{0.04}$&${14.74}_{0.05}$&${0.77}_{0.018}$&${0.80}_{0.019}$\\
&      &1000&${14.04}_{0.04}$&${13.78}_{0.03}$&${13.80}_{0.03}$&${0.65}_{0.025}$&${0.67}_{0.025}$\\
5& 3.53&50&$ {9.46}_{0.07}$&${8.95}_{0.06}$&${8.94}_{0.06}$&${0.91}_{0.006}$&${0.91}_{0.006}$\\
&        &200&$ {6.94}_{0.03}$&${6.67}_{0.03}$&${6.70}_{0.03}$&${0.92}_{0.006}$&${0.93}_{0.007}$\\
&      &1000&$ {5.49}_{0.02}$&${5.18}_{0.02}$&${5.23}_{0.02}$&${0.84}_{0.008}$&${0.87}_{0.008}$\\
   \hline
\multicolumn{3}{l}{Setting 2}\\
1&31.16&50&${36.55}_{0.14}$&${36.07}_{0.14}$&${35.93}_{0.14}$&${0.91}_{0.020}$&${0.88}_{0.020}$\\
&         &200&${32.93}_{0.08}$&${32.38}_{0.07}$&${32.42}_{0.07}$&${0.69}_{0.031}$&${0.71}_{0.032}$\\
&       &1000&${31.62}_{0.05}$&${31.37}_{0.05}$&${31.37}_{0.05}$&${0.46}_{0.065}$&${0.47}_{0.066}$\\
2&31.15&50&${37.79}_{0.13}$&${38.02}_{0.12}$&${37.90}_{0.12}$&${1.02}_{0.014}$&${1.01}_{0.015}$\\
         &&200&${33.64}_{0.08}$&${33.63}_{0.07}$&${33.54}_{0.07}$&${1.00}_{0.028}$&${0.96}_{0.026}$\\
 &      &1000&${31.83}_{0.05}$&${31.81}_{0.05}$&${31.80}_{0.05}$&${0.97}_{0.039}$&${0.95}_{0.038}$\\
5&20.10&50&${28.74}_{0.12}$&${29.16}_{0.12}$&${29.13}_{0.11}$&${1.05}_{0.011}$&${1.05}_{0.011}$\\
&         &200&${23.60}_{0.06}$&${23.75}_{0.06}$&${23.93}_{0.06}$&${1.04}_{0.014}$&${1.09}_{0.015}$\\
&       &1000&${21.86}_{0.04}$&${21.71}_{0.04}$&${21.77}_{0.04}$&${0.91}_{0.014}$&${0.95}_{0.014}$\\
  \hline
\multicolumn{3}{l}{Setting 3}\\
1&37.44&50&${44.76}_{0.10}$&${43.09}_{0.12}$&${43.08}_{0.12}$&${0.77}_{0.013}$&${0.77}_{0.013}$\\
&         &200&${41.86}_{0.08}$&${40.18}_{0.09}$&${40.23}_{0.09}$&${0.62}_{0.017}$&${0.63}_{0.017}$\\
&       &1000&${38.68}_{0.06}$&${37.85}_{0.05}$&${37.89}_{0.05}$&${0.33}_{0.033}$&${0.36}_{0.032}$\\
2&37.45&50&${46.20}_{0.09}$&${44.81}_{0.10}$&${45.24}_{0.10}$&${0.84}_{0.009}$&${0.89}_{0.009}$\\
&         &200&${43.50}_{0.07}$&${42.29}_{0.08}$&${42.86}_{0.08}$&${0.80}_{0.011}$&${0.89}_{0.011}$\\
     &&1000&${40.53}_{0.06}$&${39.64}_{0.06}$&${39.96}_{0.06}$&${0.71}_{0.013}$&${0.82}_{0.014}$\\
5&23.23&50&${41.56}_{0.11}$&${38.13}_{0.11}$&${39.26}_{0.12}$&${0.81}_{0.005}$&${0.87}_{0.005}$\\
         &&200&${36.02}_{0.07}$&${33.34}_{0.06}$&${34.68}_{0.07}$&${0.79}_{0.004}$&${0.90}_{0.004}$\\
       &&1000&${31.46}_{0.05}$&${29.91}_{0.05}$&${30.58}_{0.05}$&${0.81}_{0.004}$&${0.89}_{0.004}$\\
  \end{tabular}
\end{table}

For the standard $k$nn classifier, we use 5-fold cross validation to choose $k$, based on a sequence of equally-spaced values between 1 and $\lfloor n/4 \rfloor$ of length at most 40.  For the oracle classifier, we set 
\[
\hat{k}_{\mathrm{O}}(x) : = \max\Bigl[1, \min\bigl[\lfloor \hat{B}_{\mathrm{O}}\{\bar{f}(x)n/\|\bar{f}\|_{\infty}\}^{4/(d+4)} \rfloor, n/2\bigr]\Bigr],
\]
where $\hat{B}_{\mathrm{O}}$ was again chosen via 5-fold cross validation, but based on a sequence of 40 equally-spaced points between $n^{-4/(d+4)}$ (corresponding to the 1-nearest neighbour classifier) and $n^{d/(d+4)}$.  Similarly, for the semi-supervised classifier, we set 
\[
\hat{k}_{\mathrm{SS}}(x) : = \max\Bigl[1, \min\bigl[\lfloor \hat{B}_{\mathrm{SS}} \{\hat{f}_m(x)n/\|\hat{f}_m\|_{\infty}\}^{4/(d+4)} \rfloor, n/2\bigr] \Bigr],
\]
where $\hat{B}_{\mathrm{SS}}$ was chosen analogously to $\hat{B}_{\mathrm{O}}$, and where $\hat{f}_{m}$ is the $d$-dimensional kernel density estimator constructed using a truncated normal kernel and bandwidths chosen via the default method in the \texttt{R} package \texttt{ks} \citep{Duong:15}.   In practice, we estimated $\|\hat{f}_m\|_{\infty}$ by the maximum value attained on the unlabelled training set. 

In each of the three settings above, we generated a training set of size $n \in \{50,200,1000\}$ in dimensions $d \in \{1,2,5\}$, an unlabelled training set of size 1000, and a test set of size 1000.  In Table~\ref{tab:1}, we present the sample mean and standard error (in subscript) of the risks computed from 1000 repetitions of each experiment. Further, we present estimates of the regret ratios, given by
\begin{equation}
\label{eq:regretratio}
\frac{R(\hat{C}_n^{\hat{k}_{\mathrm{O}}\mathrm{nn}}) - R(C^{\mathrm{Bayes}})}{R(\hat{C}_n^{\hat{k}\mathrm{nn}}) - R(C^{\mathrm{Bayes}})} \ \ \mathrm{and} \ \ \frac{R(\hat{C}_n^{\hat{k}_{\mathrm{SS}}\mathrm{nn}}) - R(C^{\mathrm{Bayes}})}{R(\hat{C}_n^{\hat{k}\mathrm{nn}}) - R(C^{\mathrm{Bayes}})},
\end{equation}
for which the standard errors given are estimated via the delta method.  From Table~\ref{tab:1}, we saw improvement in performance from the oracle and semi-supervised classifiers in 22 of the 27 experiments, comparable performance in three experiments, and there were two where the standard $k$nn classifier was the best of the three classifiers considered.  In those latter two cases, the theoretical improvement expected for the local classifiers is small; for instance, when $d=5$ in Setting~2, the excess risk for the local classifiers converges at rate $O(n^{-4/9})$, while the standard $k$-nearest neighbour classifier can attain a rate at least as fast as $o(n^{-1/3+\epsilon})$ for every $\epsilon > 0$.  It is therefore perhaps unsurprising that we require the larger sample size of $n=1000$ for the local classifiers to yield an improvement in this case.  The semi-supervised classifier exhibits similar performance to the oracle classifier in all settings, though some deterioration is noticeable in higher dimensions, where it is harder to construct a good estimate of $\bar{f}$ from the unlabelled training data.

\section{An introduction to differential geometry, tubular neighbourhoods and integration on manifolds}
\label{Sec:ManifoldIntro}
The purpose of this section is to give a brief introduction to the ideas from differential geometry, specifically tubular neighbourhoods and integration on manifolds, which play an important role in our analysis of misclassification error rates, but which we expect are unfamiliar to many statisticians.  For further details and several of the proofs, we refer the reader to the many excellent texts on these topics, e.g. \citet{GuilleminPollack1974}, \citet{Gray:04}.

\subsection{Manifolds and regular values}
\label{Sec:Regular}

Recall that if $\mathcal{X}$ is an arbitrary subset of $\mathbb{R}^M$, we say $\phi:\mathcal{X} \rightarrow \mathbb{R}^N$ is \emph{differentiable} if for each $x \in \mathcal{X}$, there exists an open subset $U \subseteq \mathbb{R}^M$ containing $x$ and a differentiable function $F:U \rightarrow \mathbb{R}^N$ such that $F(z) = \phi(z)$ for $z \in U \cap \mathcal{X}$.  If $\mathcal{Y}$ is also a subset of $\mathbb{R}^M$, we say $\phi:\mathcal{X} \rightarrow \mathcal{Y}$ is a \emph{diffeomorphism} if $\phi$ is bijective and differentiable and if its inverse $\phi^{-1}$ is also differentiable.  We then say $\mathcal{S} \subseteq \mathbb{R}^d$ is an \emph{$m$-dimensional manifold} if for each $x \in \mathcal{S}$, there exist an open subset $U_x \subseteq \mathbb{R}^m$, a neighbourhood $V_x$ of $x$ in $\mathcal{S}$ and a diffeomorphism $\phi_x: U_x \rightarrow V_x$.  Such a diffeomorphism $\phi_x$ is called a \emph{local parametrisation} of $\mathcal{S}$ around $x$, and we sometimes suppress the dependence of $\phi_x, U_x$ and $V_x$ on $x$.  It turns out that the specific choice of local parametrisation is usually not important, and properties of the manifold are well-defined regardless of the choice made.

Let $\mathcal{S} \subseteq \mathbb{R}^d$ be an $m$-dimensional manifold and let $\phi:U \rightarrow \mathcal{S}$ be a local parametrisation of $\mathcal{S}$ around $x \in \mathcal{S}$, where $U$ is an open subset of $\mathbb{R}^m$.  Assume that $\phi(0) = x$ for convenience.  The \emph{tangent space} $T_x(\mathcal{S})$ to $\mathcal{S}$ at $x$ is defined to be the image of the derivative $D\phi_0:\mathbb{R}^m \rightarrow \mathbb{R}^d$ of $\phi$ at $0$.  Thus $T_x(\mathcal{S})$ is the $m$-dimensional subspace of $\mathbb{R}^d$ whose parallel translate $x + T_x(\mathcal{S})$ is the best affine approximation to $\mathcal{S}$ through $x$, and $(D\phi_0)^{-1}$ is well-defined as a map from $T_x(\mathcal{S})$ to $\mathbb{R}^m$.  If $f:\mathcal{S} \rightarrow \mathbb{R}$ is differentiable, we define the derivative $Df_x:T_x(\mathcal{S}) \rightarrow \mathbb{R}$ of $f$ at $x$ by $Df_x := Dh_0 \circ (D\phi_0)^{-1}$, where $h := f \circ \phi$.

In practice, it is usually rather inefficient to define manifolds through explicit diffeomorphisms.  Instead, we can often obtain them as level sets of differentiable functions.  Suppose that $\mathcal{R} \subseteq \mathbb{R}^d$ is a manifold and $\eta:\mathcal{R} \rightarrow \mathbb{R}$ is differentiable.  We say $y \in \mathbb{R}$ is a \emph{regular value} for $\eta$ if $\mathrm{image}(D\eta_x) = \mathbb{R}$ for every $x \in \mathcal{R}$ for which $\eta(x) = y$.  If $y \in \mathbb{R}$ is a regular value of $\eta$, then $\eta^{-1}(y)$ is a $(d-1)$-dimensional submanifold of $\mathcal{R}$ \citep[][p.~21]{GuilleminPollack1974}.

\subsection{Tubular neighbourhoods of level sets}
\label{Sec:Tubular}

For any set $\mathcal{S} \subseteq \mathbb{R}^d$ and $\epsilon > 0$, we call $\mathcal{S} + \epsilon B_1(0)$ the \emph{$\epsilon$-neighbourhood} of $\mathcal{S}$.  In circumstances where $\mathcal{S}$ is a $(d-1)$-dimensional manifold defined by the level set of a continuously differentiable function $\eta:\mathbb{R}^d \rightarrow \mathbb{R}$ with non-vanishing derivative on $\mathcal{S}$, the set $\mathcal{S}^\epsilon$ is often called a \emph{tubular neighbourhood}, and $\dot{\eta}(x)^Tv = 0$ for all $x \in \mathcal{S}$ and $v \in T_x(\mathcal{S})$.  We therefore have the following useful representation of the $\epsilon$-neighbourhood of $\mathcal{S}$ in terms of points on $\mathcal{S}$ and a perturbation in a normal direction.
\begin{proposition}
\label{Prop:DG1}
Let $\eta : \mathbb{R}^d \rightarrow [0,1]$, suppose that $\mathcal{S}:= \{x \in \mathbb{R}^d : \eta(x) =1/2\}$ is non-empty, and suppose further that $\eta$ is continuously differentiable on $\mathcal{S} + \epsilon B_1(0)$ for some $\epsilon > 0$, with $\dot{\eta}(x) \neq 0$ for all $x \in \mathcal{S}$, so that $\mathcal{S}$ is a $(d-1)$-dimensional manifold. Then 
\[
	\mathcal{S} + \epsilon B_1(0) = \Bigl\{x_0+\frac{t \dot{\eta}(x_0)}{\|\dot{\eta}(x_0)\|} : x_0 \in \mathcal{S}, |t|<\epsilon \Bigr\} =: \mathcal{S}^\epsilon.
\]
\end{proposition}
\begin{proof}
For any $x_0 \in \mathcal{S}$ and $|t| < \epsilon$, we have $x_0 + t\dot{\eta}(x_0)/\|\dot{\eta}(x_0)\| \in \mathcal{S} + \epsilon B_1(0)$.  On the other hand, suppose that $x \in \mathcal{S} + \epsilon B_1(0)$. Since $\mathcal{S}$ is closed, there exists $x_0 \in \mathcal{S}$ such that $\|x-x_0\| \leq \|x-y\|$ for all $y \in \mathcal{S}$. Rearranging this inequality yields that, for $y \neq x_0$,
\begin{equation}
\label{Eq:nearlyperp}
	2(x-x_0)^T \frac{(y-x_0)}{\|y-x_0\|} \leq \|y-x_0\|.
\end{equation}
Let $U$ be an open subset of $\mathbb{R}^{d-1}$ and $\phi:U \rightarrow \mathcal{S}$ be a local parametrisation of $\mathcal{S}$ around $x_0$, where without loss of generality we assume $\phi(0) = x_0$.  Let $v \in T_{x_0}(\mathcal{S}) \setminus \{0\}$ be given and let $h \in \mathbb{R}^{d-1} \setminus \{0\}$ be such that $D\phi_0(h)=v$. Then for $t>0$ sufficiently small we have $th \in U$, so by~\eqref{Eq:nearlyperp},
\[
	2(x-x_0)^T \frac{\{\phi(th)-\phi(0)\}}{\|\phi(th)-\phi(0)\|} \leq \|\phi(th)-\phi(0)\|.
\]
Letting $t \searrow 0$ we see that $(x-x_0)^T v \leq 0$. Since $v \in T_{x_0}(\mathcal{S}) \setminus \{0\}$ was arbitrary and $-v \in T_{x_0}(\mathcal{S}) \setminus \{0\}$, we therefore have that $(x-x_0)^Tv=0$ for all $v \in T_{x_0}(\mathcal{S})$.  Moreover, $\dot{\eta}(x_0)^Tv = 0$ for all $v \in T_{x_0}(\mathcal{S})$, so $x-x_0 \propto \dot{\eta}(x_0)$, which yields the result.
\end{proof}
In fact, under a slightly stronger condition on $\eta$, we have the following useful result:
\begin{proposition}
\label{Prop:DG2}
Let $\mathcal{R}$ be a $d$-dimensional manifold in $\mathbb{R}^d$, suppose that $\eta : \mathcal{R} \rightarrow [0,1]$ satisfies the condition that $\mathcal{S}:= \{x \in \mathcal{R} : \eta(x) =1/2\}$ is non-empty.  Suppose further that there exists $\epsilon > 0$ such that $\eta$ is twice continuously differentiable on $\mathcal{S}^\epsilon$.  Assume that $\dot{\eta}(x_0) \neq 0$ for all $x_0 \in \mathcal{S}$.  Define $g:\mathcal{S} \times (-\epsilon,\epsilon) \rightarrow \mathcal{S}^\epsilon$ by
\[
g(x_0,t) := x_0 + \frac{t\dot{\eta}(x_0)}{\|\dot{\eta}(x_0)\|}.
\]
If
\begin{equation}
\label{Eq:epsAssump}
\epsilon \leq \inf_{x_0 \in \mathcal{S}} \frac{\|\dot{\eta}(x_0)\|}{\sup_{z \in B_{2\epsilon}(x_0) \cap \mathcal{S}^\epsilon}\|\ddot{\eta}(z)\|_{\mathrm{op}}},
\end{equation}
then $g$ is injective.  In fact $g$ is a diffeomorphism, with
\begin{equation}
\label{Eq:Dg}
Dg_{(x_0,t)}(v_1,v_2) = (I + tB)\biggl(v_1 + \frac{\dot{\eta}(x_0)}{\|\dot{\eta}(x_0)\|}v_2\biggr),
\end{equation}
for $v_1 \in T_{x_0}(\mathcal{S})$ and $v_2 \in \mathbb{R}$, where 
\begin{equation}
\label{Eq:B}
B := \frac{1}{\|\dot{\eta}(x_0)\|}\biggl(I - \frac{\dot{\eta}(x_0)\dot{\eta}(x_0)^T}{\|\dot{\eta}(x_0)\|^2}\biggr)\ddot{\eta}(x_0).
\end{equation}
\end{proposition}
\begin{proof}
\ \ Assume for a contradiction that there exist distinct points $x_1, x_2 \in \mathcal{S}$ and $t_1,t_2 \in (-\epsilon,\epsilon)$ with $|t_1| \geq |t_2|$ such that
\[
x_1 + \frac{t_1\dot{\eta}(x_1)}{\|\dot{\eta}(x_1)\|} = x_2 + \frac{t_2\dot{\eta}(x_2)}{\|\dot{\eta}(x_2)\|}.
\]
Then
\begin{equation}
\label{Eq:x2x1}
0 < \|x_2-x_1\|^2 = \frac{2t_1\dot{\eta}(x_1)^T(x_2-x_1)}{\|\dot{\eta}(x_1)\|} + t_2^2 - t_1^2 \leq \frac{2t_1\dot{\eta}(x_1)^T(x_2-x_1)}{\|\dot{\eta}(x_1)\|}.
\end{equation}
By Taylor's theorem and~\eqref{Eq:x2x1},
\begin{align*}
\label{Eq:x2x12}
|\dot{\eta}(x_1)^T(x_2-x_1)| &= |\eta(x_2) - \eta(x_1) - \dot{\eta}(x_1)^T(x_2-x_1)| \\
& \leq \frac{1}{2}\sup_{z \in B_{2\epsilon}(x_1) \cap \mathcal{S}^\epsilon}\|\ddot{\eta}(z)\|_{\mathrm{op}}\|x_2-x_1\|^2 \\
&< \sup_{z \in B_{2\epsilon}(x_1) \cap \mathcal{S}^\epsilon}\|\ddot{\eta}(z)\|_{\mathrm{op}}\frac{\epsilon|\dot{\eta}(x_1)^T(x_2-x_1)|}{\|\dot{\eta}(x_1)\|},
\end{align*}
contradicting the hypothesis~\eqref{Eq:epsAssump}.

To show that $g$ is a diffeomorphism, let $x_0 \in \mathcal{S}$ be given and let $\phi:U \rightarrow \mathcal{S}$ be a local parametrisation around $x_0$ with $\phi(0) = x_0$.  Define $\Phi:U \times (-\epsilon,\epsilon) \rightarrow \mathcal{S} \times (-\epsilon,\epsilon)$ by $\Phi(u,t) := (\phi(u),t)$, and $H: U \times (-\epsilon,\epsilon) \rightarrow \mathcal{S}^\epsilon$ by $H := g \circ \Phi$.  Finally, define the \emph{Gauss map} $n: \mathcal{S} \rightarrow \mathbb{R}^d$ by $n(x_0) := \dot{\eta}(x_0)/\|\dot{\eta}(x_0)\|$.  Then, for $h = (h_1^T,h_2)^T \in \mathbb{R}^{d-1} \times \mathbb{R}$ and $s \in \mathbb{R} \setminus \{0\}$, 
\begin{align*}
&\lim_{s \rightarrow 0} \frac{H(sh_1,t + sh_2) - H(0,t)}{s} \\
&\hspace{1.5cm}= \lim_{s \rightarrow 0} \biggl\{\frac{\phi(sh_1) - \phi(0)}{s} + \frac{t\{n(\phi(sh_1)) - n(\phi(0))\}}{s} + h_2n\bigl(\phi(sh_1)\bigr)\biggr\} \\
&\hspace{1.5cm}= D\phi_0(h_1) + tDn_{x_0} \circ D\phi_0(h_1) + h_2n(x_0) \\
&\hspace{1.5cm}= Dg_{(x_0,t)} \circ D\Phi_{(0,t)}(h_1,h_2),
\end{align*}
where $Dg_{(x_0,t)}:T_{x_0}(\mathcal{S}) \times \mathbb{R} \rightarrow \mathbb{R}^d$ is given in~\eqref{Eq:Dg}.  

To show that $Dg_{(x_0,t)}$ is invertible, note that for $v_1 \in T_{x_0}(\mathcal{S})$ and $|t| < \epsilon$,
\[
\frac{|t|}{\|\dot{\eta}(x_0)\|}\biggl\|\biggl(I - \frac{\dot{\eta}(x_0)\dot{\eta}(x_0)^T}{\|\dot{\eta}(x_0)\|^2}\biggr)\ddot{\eta}(x_0)v_1\biggr\| \leq \frac{|t|\|\ddot{\eta}(x_0)\|_{\mathrm{op}}}{\|\dot{\eta}(x_0)\|}\|v_1\| < \|v_1\|,
\]
where the final inequality follows from~\eqref{Eq:epsAssump}.  Then, since $v_1 + \frac{t}{\|\dot{\eta}(x_0)\|}\Bigl(I - \frac{\dot{\eta}(x_0)\dot{\eta}(x_0)^T}{\|\dot{\eta}(x_0)\|^2}\Bigr)\ddot{\eta}(x_0)v_1$ and $n(x_0)v_2$ are orthogonal, it follows that $Dg_{(x_0,t)}$ is indeed invertible.  The inverse function theorem \citep[e.g.][p.~13]{GuilleminPollack1974} then gives that $g$ is a local diffeomorphism, and moreover, by \citet[][Exercise~5, p.~18]{GuilleminPollack1974} and the fact that $g$ is bijective, we can conclude that $g$ is in fact a diffeomorphism.
\end{proof}

\subsection{Forms, pullbacks and integration on manifolds}
\label{Sec:Forms}

Let $V$ be a (real) vector space of dimension $m$.  We say $T:V^p \rightarrow \mathbb{R}$ is a \emph{$p$-tensor} on $V$ if it is $p$-linear, and write $\mathcal{F}^p(V^*)$ for the set of $p$-tensors on $V$.  If $T \in \mathcal{F}^p(V^*)$ and $S \in \mathcal{F}^q(V^*)$, we define their \emph{tensor product} $T \otimes S \in \mathcal{F}^{p+q}(V^*)$ by
\[
T \otimes S(v_1,\ldots,v_p,v_{p+1},\ldots,v_{p+q}) := T(v_1,\ldots,v_p)S(v_{p+1},\ldots,v_{p+q}).
\]
Let $S_p$ denote the set of permutations of $\{1,\ldots,p\}$.  If $\pi \in S_p$ and $T \in \mathcal{F}^p(V^*)$, we can define $T^\pi \in \mathcal{F}^p(V^*)$ by $T^\pi(v) := T(v_{\pi(1)},\ldots,v_{\pi(p)})$ for $v = (v_1,\ldots,v_p) \in V^p$.  We say $T$ is \emph{alternating} if $T^\sigma = -T$ for all transpositions $\sigma:\{1,\ldots,p\} \rightarrow \{1,\ldots,p\}$.  The set of alternating $p$-tensors on $V$, denoted $\Lambda^p(V^*)$, is a vector space of dimension $\binom{m}{p}$.  The function $\mathrm{Alt}:\mathcal{F}^p(V^*) \rightarrow \Lambda^p(V^*)$ is defined by
\[
\mathrm{Alt}(T) := \frac{1}{p!}\sum_{\pi \in S_p} (-1)^{\mathrm{sgn}(\pi)}T^\pi,
\]
where $\mathrm{sgn}(\pi)$ denotes the sign of the permutation $\pi$.  If $T \in \Lambda^p(V^*)$ and $S \in \Lambda^q(V^*)$, we define their \emph{wedge product} $T \wedge S \in \Lambda^{p+q}(V^*)$ by
\[
T \wedge S := \mathrm{Alt}(T \otimes S).
\] 
If $W$ is another (real) vector space and $A:V \rightarrow W$ is a linear map, we define the \emph{transpose} $A^\ast:\Lambda^p(W^*) \rightarrow \Lambda^p(V^*)$ of $A$ by
\[
A^\ast T(v_1,\ldots,v_p) := T(Av_1,\ldots,Av_p).
\]
Let $\mathcal{S}$ be a manifold.  A \emph{$p$-form} $\omega$ on $\mathcal{S}$ is a function which assigns to each $x \in \mathcal{S}$ an element $\omega(x) \in \Lambda^p(T_x(\mathcal{S})^*)$.  If $\omega$ is a $p$-form on $\mathcal{S}$ and $\theta$ is a $q$-form on $\mathcal{S}$, we can define their wedge product $\omega \wedge \theta$ by $(\omega \wedge \theta)(x) := \omega(x) \wedge \theta(x)$.  For $j=1,\ldots,m$, let $x_j:\mathbb{R}^m \rightarrow \mathbb{R}$ denote the coordinate function $x_j(y_1,\ldots,y_m) := y_j$.  These functions induce $1$-forms $dx_j$, given by $dx_j(x)(y_1,\ldots,y_m) = y_j$ (so $dx_j(x) = D(x_j)_x$ in our previous notation).  Letting $\mathcal{I} := \{(i_1,\ldots,i_p):1 \leq i_1 < \ldots < i_p \leq m\}$, for $I = (i_1,\ldots,i_p) \in \mathcal{I}$, we write
\[
dx_I := dx_{i_1} \wedge \ldots \wedge dx_{i_p}.
\]
It turns out \citep[][p.~163]{GuilleminPollack1974} that any $p$-form on an open subset $U$ of $\mathbb{R}^m$ can be uniquely expressed as 
\begin{equation}
\label{Eq:sumI}
\sum_{I \in \mathcal{I}} f_I \, dx_I, 
\end{equation}
where each $f_I$ is a real-valued function on $U$.

Recall that the set of all ordered bases of a vector space $V$ is partitioned into two equivalence classes, and an \emph{orientation} of $V$ is simply an assignment of a positive sign to one equivalence class and a negative sign to the other.  If $V$ and $W$ are oriented vector spaces in the sense that an orientation has been specified for each of them, then an isomorphism $A:V \rightarrow W$ always either preserves orientation in the sense that for any ordered basis $\beta$ of $V$, the ordered basis $A\beta$ has the same sign as $\beta$, or it reverses it.  We say an $m$-dimensional manifold $\mathcal{X}$ is \emph{orientable} if for every $x \in \mathcal{X}$, there exist an open subset $U$ of $\mathbb{R}^m$, a neighbourhood $V$ of $x$ in $\mathcal{X}$ and a diffeomorphism $\phi:U \rightarrow V$ such that $D\phi_u:\mathbb{R}^m \rightarrow T_x(\mathcal{X})$ preserves orientation for every $u \in U$.  A map like $\phi$ above whose derivative at every point preserves orientation is called an \emph{orientation-preserving} map.    

If $\mathcal{X}$ and $\mathcal{Y}$ are manifolds, $\omega$ is a $p$-form on $\mathcal{Y}$ and $\psi:\mathcal{X} \rightarrow \mathcal{Y}$ is differentiable, we define the \emph{pullback} $\psi^\ast\omega$ of $\omega$ by $\psi$ to be the $p$-form on $\mathcal{X}$ given by
\[
\psi^\ast\omega(x) := (D\psi_x)^\ast \omega\bigl(\psi(x)\bigr).
\]
If $V$ is an $p$-dimensional vector space and $A:V \rightarrow V$ is linear, then $A^\ast T = (\det A)T$ for all $T \in \Lambda^p(V)$ \citep[][p.~160]{GuilleminPollack1974}.

If $\omega$ is an $m$-form on an open subset $U$ of $\mathbb{R}^m$, then by~\eqref{Eq:sumI}, we can write $\omega = f \, dx_1 \wedge \ldots \wedge dx_m$.  If $\omega$ is an integrable form on $U$ (i.e.\ $f$ is an integrable function on $U$), we can define the integral of $\omega$ over $U$ by
\[
\int_U \omega := \int_U f(x_1,\ldots,x_m) \, dx_1\ldots dx_m,
\]
where the integral on the right-hand side is a usual Lebesgue integral.  Now let $\mathcal{S}$ be an $m$-dimensional orientable manifold that can be parametrised with a single chart, in the sense that there exists an open subset $U$ of $\mathbb{R}^m$ and an orientation-preserving diffeomorphism $\phi:U \rightarrow \mathcal{S}$.  Define the \emph{support} of an $m$-form $\omega$ on $\mathcal{S}$ to be the closure of $\{x \in \mathcal{S}:\omega(x) \neq 0\}$.  If $\omega$ is compactly supported, then its pullback $\phi^\ast \omega$ is a compactly supported $m$-form on $U$; moreover $\phi^\ast \omega$ is integrable, and we can define the integral over $\mathcal{S}$ of $\omega$ by
\begin{equation}
\label{Eq:IntOmega}
\int_{\mathcal{S}} \omega := \int_U \phi^\ast \omega.
\end{equation}
Alternatively, we can suppose that $\omega$ is non-negative and measurable in the sense that $\phi^\ast \omega = f \, dx_1 \wedge \ldots \wedge dx_m$, say, with $f$ non-negative and measurable on $U$.  In this case, we can also define the integral of $\omega$ over $\mathcal{S}$ via~\eqref{Eq:IntOmega}.  

More generally, integrals of forms over more complicated manifolds can be defined via partitions of unity.  Recall \citep[][p.~52]{GuilleminPollack1974} that if $\mathcal{X}$ is an arbitrary subset of $\mathbb{R}^M$, and $\{V_\alpha:\alpha \in A\}$ is a (relatively) open cover of $\mathcal{X}$, then there exists a sequence of real-valued, differentiable functions $(\rho_n)$ on $\mathcal{X}$, called a \emph{partition of unity} with respect to $\{V_\alpha:\alpha \in A\}$, with the following properties:
\begin{enumerate}
\item $\rho_n(x) \in [0,1]$ for all $n \in \mathbb{N}$;
\item Each $x \in \mathcal{X}$ has a neighbourhood on which all but finitely many functions $\rho_n$ are identically zero;
\item Each $\rho_n$ is identically zero except on some closed set contained in some $V_\alpha$;
\item $\sum_{n=1}^\infty \rho_n(x) = 1$ for all $x \in \mathcal{X}$.
\end{enumerate}
Now let $\mathcal{S} \subseteq \mathbb{R}^d$ be an $m$-dimensional, orientable manifold, so for each $x \in \mathcal{S}$, there exist an open subset $U_x$ of $\mathbb{R}^m$, a neighbourhood $V_x$ of $x$ in $\mathcal{S}$ and an orientation-preserving diffeomorphism $\phi_x:U_x \rightarrow V_x$.  If $\omega$ is a compactly supported $m$-form on $\mathcal{S}$ and $(\rho_n)$ denotes a partition of unity on $\mathcal{S}$ with respect to $\{V_x:x \in \mathcal{S}\}$, we can define the integral of $\omega$ over $\mathcal{S}$ by
\begin{equation}
\label{Eq:IntOmega2}
\int_{\mathcal{S}} \omega := \sum_{n=1}^\infty \int_{\mathcal{S}} \rho_n \omega.
\end{equation}
In fact, writing $\Omega$ for the compact support of $\omega$, we can find a neighbourhood $W_x$ of $x \in \Omega$, $x_1,\ldots,x_N \in \Omega$ and a finite subset $N_j$ of $\mathbb{N}$ such that $\{\rho_n:n \notin N_j\}$ are identically zero on $W_{x_j}$, and such that 
\[
\int_{\mathcal{S}} \omega = \sum_{j=1}^N \sum_{n \in N_j} \int_{\mathcal{S}} \rho_n \omega.
\]
Thus the integral can be written as a finite sum.  Similarly, if $\omega$ is a non-negative $m$-form on $\mathcal{S}$, we can again define the integral of $\omega$ over $\mathcal{S}$ via~\eqref{Eq:IntOmega2}.  Finally, if $\omega$ is an integrable $m$-form on $\mathcal{S}$, the integral can be defined by taking positive and negative parts in the usual way.

In our work, we are especially interested in integrals of a particular type of form.  Given an $m$-dimensional, orientable manifold $\mathcal{S}$ in $\mathbb{R}^d$, the \emph{volume form} $d\mathrm{Vol}^m$ is the unique $m$-form on $\mathcal{S}$ such that at each $x \in \mathcal{S}$, the alternating $m$-tensor $d\mathrm{Vol}^m(x)$ on $T_x(\mathcal{S})$ gives value $1/m!$ to each positively oriented orthonormal basis for $T_x(\mathcal{S})$.  For example, when $\mathcal{S} = \mathbb{R}^m$, we have $d\mathrm{Vol}^m = dx_1 \wedge \ldots \wedge dx_m$, provided we consider the standard basis to be positively oriented.  As another example, if $\mathcal{R} \subseteq \mathbb{R}^d$ is a $d$-dimensional manifold and $\eta:\mathcal{R} \rightarrow \mathbb{R}$ is continuously differentiable with $\mathcal{S} = \{x \in \mathcal{R}:\eta(x)=1/2\}$ non-empty and $\dot{\eta}(x) \neq 0$ for $x \in \mathcal{S}$, then $\mathcal{S}$ is a $(d-1)$-dimensional, orientable manifold \citep[][Exercise~18, p.~106]{GuilleminPollack1974}.  If we say that an ordered, orthonormal basis $e_1,\ldots,e_{d-1}$ for $T_{x_0}(\mathcal{S})$ is positively oriented whenever $\det(e_1,\ldots,e_{d-1},\dot{\eta}(x_0)) > 0$, we have that 
\[
d\mathrm{Vol}^{d-1}(x_0) = \sum_{j=1}^d (-1)^{j+d}\frac{\eta_j(x_0)}{\|\dot{\eta}(x_0)\|}dx_1 \wedge \ldots \wedge dx_{j-1} \wedge dx_{j+1} \wedge \ldots \wedge dx_d(x_0),
\]
where $x_j$ denotes the $j$th coordinate function.  We now define an ordered, orthonormal basis $(e_1,0),\ldots,(e_{d-1},0), (0,1)$ for $T_{x_0}(\mathcal{S}) \times \mathbb{R}$ to be positively oriented.  Further, we define a $(d-1)$-form $\omega_1$ and a $1$-form $\omega_2$ on $\mathcal{S} \times (-\epsilon,\epsilon)$ by
\begin{align*}
\omega_1(x_0,t)\bigl((v_1,w_1),\ldots,(v_{d-1},w_{d-1})\bigr) &:= d\mathrm{Vol}^{d-1}(x_0)(v_1,\ldots,v_{d-1}) \\
\omega_2(x_0,t)(v_d,w_d) &:= dt(t)(w_d) = w_d. 
\end{align*}
Then, with $g$ defined as in Proposition~\ref{Prop:DG2}, and under the conditions of that proposition,
\begin{align*}
g^*(&dx_1 \wedge \ldots \wedge dx_d)(x_0,t)\bigl((e_1,0),\ldots,(e_{d-1},0), (0,1)\bigr) \\
&= dx_1 \wedge \ldots \wedge dx_d(x_0^t)\bigl(Dg_{(x_0,t)}(e_1,0),\ldots, Dg_{(x_0,t)}(e_{d-1},0),Dg_{(x_0,t)}(0,1)\bigr) \\
&= \frac{1}{d!} \det (I+tB) \\
&= \frac{1}{d} \det (I+tB) d\mathrm{Vol}^{d-1}(x_0)(e_1,\ldots,e_{d-1})dt(t)(1) \\
&= \det (I+tB) \ (\omega_1 \wedge \omega_2)(x_0,t)\bigl((e_1,0),\ldots,(e_{d-1},0), (0,1)\bigr),
\end{align*}
so $g^*(dx_1 \wedge \ldots \wedge dx_d)(x_0,t) = \det (I+tB) \ (\omega_1 \wedge \omega_2)(x_0,t)$.  It follows that if $h:\mathcal{S} \times (-\epsilon,\epsilon) \rightarrow \mathbb{R}$ is either compactly supported and integrable, or non-negative and measurable, then
\begin{equation}
\label{Eq:wedge1}
\int_{\mathcal{S} \times (-\epsilon,\epsilon)} h \, \omega_1 \wedge \omega_2 = \int_{\mathcal{S}} \int_{-\epsilon}^\epsilon h(x_0,t) \, dt \, d\mathrm{Vol}^{d-1}(x_0).
\end{equation}
We also require the change of variables formula: if $\mathcal{X}$ and $\mathcal{Y}$ are orientable manifolds and are of dimension $m$, and if $\psi:\mathcal{X} \rightarrow \mathcal{Y}$ is an orientation-preserving diffeomorphism, then
\begin{equation}
\label{Eq:wedge2}
\int_{\mathcal{X}} \psi^* \omega = \int_{\mathcal{Y}} \omega
\end{equation}
for every compactly supported, integrable $m$-form on $\mathcal{Y}$ \citep[][p.~168]{GuilleminPollack1974}.  In particular, if $f:\mathcal{S}^\epsilon \rightarrow \mathbb{R}$ is either compactly supported and integrable, or non-negative and measurable, then writing $x_0^t := x_0 + \frac{t\dot{\eta}(x_0)}{\|\dot{\eta}(x_0)\|}$, we have from~\eqref{Eq:wedge1} and~\eqref{Eq:wedge2} that
\begin{align}
\label{Eq:Conc}
\int_{\mathcal{S}^\epsilon} f(x) \, dx & = \int_{\mathcal{S} \times (-\epsilon,\epsilon)} \! \! \! \! \! \! \det(I + tB) f(x_0^t) \, (\omega_1 \wedge \omega_2)(x_0,t) \nonumber \\
& = \int_{\mathcal{S}} \int_{-\epsilon}^\epsilon \det(I + tB) f(x_0^t) \, dt \, d\mathrm{Vol}^{d-1}(x_0).
\end{align}

\end{appendix}

\end{document}